\documentclass[11pt]{article}

\usepackage[margin=1.1in]{geometry} 
\usepackage[colorlinks=true, linkcolor=blue, citecolor=blue, urlcolor=blue]{hyperref}
\usepackage{booktabs}
\usepackage{multirow}
\usepackage{graphicx}
\usepackage{caption}
\usepackage{amsfonts}
\usepackage{mathdots}
\usepackage{amsmath}
\usepackage{amsthm}  
\usepackage{amssymb}
\usepackage{mathtools}
\usepackage{enumerate} 
\usepackage{stmaryrd} 
\usepackage{xcolor} 
\usepackage{algorithmic}
\usepackage{algorithm}
\usepackage{subcaption}

\DeclarePairedDelimiter\abs{\lvert}{\rvert}%
\DeclarePairedDelimiter\norm{\lVert}{\rVert}%
\DeclareMathOperator*{\argmax}{arg\,max}
\DeclareMathOperator*{\argmin}{arg\,min}

\newcommand{\XXi}[1]{X^{(#1)}}

\makeatletter
\let\oldabs\abs
\def\abs{\@ifstar{\oldabs}{\oldabs*}}
\let\oldnorm\norm
\def\norm{\@ifstar{\oldnorm}{\oldnorm*}}
\makeatother

\newtheorem{theorem}{Theorem}[section]

\newtheorem{remark}{Remark}[section]

\title{Data-Driven Self-Supervised Learning for the Discovery of Solution Singularity for Partial Differential Equations}
\author{Difeng Cai\thanks{Department of Mathematics, Southern Methodist University, Dallas, TX 75205}
\and Paulina Sep\'ulveda\thanks{Pontificia Universidad Cat\'olica de Valpara\'iso, Av. Brasil 2950, Valpara\'iso, Chile}
}
\date{}

\begin{document}
\maketitle
\begin{abstract}
The appearance of singularities in the function of interest constitutes a fundamental challenge in scientific computing. It can significantly undermine the effectiveness of numerical schemes for function approximation, numerical integration, and the solution of partial differential equations
(PDEs), etc. The problem becomes more sophisticated if the location of the singularity is unknown, which is often encountered in solving PDEs.
Detecting the singularity is therefore critical for developing efficient adaptive methods to reduce computational costs in various applications.
In this paper, we consider singularity detection in a purely data-driven setting. Namely, the input only contains given data, such as the vertex set from a mesh.
To overcome the limitation of the raw unlabeled data, we propose a self-supervised learning (SSL) framework for estimating the location of the singularity.
A key component is a filtering procedure as the pretext task in SSL, where two filtering methods are presented, based on $k$ nearest neighbors and kernel density estimation, respectively.
We provide numerical examples to illustrate the potential pathological or inaccurate results due to the use of raw data without filtering. 
Various experiments are presented to demonstrate the ability of the proposed approach to deal with input perturbation, label corruption, and different kinds of singularities such interior circle, boundary layer, concentric semicircles, etc. 
\end{abstract}

\section{Introduction}
\label{sec:intro}

Partial differential equations (PDEs) play an indispensable role in a wide range of applications in computational science and engineering, such as mechanics, fluid dynamics, acoustics, electromagnetics, etc.
For challenging problems, the solution of the PDE can contain singularities in the form of sharp gradients, boundary layers, interior layers, etc.
The local, nonsmooth behavior of the solution can present great difficulties for numerical schemes to solve the PDE, causing numerical instability, slow convergence, loss of accuracy, etc.
The scenario can become even more sophisticated when the location of the singularity is \emph{unknown}, and it may require massive computational resources to resolve the singularity without knowing its location.
Detecting the singularity is thus critical for the reliability and efficiency of numerical simulations. 
Successful singularity detection can facilitate a number of computational tasks, including \emph{targeted} mesh generation using powerful software libraries such as Netgen \cite{schoberl1997netgen,ngsolve}, TetGen \cite{TetGen} CPAFT \cite{chengdi2025}; node generation algorithms in radial basis functions \cite{garmanjani2024adaptive,cavoretto2021adaptive,cavoretto2024RBF}, finite difference \cite{li2024stable}, RBF-FD \cite{bayona2017role,davydov2023stencil} to improve the solution accuracy near the singularity;
the selection or construction of suitable function classes for better approximation \cite{oanh2017adaptive,slak2019rbf};
adapting the loss function or collocation points in neural network-based approaches \cite{brink2021neural,mishra2024artificial}.

To efficiently discretize PDEs with unknown solution singularity, a widely used approach is to perform adaptive mesh refinement (AMR) \cite{purdue1972,AMR2Dtests,AMR2016comsol,AMReX2019}, which generates a mesh by successively refining towards the regions with large error estimates based on \emph{a posteriori} error estimation techniques \cite{ainsworthoden2000book,verf2013book,thesisDifengCai}. Consider elliptic PDEs. During AMR, the computation will generally become increasingly more demanding as the mesh size (or the number of degrees of freedom) keeps increasing during the refinement, since a larger system needs to be solved each time the mesh is refined. Besides, without the information of the singularity, the implementation of adaptive meshing guided by \emph{a posteriori} error estimators can be rather sophisticated.
In recent years, machine learning has emerged as a popular tool for scientific computing, and an active field of research is to use neural networks for solving PDEs. 
Some representative examples include deep Ritz \cite{deepRitz}, physics-informed neural network (PINN) \cite{raissi2017physics}, Fourier neural operator (FNO) \cite{FNO}, Galerkin neural network (GNN) \cite{ainsworth2021GNN}, etc.
Compared to the research along those lines, the use of machine learning for singularity detection, particularly in the absence of numerical solutions, remains an unexplored area.
The problem we are interested in is to predict the location of singularity in a data-driven fashion, namely, using given mesh data (i.e., a set of nodes) only.

\paragraph{Self-supervised learning (SSL)}
Predicting the location of the singularity using mesh data appears to be a supervised learning task in the context of machine learning. However, the mesh data (a set of nodes) is \emph{unlabeled}, and therefore supervised learning can not be applied in this setting.
A viable framework for such tasks is \emph{self-supervised learning}, which is able to learn generic representations from \emph{unlabeled} data \cite{SSLcookbook}.
For example, it can be used to predict hidden words \cite{baevski2020wav2vec,baevski2022data2vec} or masked patches of an image \cite{grill2020bootstrap,he2022masked}.
SSL is used in a variety of fundamental deep learning architectures, such as BERT \cite{BERT} and variational autoencoders \cite{VAE2014}.
Recent studies show that SSL models are more robust to label corruption or input perturbation \cite{SSL2019robust}, and more fair compared to supervised learning models \cite{SSL2022fair}.
SSL usually consists of two stages: pretext tasks and downstream tasks.
Pretext tasks aim to generate ``pseudo-labels'' from unlabeled data and to produce useful representations for downstream tasks, while a downstream task performs training with labeled data like a supervised learning task (cf. \cite{SSLcookbook}).
The design of pretext tasks is crucial and depends on the specific application.
A commonly used pretext task is \emph{masking}, which enables the downstream task to predict the masked information for training the SSL model.
This has been used in a variety of applications, such as masked language modeling \cite{BERT,transfer2020}, 
masked autoencoders \cite{masked2022vision},
masked patch prediction in vision transformer (ViT) \cite{ViT2021}.
In the context of singularity detection, the design of pretext tasks and downstream tasks is highly non-trivial, and existing SSL techniques in language modeling or computer vision are not applicable.
For example, masking can not be applied in this case since the singularity is unknown and no ground truth is available. 
Moreover, different from \emph{finite} dimensional information such as images (with finite number of pixels) and texts (with finite choices of words), it is unclear how to represent and characterize the unknown singularity, which can be a singular layer that contains \emph{infinitely} many points\cite{stynes2005confusion,verf1998reaction,ainsworth1999reaction,confusion}.
These difficulties call for novel approaches that can adapt the SSL framework to the context of singularity detection.

\paragraph{Methodology and contributions}
In this paper, we propose a novel data-driven self-supervised learning framework to perform singularity detection with unlabeled mesh data.
The framework is based on a general probabilistic formulation that allows flexible design of pretext tasks (labeling, representation learning, etc.) and straightforward downstream training.
To produce intelligible representations for the downstream task, we use filtering as a pretext task and propose two methods to filter the raw input data. The importance of filtering will be demonstrated via numerical experiments.
The proposed data-driven self-supervised learning approach offers an effective way to predict the singularity with the following benefits.
\begin{itemize}
\itemsep0in
    \item It works for \emph{unlabeled} data and the ``small data'' regime.
    \item It is insensitive to label (pseudo-label) corruption and input perturbation.
    \item It allows versatile choices for representing the singularity.
    \item It enables the use of machine learning techniques to facilitate pretext and downstream tasks.
\end{itemize}

The rest of the manuscript is organized as follows. Section~\ref{sec:setup} introduces the mathematical setup of the singularity detection problem considered in this paper. Section~\ref{sec:formulation} presents a probabilistic formulation of the data-driven singularity detection problem. Based on the probabilistic formulation, Section~\ref{sec:SSL} introduces the data-driven self-supervised singularity detection framework, including two filtering methods. Section~\ref{sec:experiments} illustrates the empirical performance of the proposed approach via extensive numerical experiments.
Throughout the presentation, $\norm{\cdot}$ denotes the Euclidean norm.

\section{Problem setup}
\label{sec:setup}
In this section, we present the mathematical setup of the singularity detection problem.

\paragraph{Singularity}
Let $\Omega\subseteq\mathbb{R}^d$ denote the domain of the PDE problem.
Assume the set of singularity is given by $\mathcal{S}=\{x\in\overline{\Omega}: F_*(x)=0\}$ for some unknown continuous function $F_*$, for example, a polynomial.
Note that the set $\mathcal{S}$ may contain one single component or multiple disconnected components.
To illustrate this, consider the following three choices of polynomial $F_*$ over $(x,y)\in\mathbb{R}^2$.
\begin{equation}
\label{eq:examples}
\begin{aligned}
    F_*(x,y)&=x^2+y^2-\frac{1}{4},\\ F_*(x,y)&=(x+1)(y+1),\\ F_*(x,y)&=(x^2+y^2-0.5^2)(x^2+y^2-0.75^2).
\end{aligned}    
\end{equation}
The plots of $\mathcal{S}$ corresponding to the three examples are shown in Figure~\ref{fig:poly_sing} (from left to right), where the choices of $\Omega$ are $(-1,1)^2$, $(-1,1)^2$, $(0,1)\times (-1,1)$, respectively.

\begin{figure}
\includegraphics[width=\textwidth]{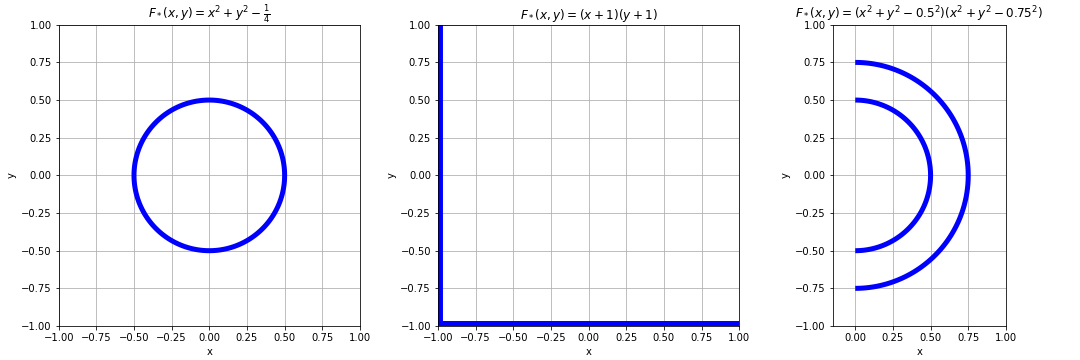}
\caption{The singularity set $\mathcal{S}$ corresponding to different polynomials $F_*$ described in~\eqref{eq:examples}.}
\label{fig:poly_sing}
\end{figure}

\paragraph{Data}
The data is assumed to be a set of nodes, for example, generated by adaptive mesh refinement (AMR)~\cite{AMR2Dtests,confusion}. 
Two types of data structures are considered:
\begin{itemize}
    \item ``Type I'' (data aggregate): $X$ consists of all nodes in the final mesh after $R$ steps of AMR.
    \item ``Type II'' (data batch): $\XXi{0},\dots,\XXi{R}$ from $R$ refinement steps, where $\XXi{0}$ consists of nodes in the initial mesh and $\XXi{i}$ ($i\geq 1$) denotes the newly generated nodes in the $i$-th step of AMR using, for example, the newest vertex bisection refinement strategy \cite{purdue1972,mitchell2016}.
\end{itemize}

The union $\bigcup\limits_{i=0}^R \XXi{i}$ for all ``Type II'' data batches is equal to the final vertex set $X$ in Type I. 
An illustration of data batches is shown in Figure~\ref{fig:batches}.

 \begin{figure}[ht]
  \begin{subfigure}[b]{0.32\textwidth}
  \includegraphics[width=\textwidth]{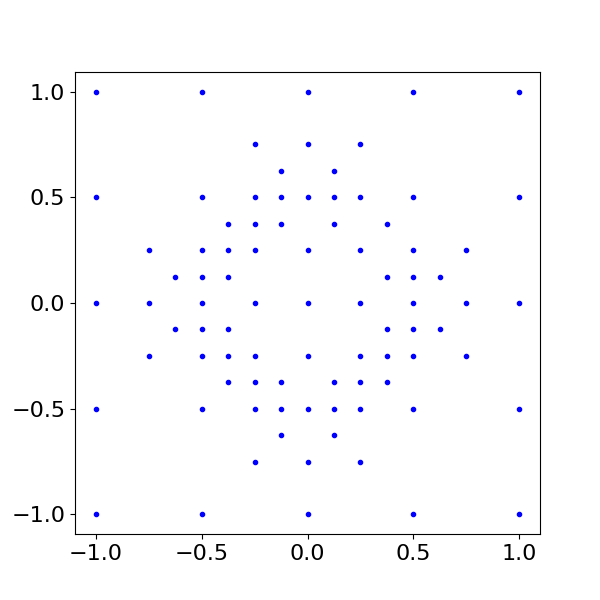}  
 \caption{$X^{(0)}$ to $X^{(10)}$}
 \end{subfigure}
 \begin{subfigure}[b]{0.32\textwidth}
 \includegraphics[width=\textwidth]{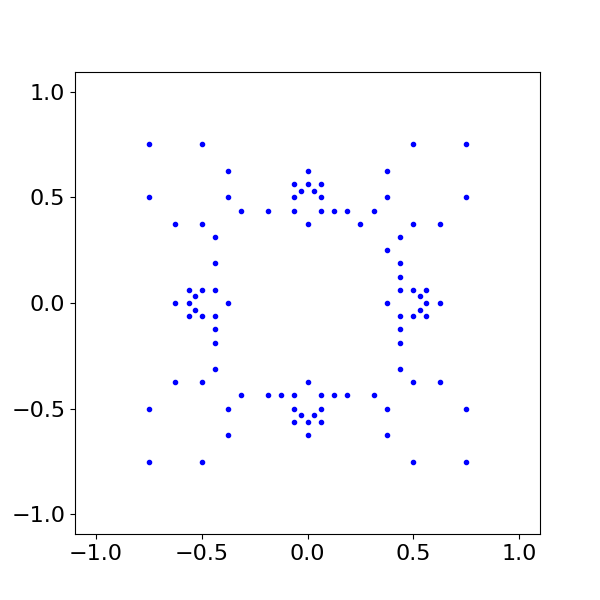}   
 \caption{$X^{(10)}$ to $X^{(14)}$}
 \end{subfigure}
 \begin{subfigure}[b]{0.32\textwidth}
 \includegraphics[width=\textwidth]{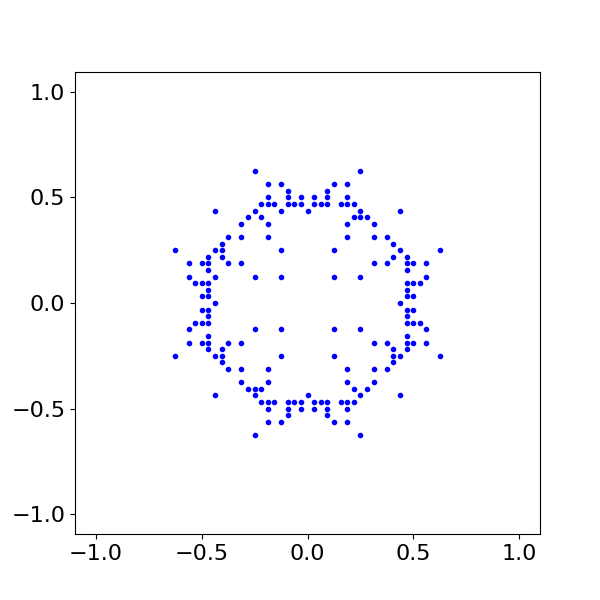}   
 \caption{$X^{(14)}$ to $X^{(18)}$}
 \end{subfigure}
 \\
  \begin{subfigure}[b]{0.32\textwidth}
 \includegraphics[width=\textwidth]{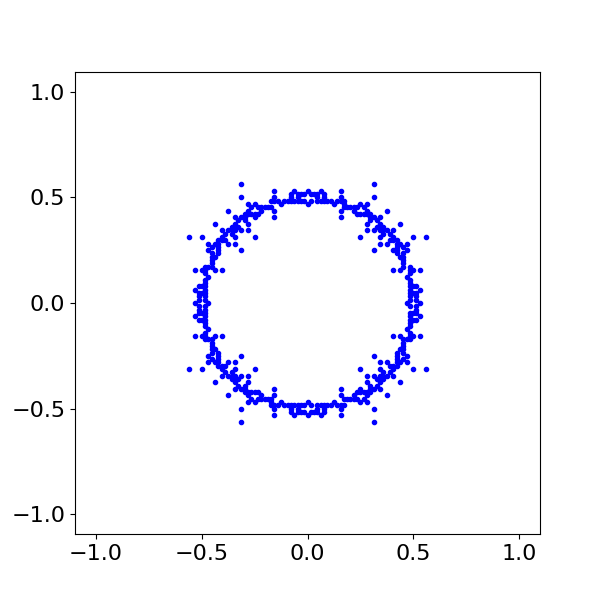}   
 \caption{$X^{(18)}$ to $X^{(24)}$}
 \end{subfigure}
 ~
\begin{subfigure}[b]{0.32\textwidth}
 \includegraphics[width=\textwidth]{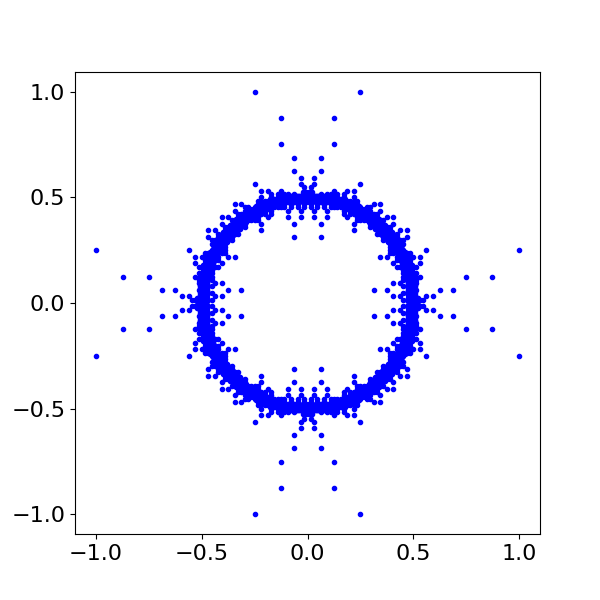}   
 \caption{$X^{(24)}$ to $X^{(36)}$}
 \end{subfigure}
~
 \begin{subfigure}[b]{0.32\textwidth}
 \includegraphics[width=\textwidth]{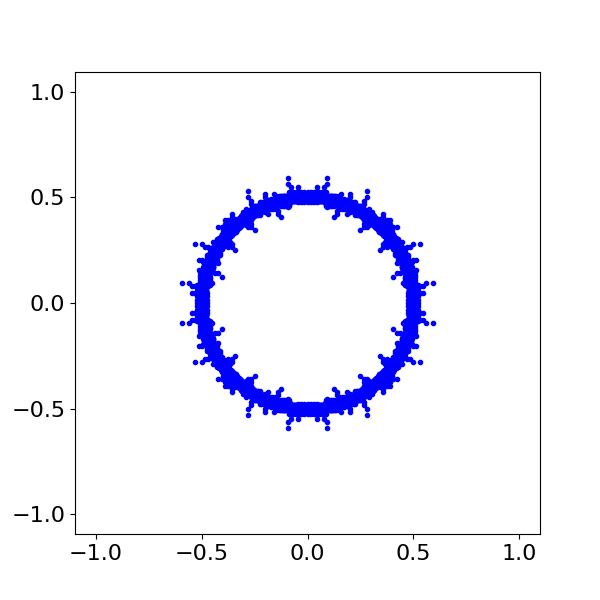}   
 \caption{$X^{(36)}$ to $X^{(44)}$}
 \end{subfigure}
 \caption{Data batches generated from AMR in \cite{confusion} for a singularly-perturbed reaction-diffusion problem \eqref{eq:reaction-difussion}. $X^{(0)}$ corresponds to the nodes of the initial mesh. The subset $X^{(i)}$ to $X^{(j)}$ corresponds to the new nodes generated from the $i$-th refinement until the $j$-th refinement.}
 \label{fig:batches}
\end{figure}

\paragraph{Goal}
The goal of singularity detection is to find a detection function $f(x)$ such that the set of roots of $f(x)$ in the closure of the domain $\Omega$ is close to the true singularity set $\mathcal{S}$.
The only data given is a set of nodes from a mesh generated by AMR,
where the mesh is \emph{not} assumed to be sufficiently fine, since the motivation is to efficiently perform singularity detection in the ``small data'' regime, without requiring a fine mesh (which can be computationally demanding to generate).
We represent the singularity detection $f$ in the following form
\begin{equation}
\label{eq:f}
    f(x)=\sum_{i=1}^k c_i \phi_i(x),
\end{equation}
where the functions $\phi_i$ are prescribed basis functions.
For example, we can choose $\phi_i$ to be polynomials.
Besides, polar coordinates and Fourier basis functions can also be used, which will be explored in numerical experiments.
To find a suitable coefficient vector $\mathbf{c}=[c_1,\dots,c_k]$, a probabilistic formulation is presented in Section \ref{sec:formulation}.

\begin{remark}
In terms of detecting the singularity set $\mathcal{S}$, it is worth noting that trying to approximate the function $F_*$ is \emph{not} well-posed as a solution approach, since there are infinitely many functions with the same set of roots as $F_*$. 
For example, any constant multiple of $F_*$, i.e. $\alpha F_*$ ($\alpha\neq 0$), yields exactly the same set of roots as $F_*$.
Also, $F_*(x,y)^k$ and $F_*(x,y)$ share the same set of roots for a positive integer $k$.
A well-posed formulation for approximating $F_*$ is presented in Section \ref{sec:formulation}.
\end{remark}

\section{Baseline singularity detection via a probabilistic formulation}
\label{sec:formulation}
In this section, we formulate the data-driven singularity detection problem using a probabilistic perspective. 
The formulation serves as the building block of the downstream task in the SSL framework for singularity detection presented later in Section \ref{sec:SSL}.
Given mesh data (Type I or Type II), we formulate the regression problem of determining the coefficient vector $\mathbf{c}$ in \eqref{eq:f} via maximum likelihood estimation (MLE) as follows.
Given $x\in\overline{\Omega}$ and $\mathbf{c}$, we model the unknown $F_*(x)$ as a random variable given by $y=f(x)+\epsilon$, where $\epsilon\sim\mathcal{N}(0,\sigma^2)$ denotes the Gaussian random noise with variance $\sigma^2$ and $f$ is the approximation ansatz in \eqref{eq:f}.
Therefore, the random variable $y$ is Gaussian, given $x$ and $\mathbf{c}$. 
\begin{equation}
\label{eq:Gaussian}
   y|x,\mathbf{c} \sim \mathcal{N}(f(x), \sigma^2), \quad p(y|x;\mathbf{c}) = \frac{1}{\sqrt{2\pi\sigma^2}} e^{-\frac{(f(x)-y)^2}{2\sigma^2}}.
\end{equation}
The setup in  \eqref{eq:Gaussian} enables MLE for the coefficient parameter $\mathbf{c}$ in $f$ once \emph{labeled} training data is given. Since the input mesh data is unlabeled, we discuss pseudo-labeling below.

\paragraph{Choice of labels}
For $x\in\mathcal{S}$, the true value of $F_*(x)$ is known to be $0$, and thus an ideal label for $x$ is $0$.
Since $x\notin\mathcal{S}$ in general for a mesh node $x$, the true label of $F_*(x)$ is unknown.
To this end, we assign a pseudo-label of $0$ to each point $x$ in the training data.
The raw training data with pseudo-labels is then
\begin{equation}
\label{eq:labeled}
\text{Type I}:\;\, (x,0) \;\, \text{with}\;\, x\in X; \quad  \text{Type II}:\;\, (x,0)\;\,\text{with}\;\,  x\in\bigcup\limits_{i=0}^R\XXi{i}.
\end{equation}
The choice of label is reasonable if $x$ is close to $\mathcal{S}$ (then $F_*(x)\approx 0$), but becomes improper if $x$ is away from $\mathcal{S}$.
Potential issues will be illustrated in Section \ref{sec:experiments} when using the raw data in \eqref{eq:labeled} directly for supervised learning due to input perturbation and label corruption discussed in Section \ref{sub:issues}.
This issue will be addressed by the self-supervised learning framework presented in Section \ref{sec:SSL}.

\paragraph{Constraint for coefficients $\mathbf{c}$}
Another issue to be fixed is that regression with labeled training data in \eqref{eq:labeled} can yield a trivial solution, namely, $f(x)=0$ (by simply setting $\mathbf{c}=\mathbf{0}$ so that $f$ achieves the true label $0$ for all nodes in the training data). 
This trivial choice gives zero training error (as true labels are all $0$ in \eqref{eq:labeled}), but is meaningless since \emph{every} point is a root of $f$ and will thus be predicted as a singularity.
To resolve this issue, we impose a constraint on the coefficient vector $\mathbf{c}$, for example, $\norm{\mathbf{c}}=1$.
We remark that the labeling in \eqref{eq:labeled} is the simplest choice, and other potential options to produce labeled data will be investigated at a later date.

\paragraph{Maximum likelihood estimation (MLE)}
Based on the probabilistic model in \eqref{eq:Gaussian}, we employ an MLE formulation to determine the coefficient vector $\mathbf{c}$ using the given data. For Type I and Type II data formats, the MLE formulations are slightly different, as described below. 

For Type I data $X$, we prescribe a single variance $\sigma^2$ for the Gaussian noise in \eqref{eq:Gaussian} for all $x\in X$.
The MLE for $\mathbf{c}$ in this case is then
\begin{equation}
\label{eq:MLE1}
    \mathbf{c}_* = \argmax_{\norm{\mathbf{c}}=1} p(\mathbf{0}|X;\mathbf{c}) 
    = \argmax_{\norm{\mathbf{c}}=1} \prod_{x\in X}p(0|x;\mathbf{c}).
\end{equation}

For Type II data batches $\XXi{0}$, $\dots$, $\XXi{R}$,
we use $\sigma_i^2$ as the variance in the Gaussian noise for data from $\XXi{i}$ ($0\leq i\leq R$), which yields
$$p_i(y|x;\mathbf{c}) = \frac{1}{\sqrt{2\pi\sigma_i^2}} e^{-\frac{(f(x)-y)^2}{2\sigma_i^2}},\quad x\in\XXi{i}.$$
The MLE for $\mathbf{c}$ in this case is
\begin{equation}
\label{eq:MLE2}
    \mathbf{c}_* = \argmax_{\norm{\mathbf{c}}=1} 
    \prod_{i=0}^R p_i(\mathbf{0}|\XXi{i};\mathbf{c})
    = \argmax_{\norm{\mathbf{c}}=1} 
    \prod_{i=0}^R\prod_{x\in \XXi{i}}p_i(0|x;\mathbf{c}).
\end{equation}
As stated in the theorem below, the MLE problems in \eqref{eq:MLE1} and \eqref{eq:MLE2} can be shown to be equivalent to certain quadratic programming problems.

\begin{theorem}
\label{thm:MLE}
    The MLE formulations in \eqref{eq:MLE1} and \eqref{eq:MLE2} can be written as 
    \begin{equation}
    \label{eq:minL1}
        \mathbf{c}_* = \argmin_{\norm{\mathbf{c}}=1} L_1(\mathbf{c}) = \argmin_{\norm{\mathbf{c}}=1} \sum_{x\in X} \frac{f(x)^2}{2\sigma^2}
    \end{equation}
    and 
    \begin{equation}
    \label{eq:minL2}
     \mathbf{c}_* = \argmin_{\norm{\mathbf{c}}=1} L_2(\mathbf{c}) = \argmin_{\norm{\mathbf{c}}=1} \sum_{i=0}^R  \sum_{x\in \XXi{i}} \frac{f(x)^2}{2\sigma_i^2},
    \end{equation}
    respectively.
    Here, the loss functions $L_1,L_2$ for Type I data and Type II data are defined as
\begin{equation}
\label{eq:L1L2}
L_1(\mathbf{c})=\sum\limits_{x\in X} \frac{f(x)^2}{2\sigma^2},\quad L_2(\mathbf{c})= \sum_{i=0}^R \sum\limits_{x\in \XXi{i}} \frac{f(x)^2}{2\sigma_i^2}, \quad \text{with}\;\; f=\sum_{i=1}^k c_i \phi_i(x) \;\, \text{in} \;\,  \eqref{eq:f}.
\end{equation}
\end{theorem}
\begin{proof}
    We first prove that \eqref{eq:MLE1} (for Type I data $X$) is equivalent to \eqref{eq:minL2}.
\begin{equation*}
\begin{aligned}
    \mathbf{c}_* = \argmax_{\norm{\mathbf{c}}=1} p(\mathbf{0}|X;\mathbf{c}) 
    &= \argmax_{\norm{\mathbf{c}}=1} \prod_{x\in X}p(0|x;\mathbf{c})
    = \argmax_{\norm{\mathbf{c}}=1} \left( \ln \prod_{x\in X}p(0|x;\mathbf{c})\right)\\
    &= \argmax_{\norm{\mathbf{c}}=1} \sum_{x\in X} \ln p(0|x;\mathbf{c})
    = \argmax_{\norm{\mathbf{c}}=1}  \sum_{x\in X} -\frac{f(x)^2}{2\sigma^2}.
\end{aligned}
\end{equation*}
Similarly, we show that \eqref{eq:MLE2} (for Type II data batches $\XXi{i}$) and \eqref{eq:minL2} are equivalent.
\begin{equation*}
\begin{aligned}
    \mathbf{c}_*
    &= \argmax_{\norm{\mathbf{c}}=1} 
    \prod_{i=0}^R\prod_{x\in \XXi{i}}p_i(0|x;\mathbf{c}) = \argmax_{\norm{\mathbf{c}}=1} \sum_{i=0}^R\sum_{x\in \XXi{i}} \ln p_i(0|x;\mathbf{c})\\
    &= \argmax_{\norm{\mathbf{c}}=1} \sum_{i=0}^R  \sum_{x\in \XXi{i}} -\frac{f(x)^2}{2\sigma_i^2}=\argmin_{\norm{\mathbf{c}}=1} \sum_{i=0}^R  \sum_{x\in \XXi{i}} \frac{f(x)^2}{2\sigma_i^2}.
\end{aligned}
\end{equation*}
The proof is complete.
\end{proof}

For Type II data, the use of potentially different variances for different batches is encouraged. Specifically, it is reasonable to impose a smaller $\sigma_i^2$ for larger $i$ due to the fact that $\XXi{i}$ will be produced later in the refinement and will generally be closer to $\mathcal{S}$. As a result, the uncertainty of $F_*(x)$ for $x\in\XXi{i}$ should be smaller compared to coarse meshes like $\XXi{1}$, $\XXi{2}$.
The choice of $\sigma_i$ can also be seen from the loss function in \eqref{eq:L1L2}.
For larger $i$, smaller $\sigma_i^2$ gives a larger weight (i.e. $\frac{1}{2\sigma_i^2}$) in the loss $L_2$, which will put more emphasis on reducing $f(x)^2$ over $x\in\XXi{i}$ in the minimization \eqref{eq:minL2}. This makes sense if $\XXi{i}$ is close to $\mathcal{S}$ and $F_*(x)\approx 0$ for $x\in\XXi{i}$.
On the other hand, for small $i$, for example, $i=1$, $\XXi{1}$ is from a coarse mesh, and can contain points that are far from $\mathcal{S}$.
Hence $F_*(x)$ will not be close to $0$ for $x\in\XXi{i}$, and minimizing $f(x)^2$ should be discouraged. Note that \eqref{eq:minL2} is analogous to the weighted least squares regression.

\paragraph{Optimization}
In practice, the constraint in the optimization problems \eqref{eq:MLE1} and \eqref{eq:MLE2} is implemented as the quadratic constraint $\mathbf{c}\cdot\mathbf{c}=1$.
It can be seen that for the representation of $f$ in \eqref{eq:f}, the loss functions $L_1$ and $L_2$ in \eqref{eq:MLE1} and \eqref{eq:MLE2} are quadratic in $\mathbf{c}$. Both \eqref{eq:MLE1} and \eqref{eq:MLE2} can be written as a least squares problem with a quadratic constraint, and can be solved by standard optimization techniques \cite{convex2004book}. In numerical experiments in Section \ref{sec:experiments}, the constrained minimization is solved using the sequential least squares quadratic programming (SLSQP)~\cite{kraft1988software}. 

\paragraph{Algorithm}
The baseline singularity detection algorithm (without any pretext task to process the raw data) is presented in Algorithm \ref{alg:alg1}. 
It is data-driven as it only requires the mesh data as input. 
We remark that the representation $f$ to characterize the singularity is quite flexible and does not have to be limited to \eqref{eq:f}. 
One can customize the parameterization of $f$ or the choice of basis functions according to the underlying problem, the domain-specific knowledge, computational efficiency, and so on.
Algorithm \ref{alg:alg1} can be sensitive to input perturbation, pseudo-labeling, and may give undesired results (cf. Section \ref{sec:experiments}), and thus is \emph{not} recommended to be used alone. Instead, it serves as a building block in the SSL framework and will be used as a downstream task in SSL discussed later in Section \ref{sub:SSLfull}.

\begin{algorithm}[hbt!]
\caption{Primitive data-driven singularity detection (\emph{without} filtering)}
\label{alg:alg1}
\begin{algorithmic}
\STATE \textbf{Input:} Type I data aggregate $X$ (nodes from a mesh) or Type II data batches $\XXi{0},\dots,\XXi{R}$ from $R$ steps of AMR
\STATE \textbf{Hyperparameters:} $\sigma_i$ ($0\leq i\leq R$) if input is Type II, basis functions $\phi_1,\phi_2,\dots$ in \eqref{eq:f}
\STATE \textbf{Output:} A function $f$ from \eqref{eq:f} that minimizes the loss $L_1$ 
\IF{Input is Type I}
\STATE Define the loss function $L_1$ in \eqref{eq:L1L2} with $f$ in \eqref{eq:f}
\STATE Solve the constrained minimization in \eqref{eq:minL1} to obtain $\mathbf{c}_*$
\ENDIF
\IF{Input is Type II}
\STATE Define the loss function $L_2$ in \eqref{eq:L1L2} with $f$ in \eqref{eq:f}
\STATE Solve the constrained minimization in \eqref{eq:minL2} to obtain $\mathbf{c}_*$
\ENDIF
\RETURN $f$ in the form \eqref{eq:f} with coefficients $\mathbf{c}_*$
\end{algorithmic}
\end{algorithm}

\section{Data-driven self-supervised singularity detection}
\label{sec:SSL}
In this section, we present the complete data-driven self-supervised singularity detection framework, which consists of a pretext task and a downstream task.
The probabilistic formulation in Section \ref{sec:formulation} is used as the downstream task, and the pretext task will be introduced in Section \ref{sub:filtering}.
Section \ref{sub:issues} discusses potential issues associated with regression using raw data directly (Algorithm \ref{alg:alg1}), which will serve as the motivation for the development of pretext tasks.
The complete self-supervised singularity detection algorithm is presented in Section \ref{sub:SSLfull}.

\subsection{Input perturbation and label corruption}
\label{sub:issues}
Since the singularity set $\mathcal{S}$ is \emph{unknown} and points in the training data (mesh vertices) generally do not belong to $\mathcal{S}$, the problem we consider is substantially different the scenario of learning an equation for describing $\mathcal{S}$ using given points in $\mathcal{S}$.
The training data $X$ can be viewed as a ``perturbation'' of the true singularities in $\mathcal{S}$.
As will be illustrated in Section \ref{exp:circle-type1-type2}, \emph{input perturbation} can lead to entirely wrong results if Algorithm \ref{alg:alg1} is used directly without any pretext task to process the raw data. 
To address this issue, Section \ref{sub:filtering} introduces a filtering procedure to be used as the pretext task in the SSL framework.

Besides input perturbation, label corruption is another factor to take into account in developing a data-driven SSL framework.
Since the data is unlabeled, pseudo-labeling is used in Section \ref{sec:formulation} to assign $y=0$ to each point in the given dataset. As already mentioned in Section \ref{sec:formulation}, a point $x$ in the given dataset generally does not belong to $\mathcal{S}$ and thus $F_*(x)\neq 0$ in general, rendering a mismatch between true label and pseudo-label. This can be viewed as \emph{label corruption}. The corruption becomes more severe for input points that are further away from $\mathcal{S}$. Recall the formulation in \eqref{eq:minL2} for Type II data. We can see that the issue can be mitigated when using Type II data with properly chosen $\sigma_i$'s to account for the proximity of points in $\XXi{i}$ to $\mathcal{S}$. 
To see this, note that the pseudo-label $0$ for a node $x$ serves as an accurate approximation to the true label $F_*(x)$ if $x$ is close to $\mathcal{S}$ (since $F_*|_{\mathcal{S}}=0$ and $F_*$ is continuous). We know that $\XXi{i}$ with larger $i$ is generally closer to $\mathcal{S}$ compared to $\XXi{i}$ with smaller $i$. Therefore, it is reasonable to emphasize node batches with larger $i$ in the loss function \eqref{eq:L1L2} by using a larger weight $\frac{1}{2\sigma_i^2}$.
The weight $\frac{1}{2\sigma_i^2}$ for nodes in $\XXi{i}$ indicates the level of confidence about the points in $\XXi{i}$ to be close to the singularity set $\mathcal{S}$.
As can be seen from Figure \ref{fig:batches} in Section \ref{sec:setup}, a patch $\XXi{i}$ generated from finer refinement (corresponding to a larger $i$) is more likely to be close to $\mathcal{S}$, and thus $\frac{1}{2\sigma_i^2}$ should be larger as $i$ increases.
We therefore choose $\sigma_i$ to decrease with $i$, and several choices are tested in the numerical experiments Section~\ref{exp:circle-type1-type2} where the exponential decay is chosen.
Label corruption is correlated with input perturbation, and can be alleviated by reducing the ``noise'' in the data, namely, points far from $\mathcal{S}$. Such issues will be addressed by the filtering procedure presented in Sections \ref{sub:filtering}.

\subsection{Filtering as a pretext task in SSL}
\label{sub:filtering}
In this section, we develop pretext tasks to help reduce ``noise'' in the raw dataset, i.e., points that are far from the singularity set $\mathcal{S}$.
Since the set $\mathcal{S}$ is \emph{unknown}, it is impossible to measure the distance from a given point to any point in $\mathcal{S}$.
To handle this challenge, we use machine learning techniques below to indicate the proximity to $\mathcal{S}$ for each point in the given dataset (without knowing $\mathcal{S}$). Based on the approximate proximity to $\mathcal{S}$, a filtering procedure can be designed to generate a subset $\widetilde{X}$ of $X$ such that points in $\widetilde{X}$ are all ``close to'' $\mathcal{S}$. Here, the proximity to $\mathcal{S}$ is transformed into estimating the density of points in $X$, as dense regions are more likely to contain the singularity set $\mathcal{S}$ as a result of AMR with \emph{robust} a posteriori estimators.
We'd like to point out that AMR may generate poorly refined meshes (over-refined or under-refined in certain subdomains) if the underlying error estimator is not robust. This can lead to a high node density in a subdomain where the singularity is not present at all (cf. \cite{caizhang2009,verf2013book,thesisDifengCai}).
Remark \ref{rm:AMR} discusses the quality of mesh data and the importance of \emph{robust} a posteriori error estimation.

\begin{figure}
\begin{subfigure}[b]{0.22\textwidth}
   \includegraphics[width=\textwidth]{figsubmit/mesh0to10.png}  
 \caption{$X^{(0)}$ to $X^{(10)}$}
 \end{subfigure}
\begin{subfigure}[b]{0.22\textwidth}
  \includegraphics[width=\textwidth]{figsubmit/mesh10to14.png}  
 \caption{$X^{(10)}$ to $X^{(14)}$}
 \end{subfigure}
\begin{subfigure}[b]{0.22\textwidth}
  \includegraphics[width=\textwidth]{figsubmit/mesh24to36.png}  
 \caption{$X^{(24)}$ to $X^{(36)}$}
 \end{subfigure}
 \begin{subfigure}[b]{0.22\textwidth}
  \includegraphics[width=\textwidth]{figsubmit/mesh36to44.png}  
 \caption{$X^{(36)}$ to $X^{(44)}$}
 \end{subfigure}
\\
\begin{subfigure}[b]{0.22\textwidth}
  \includegraphics[width=\textwidth]{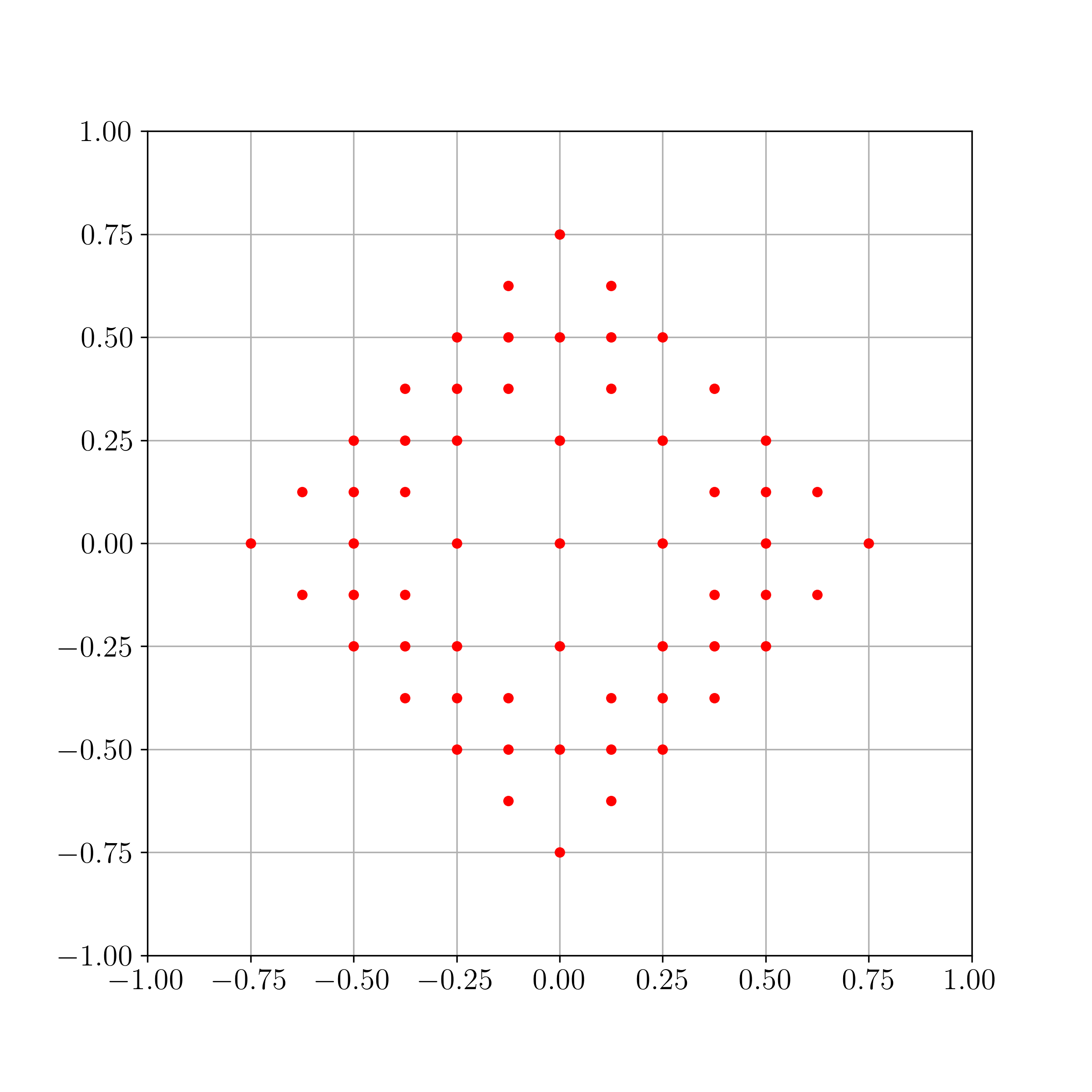}
 \caption{Filtering result for (a)}
 \end{subfigure}
\begin{subfigure}[b]{0.22\textwidth}
  \includegraphics[width=\textwidth]{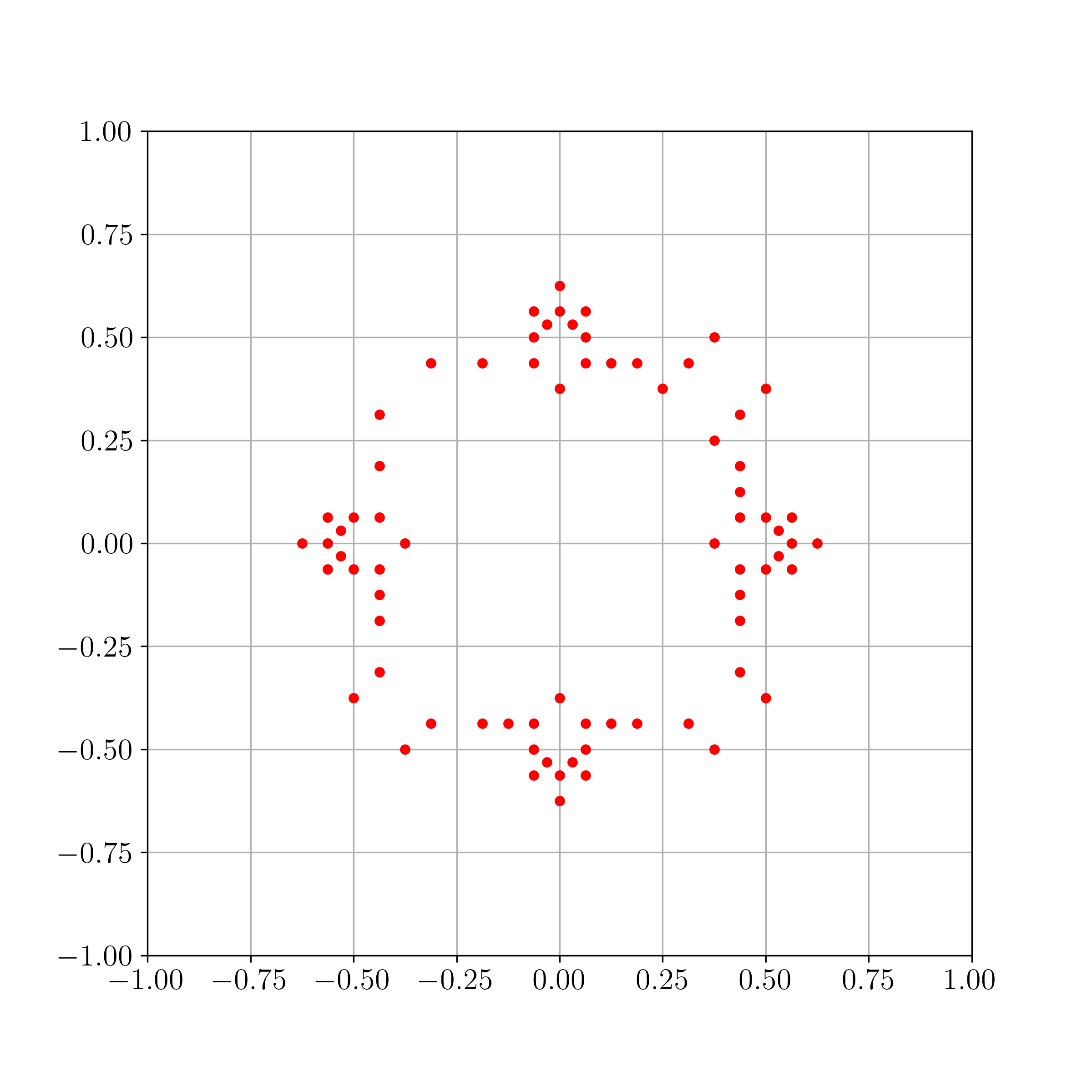}
 \caption{Filtering result for (b)}
 \end{subfigure}
\begin{subfigure}[b]{0.22\textwidth}
  \includegraphics[width=\textwidth]{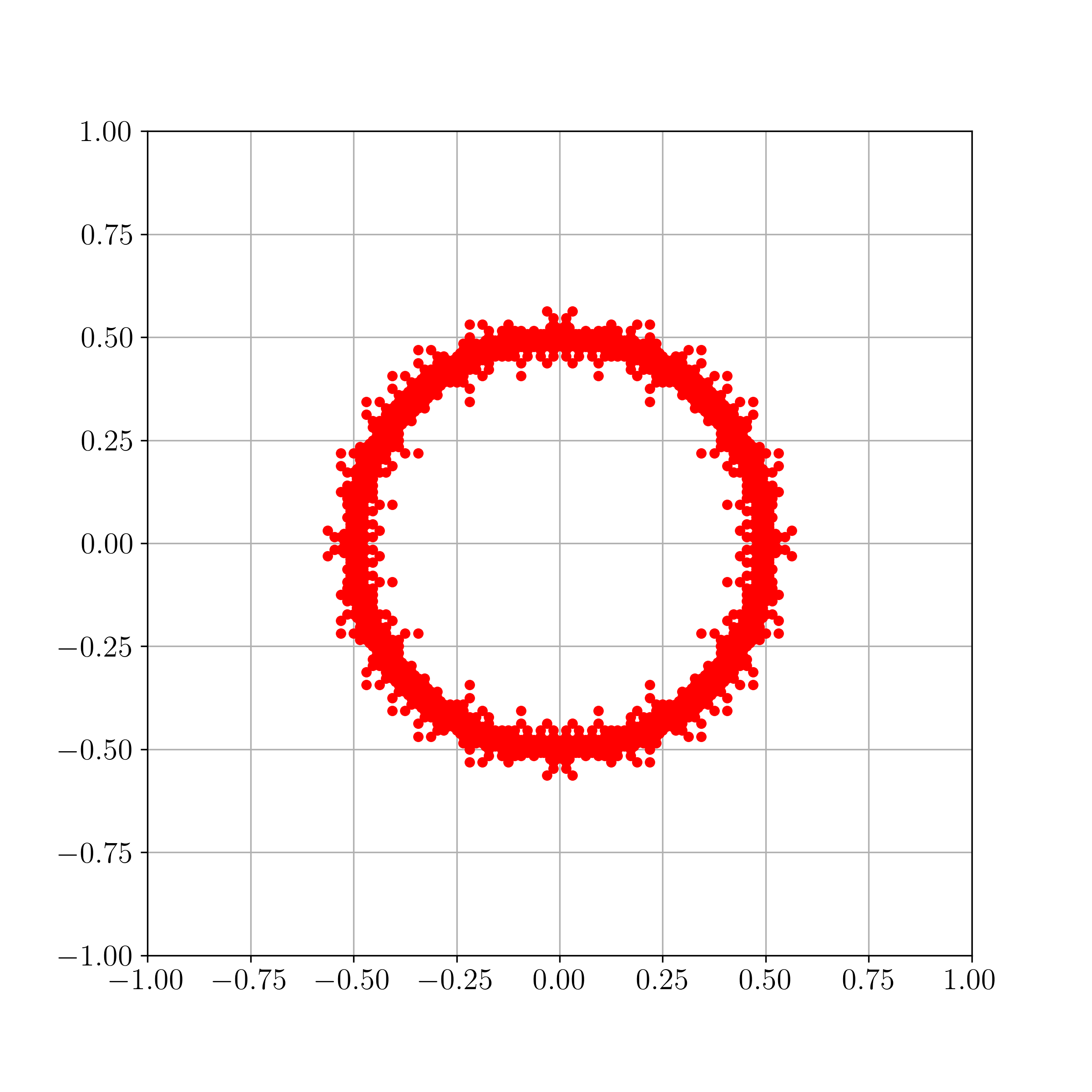}
 \caption{Filtering result for (c)}
 \end{subfigure}
 \begin{subfigure}[b]{0.22\textwidth}
  \includegraphics[width=\textwidth]{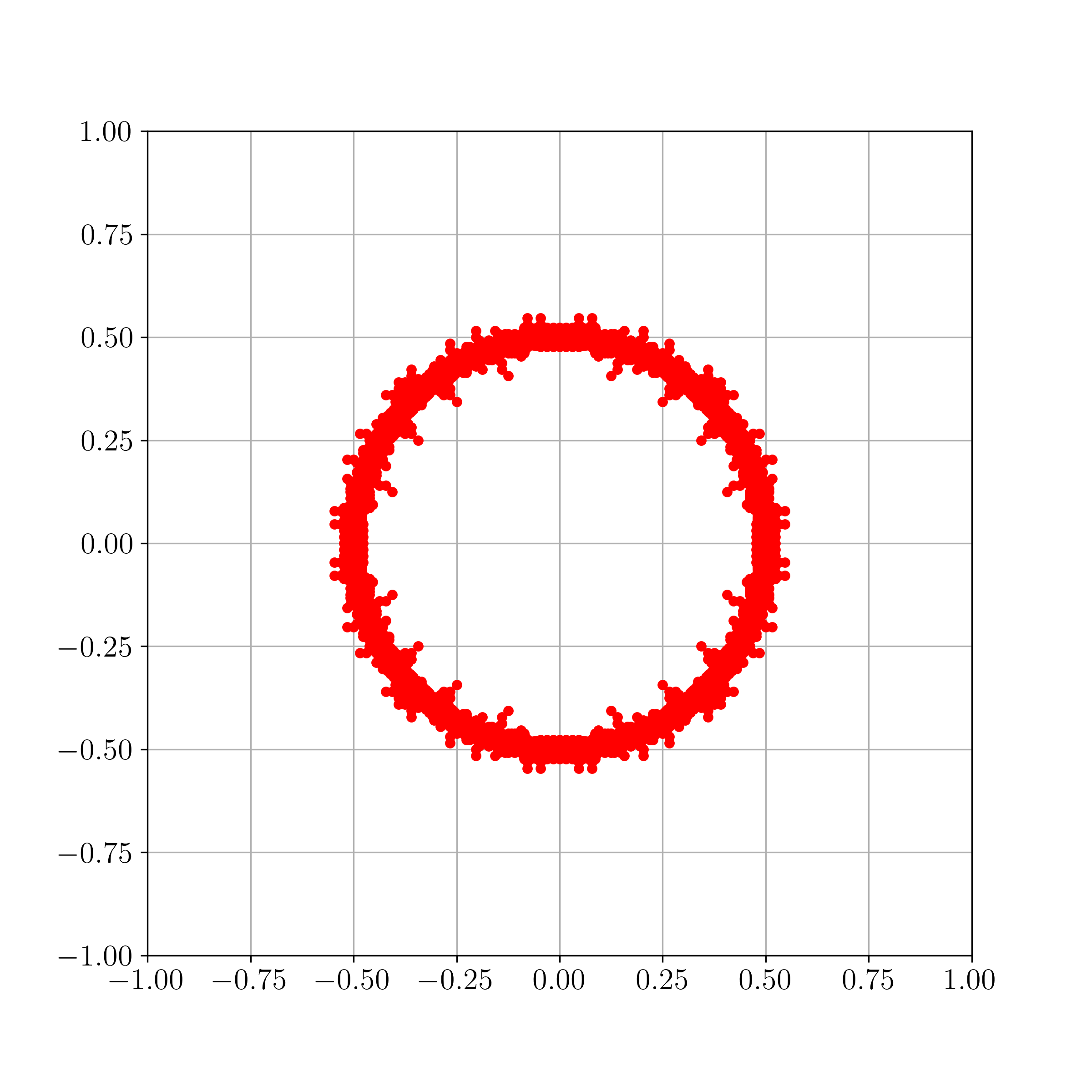}
 \caption{Filtering result for (d)}
 \end{subfigure}
 \caption{Raw data batches (top row) from AMR in \cite{confusion} and the corresponding filtered data (bottom row) using Algorithm \ref{alg:kde} with $\gamma=0.6$.}
\end{figure}

\vspace{.15in}

\noindent\textbf{Filtering using kernel density estimation (KDE).}
The first filtering method is based on \emph{kernel density estimation} \cite{kde1962,kde1956,silverman86density,deng2011density,bishop2006book}. 
Simply speaking, we first estimate the density value at each point $x$ and then remove the points with a relatively low density.
We describe the algorithm below.
Given a density kernel $k:\mathbb{R}^d\to[0,\infty)$ and a set of points $x_1,\dots,x_N$,
the kernel density estimator at $x$ is given by
\begin{equation}
\label{eq:kde}
    \rho(x) = \frac{1}{N}\sum_{i=1}^N \frac{1}{h^d}k\left(\frac{x-x_i}{h}\right),
\end{equation}
where $h>0$ is a length scale parameter (also known as bandwidth parameter \cite{KDEvariable1992,GPpost}, correlation length \cite{randomFields2011,panayot2017random}, or smoothing parameter \cite{kernelBandwidth1989}).
A representative example of a density kernel $k(\cdot)$ is the widely used Gaussian kernel 
$$k(x)=\frac{1}{(2\pi)^{d/2}}e^{-\frac{\norm{x}^2}{2}},\quad x\in\mathbb{R}^d.$$
We filter the dataset $X$ by evaluating the kernel density estimator in \eqref{eq:kde} for all points in $X$ and then keeping the ``high-density'' points for which $\rho(x)$ is larger than a threshold, i.e., $\rho(x) > \gamma\max\limits_{x\in X}\rho(x)$ for some prescribed threshold parameter $\gamma\in (0,1)$.
The set of the remaining ``high-density'' points, denoted by $\widetilde{X}$, will be used for the downstream maximum likelihood estimation outlined in Section \ref{sec:formulation} for detecting singularity.
It can be seen that the larger the threshold parameter $\gamma$ is, the smaller (or more concentrated) the filtered subset $\widetilde{X}$ will be.
The full filtering algorithm is presented in Algorithm \ref{alg:kde}.

\vspace{.15in}

\noindent\textbf{Filtering using $k$-nearest neighbors ($k$NN).}
The second filtering method relies on the \emph{$k$-nearest neighbors} search.
This approach implicitly indicates the density at a point $x$ by using the size of a cluster around $x$, where the ``cluster'' is represented by the $k$ nearest points in $X$ to $x$.
For each $x\in X$, we first compute its $k$-nearest neighbors in $X$ and then evaluate the sum of squared distances from the $k$ neighbors to $x$. A smaller distance indicates a smaller cluster, thus reflecting a higher density around $x$. Using this approach, ``high-density'' points in $X$ will be selected to form the filtered subset $\widetilde{X}$.
A threshold parameter $\gamma$ is used for the purpose of filtering, similar to the KDE-based approach. A larger $\gamma$ will lead to a smaller (more concentrated) filtered subset $\widetilde{X}$.
The $k$NN-based filtering algorithm is presented in Algorithm \ref{alg:knn}.

\begin{figure}
    \centering
    \includegraphics[width=0.9\linewidth]{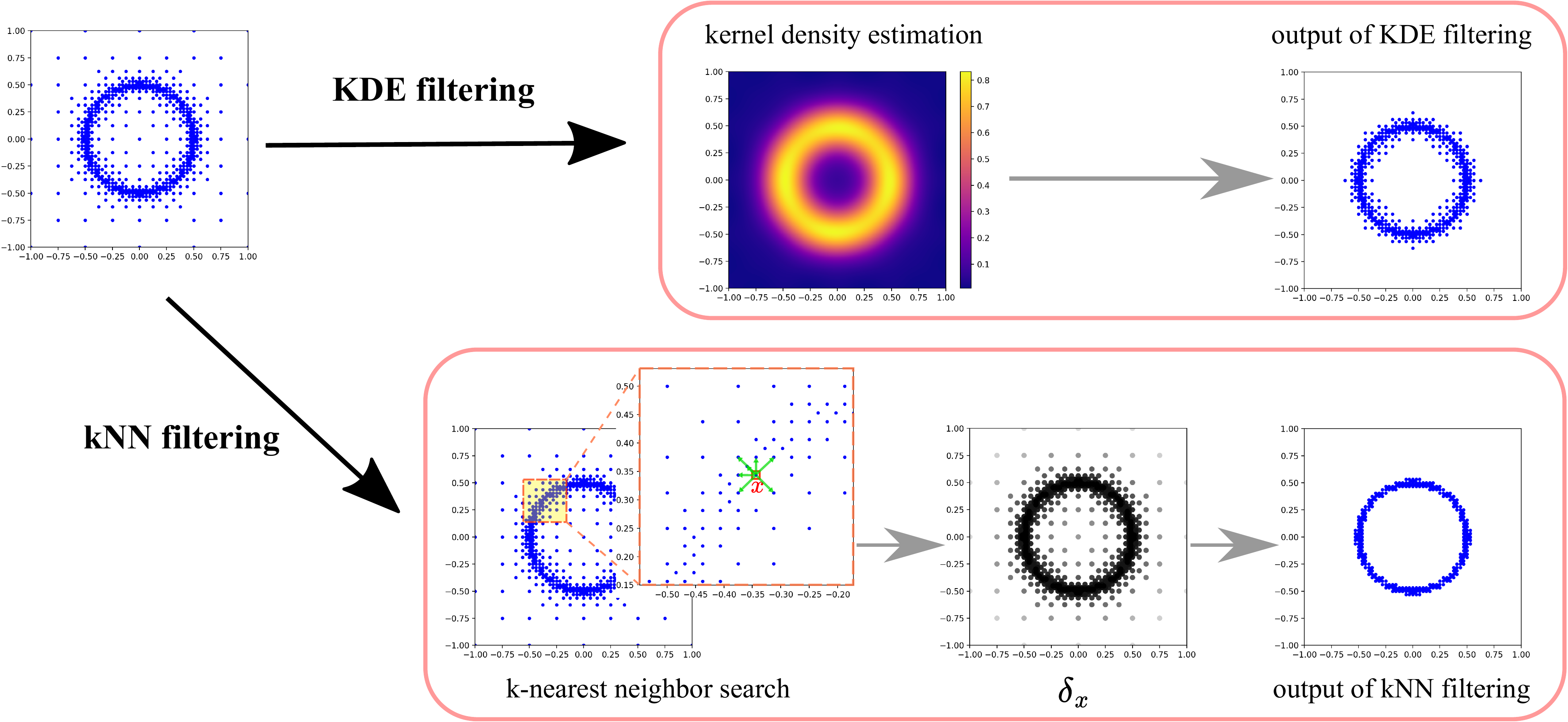}
    \caption{Illustration of two filtering methods based on kernel density estimation (KDE) and $k$ nearest neighbor ($k$NN) search}
    \label{fig:filtering}
\end{figure}

\begin{algorithm}[hbt!]
\caption{Filtering via kernel density estimation}
\label{alg:kde}
\begin{algorithmic}
\STATE \textbf{Input:} A finite set of points $X\subset \mathbb{R}^d$
\STATE \textbf{Hyperparameters:} length scale $h$ in \eqref{eq:kde}, filtering threshold $\gamma\in (0,1)$
\STATE \textbf{Output:} A subset $\widetilde{X}$ of $X$ that contains points in ``high-density'' regions with threhold parameter $\gamma$
\vspace{.1in}
\STATE Compute the kernel density $\rho(x)$ in \eqref{eq:kde} for each $x\in X$
\STATE Define $\rho_* = \max\limits_{x\in X} \rho(x)$
\STATE Initialize $\widetilde{X}=\emptyset$
\FOR{$x$ in $X$}
    \IF{$\rho(x) > \gamma\rho_*$}
        \STATE $\widetilde{X}=\widetilde{X}\cup \{x\}$
        \hfill \COMMENT{add $x$ if density estimate at $x$ is sufficient high}
    \ENDIF
\ENDFOR
\RETURN Filtered data $\widetilde{X}$
\end{algorithmic}
\end{algorithm}

\begin{algorithm}[hbt!]
\caption{Filtering via $k$-nearest neighbors}
\label{alg:knn}
\begin{algorithmic}
\STATE \textbf{Input:} A finite set of points $X\subset \mathbb{R}^d$
\STATE \textbf{Hyperparameters:} number of nearest neighbors $k$ in $k$NN, filtering threshold $\gamma\in (0,1)$
\STATE \textbf{Output:} A subset $\widetilde{X}$ of $X$ that contains points in ``high-density'' regions based on the number of neighbors relative to a threshold parameter $\gamma$
\vspace{.1in}
\STATE Apply $k$-nearest neighbors algorithm to $X$ to obtain $k$ neighbors for each $x$, stored as $\mathcal{N}_x$
\FOR{$x$ in $X$}
    \STATE Compute the sum of squared distances to the neighbors: $\delta_x=\sum\limits_{z\in \mathcal{N}_x} \norm{x-z}^2$
\ENDFOR
\STATE Compute $\delta_*=\min\limits_{x\in X}\delta_x$
\FOR{$x$ in $X$}
    \IF{$\delta_x<\delta_*/\gamma$}
        \STATE $\widetilde{X}=\widetilde{X}\cup \{x\}$
        \hfill \COMMENT{add $x$ if the cluster size of $\mathcal{N}_x$ is sufficiently small}
    \ENDIF
\ENDFOR
\RETURN Filtered data $\widetilde{X}$
\end{algorithmic}
\end{algorithm}

\paragraph{Computational efficiency}
For a training dataset with $N$ points, evaluating the kernel density estimator \eqref{eq:kde} at \emph{one} point induces $O(N)$ computational complexity, thus leading to a total of $O(N^2)$ complexity for the evaluation at $N$ points.
Similarly, performing $k$-nearest neighbor search for $N$ points will also require $O(N^2)$ computational complexity.
Such a computational bottleneck is well-known \cite{nando2005fast,bishop2006book} and a variety of methods have been developed to reduce the complexity.
For example, tree-based algorithms \cite{gray2000nbody,gray2003nonparametric} and hierarchical matrix techniques \cite{fastgauss1991,yang2003IEEE,hack2015book,smash,HiDR} can be used to accelerate the non-local kernel computation in kernel density estimation.
For $k$NN, a large number of techniques and software packages have been developed to improve the computational efficiency, including tree based nearest neighbor search \cite{kdtree1984,gray2003nonparametric,kdtree2006}, approximate nearest neighbors \cite{knn1998optimal,knn2009app,knn2018app,annoy}, gradient descent based search \cite{NNdescent2011}, etc.
The computational complexity can usually be reduced to $O(N\ln N)$ or $O(N)$ using these techniques.
We remark that the focus of the manuscript is on singularity detection, and computational accelerations will be investigated in a future date.

\subsection{The full data-driven SSL algorithm for singularity detection}
\label{sub:SSLfull}
Using filtering (Section \ref{sub:filtering}) as the pretext task and MLE (Section \ref{sec:formulation}) as the downstream task, a self-supervised learning framework can be built for singularity detection with unlabeled mesh data.
The full self-supervised singularity detection algorithm is summarized in Algorithm \ref{alg:alg2}.
If the given data is Type II, then we merge all batches into to a single one that contains all points.
Note that for Type II data, Algorithm \ref{alg:alg1} applies different weights to different batches in order to mitigate input perturbation or label corruption. Compared to Algorithm \ref{alg:alg1}, the SSL approach in Algorithm \ref{alg:alg2} does not distinguish Type I and Type II data formats much since it processes a single set that contains all data points.
More importantly, Algorithm \ref{alg:alg2} uses filtering to directly remove ``noise'' in the data, which helps to address input perturbation or label corruption.
The benefit of using filtering as a pretext task is two-fold:
(1) it reduces ``noise'' in the raw data and helps to improve the accuracy of singularity detection;
(2) using a filtered subset $\widetilde{X}$ of the given data $X$ as training data reduces the computational cost when minimizing the loss function, since $\widetilde{X}$ has a smaller size than $X$.
Therefore, using the SSL framework with filtering improves both accuracy and efficiency for singularity detection compared to the baseline Algorithm \ref{alg:alg1}.

\begin{remark}
\label{rm:AMR}
The performance of data-driven singularity detection depends on the quality of mesh data, which, if generated by adaptive mesh refinement, will rely on the a posteriori error estimator used for estimating the error and guiding the refinement (cf. \cite{ainsworthoden2000book,verf2013book}). Specifically, for PDEs with certain small-scale parameters, it is well-known that a \emph{non-robust} error estimator can possibly guide the mesh refinement towards regions where the function is smooth, generating excessive nodes \emph{away} from the singularity (cf. \cite{bernardi2000,localL2}). \emph{Robust} error estimators can avoid this issue and generate a good mesh refined towards the singularity \cite{bernardi2000,localL2,confusion}. Developing \emph{robust} a posteriori error estimators is crucial for resolving the solution singularity and the efficient numerical solution of challenging PDEs. Detailed study on this topic can be found in \cite{bernardi2000,verf1998reaction,verf2005confusion,ainsworth2011confusion,localL2,confusion}.
    
\end{remark}

\begin{algorithm}[hbt!]
\caption{Data-driven self-supervised singularity detection (\emph{with} filtering)}
\label{alg:alg2}
\begin{algorithmic}
\STATE \textbf{Input:} Type I data aggregate $X$ (nodes from a mesh) or Type II data batches $\XXi{0},\dots,\XXi{R}$ from $R$ steps of AMR
\STATE \textbf{Hyperparameters:} Basis functions $\phi_1,\phi_2,\dots$ in \eqref{eq:f}, filtering hyperparameters $\gamma$, $h$ or $k$
\STATE \textbf{Output:} A function $f$ that minimizes the loss $\widetilde{L}_1$ defined below
\IF{Input is Type II}
\STATE Define $X=\bigcup\limits_{i=1}^R \XXi{i}$
\ENDIF
\STATE Compute the filtered subset $\widetilde{X}$ by applying filtering to $X$ via Algorithm \ref{alg:kde} (with parameters $\gamma,h$) or Algorithm \ref{alg:knn} (with parameters $\gamma,k$)
\STATE Define the loss function 
$\widetilde{L}_1(\mathbf{c})=\sum\limits_{x\in \widetilde{X}} f(x)^2$
\text{with} $f$ \text{in} \eqref{eq:f}
\STATE Solve the constrained minimization 
$\mathbf{c}_*=\argmin\limits_{\norm{\mathbf{c}}=1} \widetilde{L}_1(\mathbf{c})$
in \eqref{eq:minL1} to obtain $\mathbf{c}_*$
\RETURN $f$ in the form \eqref{eq:f} with coefficients $\mathbf{c}_*$
\end{algorithmic}
\end{algorithm}

\section{Numerical experiments}
\label{sec:experiments}

In this section, we perform numerical experiments to investigate algorithms for singularity detection. 
In each experiment, the singularity set $\mathcal{S}$ is the set of polynomial roots in a two-dimensional rectangular domain $\Omega$ depending on the underlying problem. Polynomial ansatz is used in most experiments. For $(x,y)\in\mathbb{R}^2$, the general form of a polynomial with degree $n$ is given by
$$f(x,y) = \sum_{i=0}^n \sum_{j=0}^{n-i} c_{k}x^iy^j.$$
 The coefficient $c_{k}$ corresponding to the term $x^iy^j$ is indexed by:  
 \[
 k= i(n+2) -i(i+1)/2+j.
 \]
For example, when $n=2$, the polynomial expands as:
\begin{equation}
f(x,y) = c_0 + c_1y+ c_2y^2+c_3x+c_4xy+c_5x^2. \label{eq:polydegree2}
\end{equation}
Given the training data, the unknown polynomial coefficients will be calculated using the proposed singularity detection algorithm.
The choice of hyperparameters in each filtering method will be stated in the experiment. In particular, for KDE-based filtering, Silverman's rule \cite{silverman2018density} is used for the length scale parameter $h$ in all experiments, i.e.,
$h=\left[\frac{1}{4}N(d+2)\right]^{-\frac{1}{d+4}}$ for $N$ data points in $d$ dimensions.


Extensive experiments are conducted to investigate different aspects of the problem, and we give a brief overview as follows.
Example 1 applies the baseline Algorithm \ref{alg:alg1} to Type I and Type II data to demonstrate the possible pathological outcomes for singularity detection when using the raw data directly for training.
Example 2 illustrates the superior performance of the self-supervised singularity detection framework in Algorithm \ref{alg:alg2}, where filtering is used as a pretext task.
Example 3 considers singularity detection for various types of challenging singularities, including boundary layer, `X' shape, and concentric semicircles. Both KDE and $ k$NN-based filtering methods will be tested.
Example 4 presents singularity detection with Fourier basis functions to demonstrate the flexibility of the proposed framework.

\subsection{Example 1. Baseline Algorithm \ref{alg:alg1} (no filtering) for Type I and Type II data}
\label{exp:circle-type1-type2}

We first consider an example of the solution of the singularly-perturbed reaction-diffusion equation:
\begin{align}
-\varepsilon \Delta u + u &= g_1 \;\;\text{ in }\;\; \Omega\quad \text{with}\;\; \varepsilon=10^{-4}, \nonumber\\
u & = g_b \;\;\text{ on }\;\partial\Omega. 
\label{eq:reaction-difussion}
\end{align}
The exact solution is chosen as
$$u(x,y)= \tanh \left(\varepsilon^{-1/2} (x^2 + y ^2 - \frac{1}{4}))\right)-1.$$ 
This is a commonly used test problem in the adaptive mesh refinement and a posteriori error estimation literature (cf. \cite{caizhang2010,confusion}).
This solution captures a reaction-diffusion phenomenon with a localized steep transition influenced by the small parameter $\varepsilon$. The transition occurs at an interior singular layer (see Figure~\ref{fig:circle}) given by the equation below
\begin{equation}
    x^2+y^2=\frac{1}{4}.\label{eq:singularity-circle}
\end{equation}
For singularity detection, we choose $f$ to be a polynomial of degree two as in \eqref{eq:polydegree2}.

\paragraph{Datasets}
The dataset is generated by AMR using the robust hybrid estimator proposed in~\cite{confusion}. 
The initial mesh in AMR consists of $4\times 4$ congruent squares, each of which is partitioned into two triangles connecting bottom-left and top-right corners. After 44 refinement steps, we obtained a mesh containing $13664$ nodes in total. We denote this large array of nodes as $X_0$ and use much smaller subsets to construct the training data.
For Type I data format, four different training sets are used in this experiment:
\begin{itemize}
    \itemsep0in
    \item Figure \ref{fig:1a}: the first $200$ nodes in $X_0$;
    \item Figure \ref{fig:1b}: the first $300$ nodes in $X_0$;
    \item Figure \ref{fig:1c}: the subset of $X_0$ that contains the $200$th to the $300$th nodes;
    \item Figure \ref{fig:1d}: the first $2500$ nodes in $X_0$.
\end{itemize}
For Type II data format, we use the batches generated after $R=17$ AMR steps $\XXi{0},\XXi{1},\dots,\XXi{R}$, whose union contains $322$ nodes in total.
To study the performance of the baseline singularity detection in Algorithm \ref{alg:alg1}, three different choices of the variance parameter $\sigma_i^2$ for batch $\XXi{i}$ are considered.
\begin{equation}
\label{eq:weights}
\sigma_i^2 = 1, \quad \sigma_i^2 = \frac{1}{2} \cdot 2^{2(R-i)},\quad \sigma_i^2 = \frac{1}{2} \cdot 4^{2(R-i)}\quad (i=0,\dots,R).
\end{equation}
Note that in the second choice and the third choice, the weight $\frac{1}{2\sigma_i^2}$ in \eqref{eq:minL2} for the $i$-th batch increases exponentially with $i$.

\begin{figure}
\centering
  \begin{subfigure}[t]{0.25\textwidth}
\includegraphics[width=\textwidth]{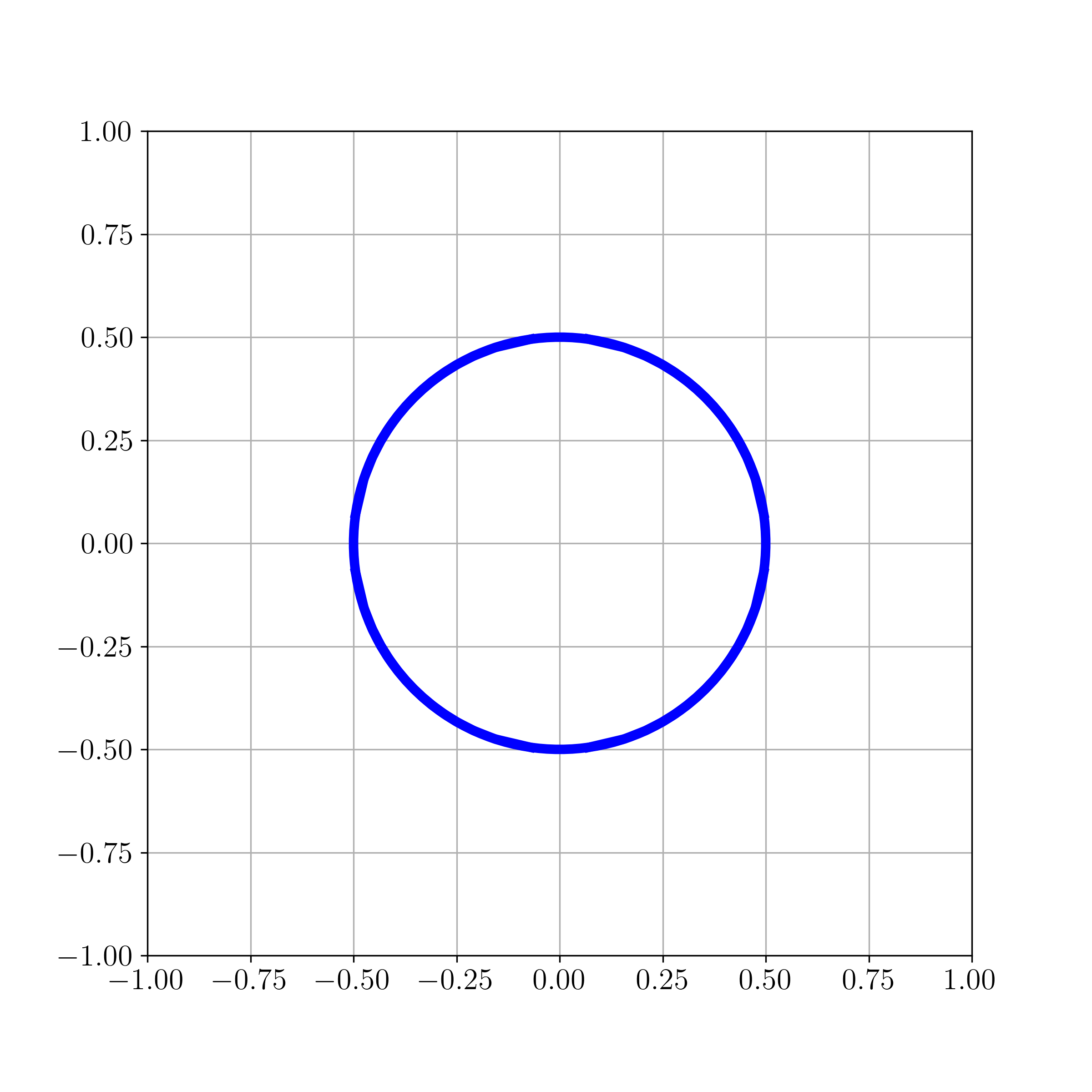}
\end{subfigure}
\hspace{.1in}
\begin{subfigure}[t]{0.25\textwidth}
\includegraphics[width=\textwidth]{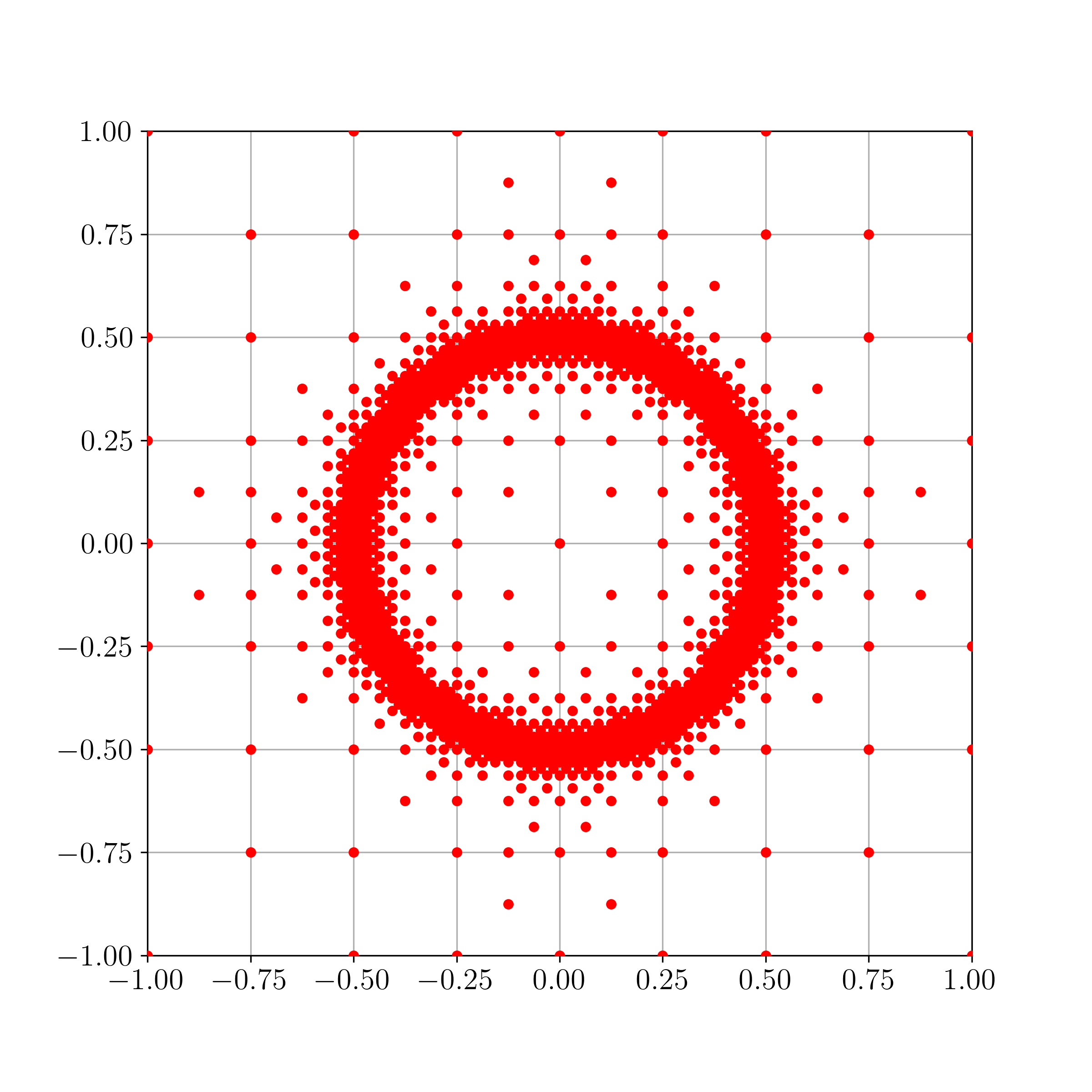}
\end{subfigure}
\caption{Left: the exact singularity given by $x^2+y^2=\frac{1}{4}$ for the PDE in \ref{eq:reaction-difussion}. Right: mesh data $X_0$ from AMR \cite{confusion} with 13,664 nodes. Subsets of $X_0$ with much smaller sizes will be used as training data in Example 1.}
\label{fig:circle}
\end{figure}

\paragraph{Observations for Figure \ref{fig:fail}} 
The singularity detection results for four Type I datasets are shown in Figure \ref{fig:fail}.
We see from the first three plots in Figure \ref{fig:fail} that the baseline singularity detection presented in Algorithm \ref{alg:alg1} can yield incorrect estimates for small-scale Type I data. For example, highly pathological results are observed in Figure \ref{fig:1a} and Figure \ref{fig:1b}, where the predicted singularity set is far from being a circle. 
Figure \ref{fig:1c} illustrates better detection with the dataset being the $200$th to the $300$th nodes in $X_0$, which contains fewer ``outliers'' than the datasets in Figure \ref{fig:1a} and Figure \ref{fig:1b}. 
Nonetheless, the learned ellipse is still noticeably different from the true singularity (a circle). 
When a much larger training set is used, as shown in Figure \ref{fig:1d}, the detected singularity set using Algorithm \ref{alg:alg1} tends to approximate the true singularity (a circle centered at the origin with radius $\frac{1}{2}$) quite accurately. 
The results suggest that while access to tail data from AMR can improve the reliability of Algorithm \ref{alg:alg1}, it is not advantageous unless the training set is sufficiently large with enough many points concentrated near the singularity to alleviate the impact of ``outliers''. In general, we see that using the baseline Algorithm \ref{alg:alg1} to detect singularity can be quite sensitive to input perturbation and label corruption.
It should also be noted that even though a larger dataset can help achieve better detection, the computational cost can become substantially higher.

\begin{figure}[h]
    \centering
    \begin{subfigure}[t]{0.23\textwidth}
    \includegraphics[width=\textwidth]{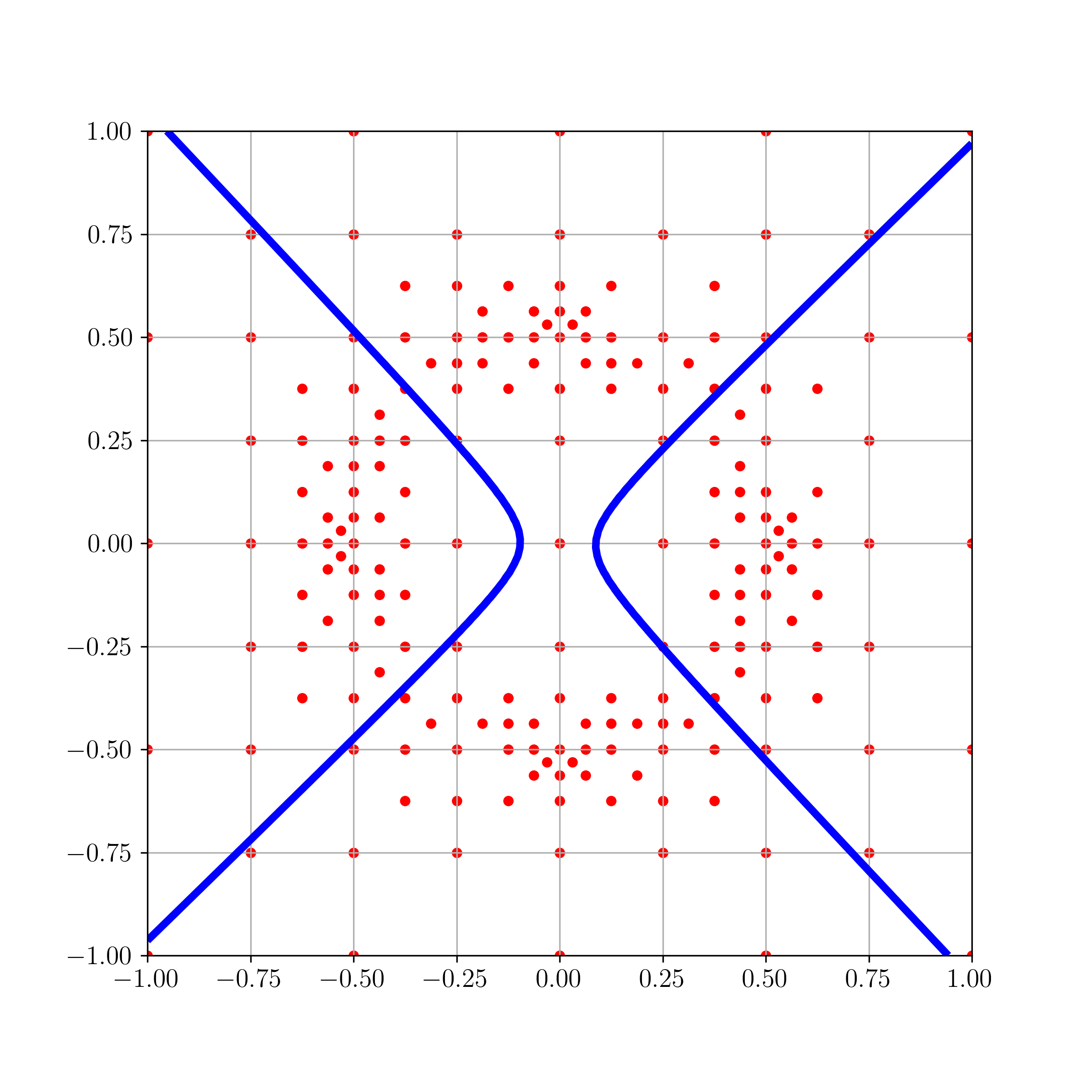}
    \caption{The first $200$ nodes in $X_0$ as training data.} \label{fig:1a}
    \end{subfigure}
     ~
     \begin{subfigure}[t]{0.23\textwidth}
    \includegraphics[width=\textwidth]{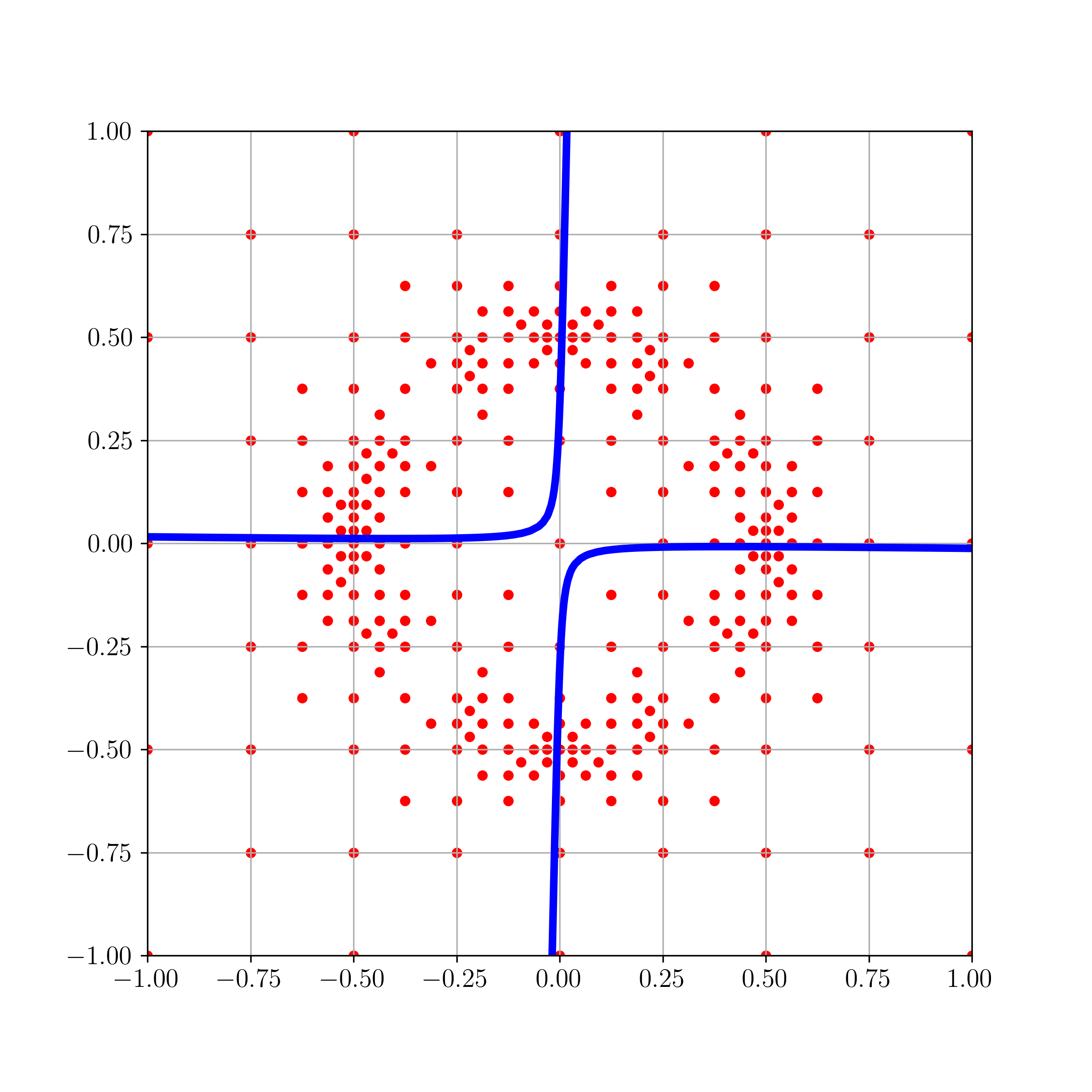}
    \caption{The first $300$ nodes in $X_0$ as training data.} \label{fig:1b}
    \end{subfigure}
     ~
    \begin{subfigure}[t]{0.23\textwidth}
    \includegraphics[width=\textwidth]{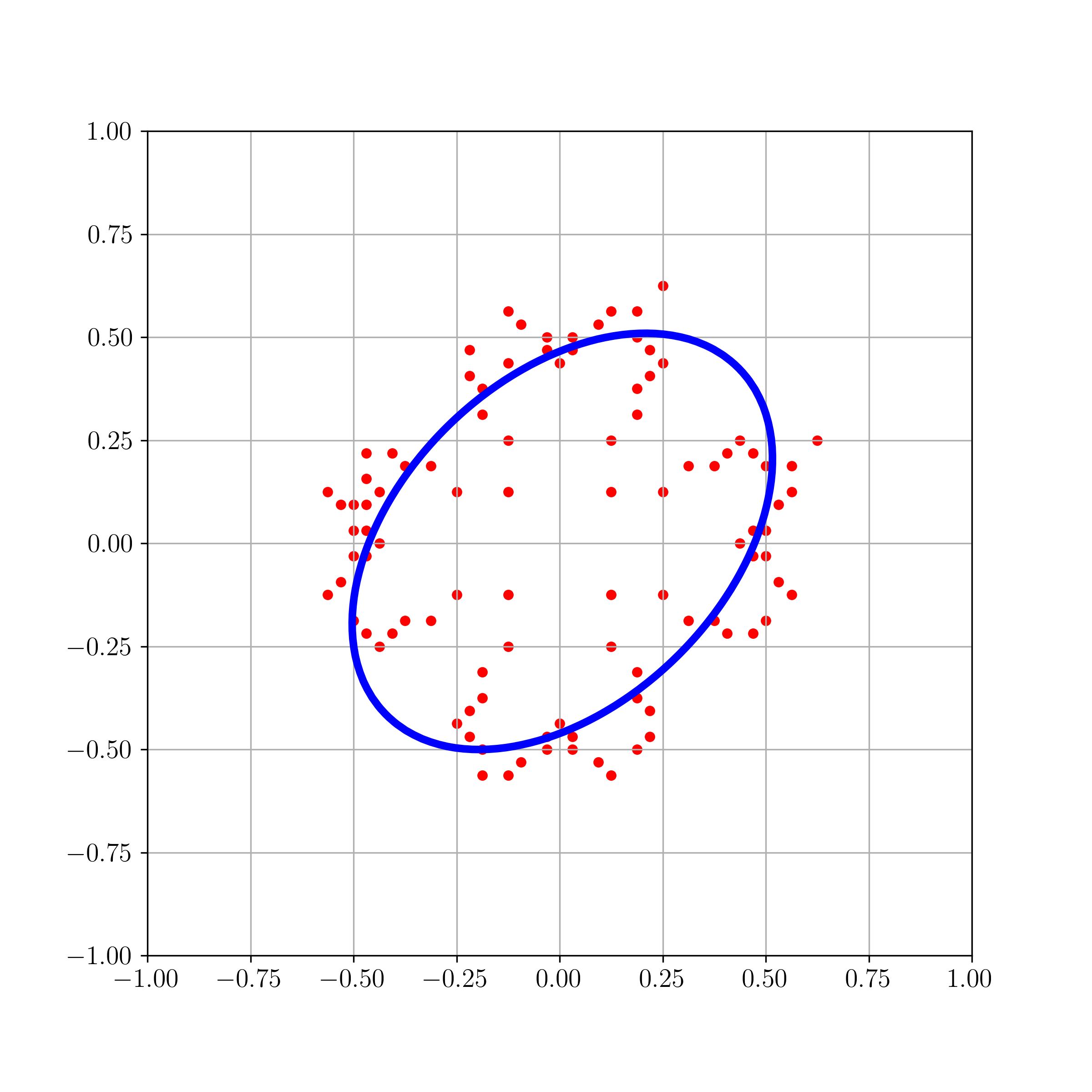}
    \caption{The $200$th to the $300$th nodes in $X_0$ as training data.}
    \label{fig:1c}
    \end{subfigure}
    ~
    \begin{subfigure}[t]{0.23\textwidth}
    \includegraphics[width=\textwidth]{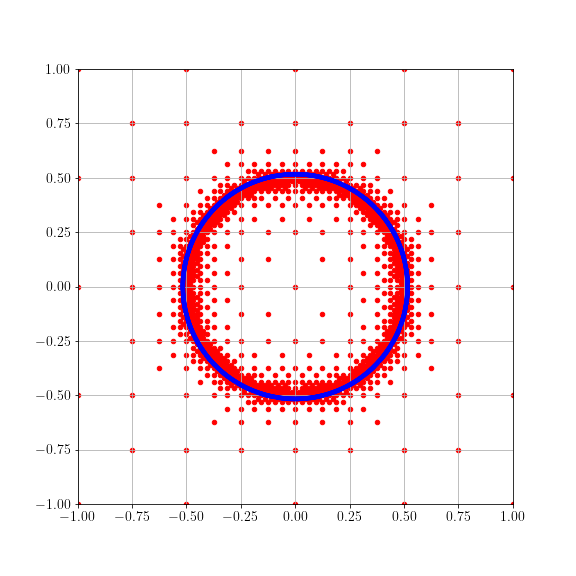}
    \caption{The first $2500$ nodes in $X_0$ as training data.}
    \label{fig:1d}
    \end{subfigure}
\caption{Example 1. Baseline singularity detection Algorithm \ref{alg:alg1} for Type I data can yield pathological results for small training sets. Here, the dots represent the training data, and the curve represents the learned singularity set.}
    \label{fig:fail}
\end{figure}

\paragraph{Observations for Figure \ref{fig:weights}}
The performance of Algorithm \ref{alg:alg1} for the Type II dataset ($322$ nodes in total) with three different variance parameters is shown in Figure \ref{fig:weights}.
It can be seen that the third choice of $\sigma_i$ yields the best result. Compared to Type I data in Figure \ref{fig:1b} with a similar number of nodes, we see that using weighted data batches can put more emphasis on the tail batches with fewer outliers, thus producing a more reliable approximation of the exact singularity.
Meanwhile, properly weighted Type II data batches can also reduce the amount of data needed to achieve a reasonable estimation. The result from Figure \ref{fig:weights} also encourages putting greater emphasis on the tail batches where the data points are more relevant (closer to the singularity) in deeper refinement stages.

Overall, this experiment shows that the batch-wise Type II data generally gives better results than Type I data when Algorithm \ref{alg:alg1} is used (without filtering the raw data). The ability to adjust the variance parameter $\sigma_i^2$ for Type II data batches allows one to mitigate the effect of ``outliers'' in the training data (nodes from the initial refinement stages of AMR).
This observation highlights the need to preprocess the Type I data to obtain better approximations to the singularity. 
This will be investigated in Section \ref{sec:preprocessing} for evaluating the performance of Algorithm \ref{alg:alg2}, where filtering is used as the pretext task in the SSL framework.

\begin{figure}[h]
 \centering
    \begin{subfigure}[b]{0.31\textwidth}
    \includegraphics[width=\textwidth]{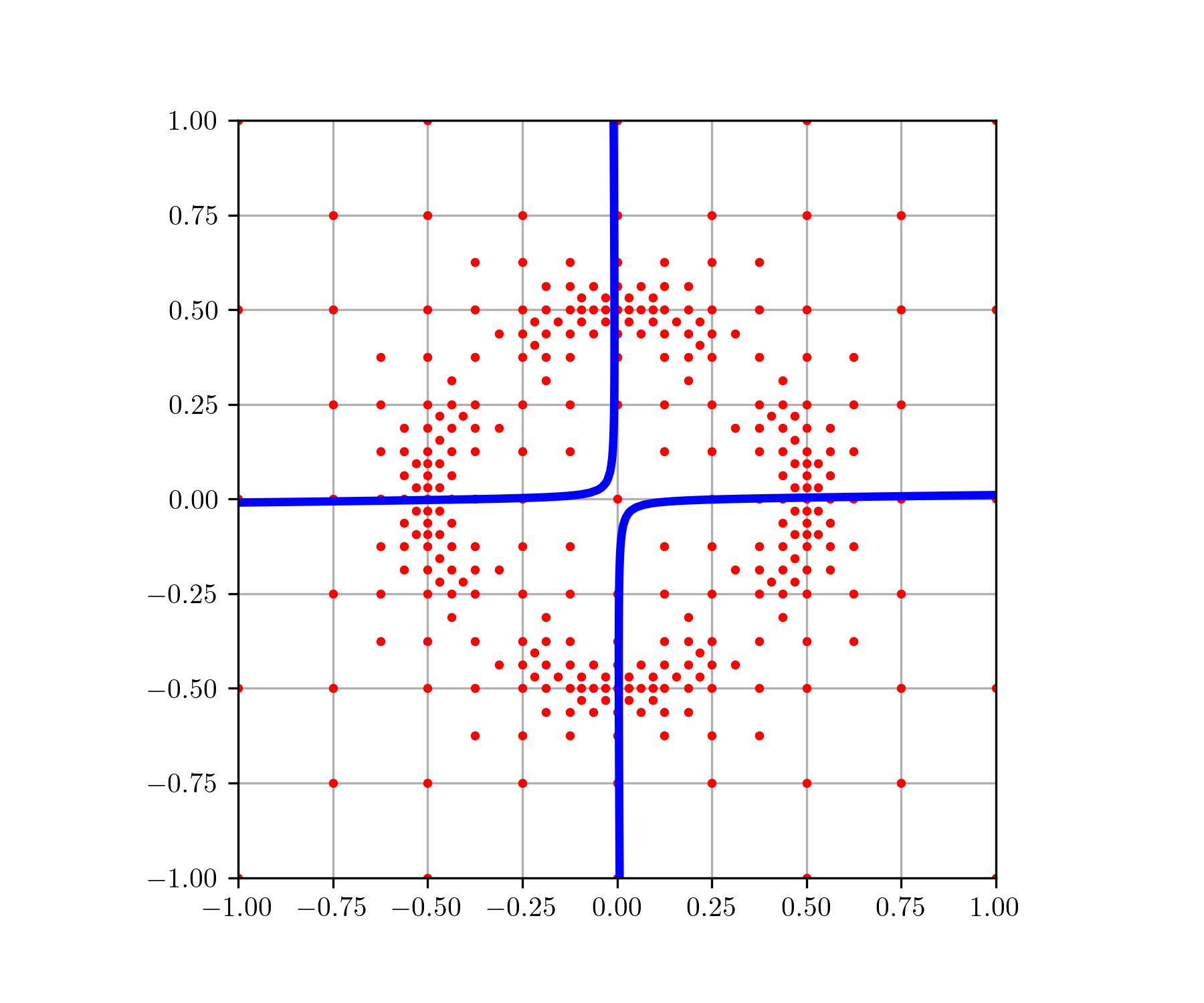}
    \caption{$\sigma_i^2=1$ in \eqref{eq:weights}}
    \label{fig:weight1}
    \end{subfigure}
    ~
    \begin{subfigure}[b]{0.31\textwidth}
    \includegraphics[width=\textwidth]{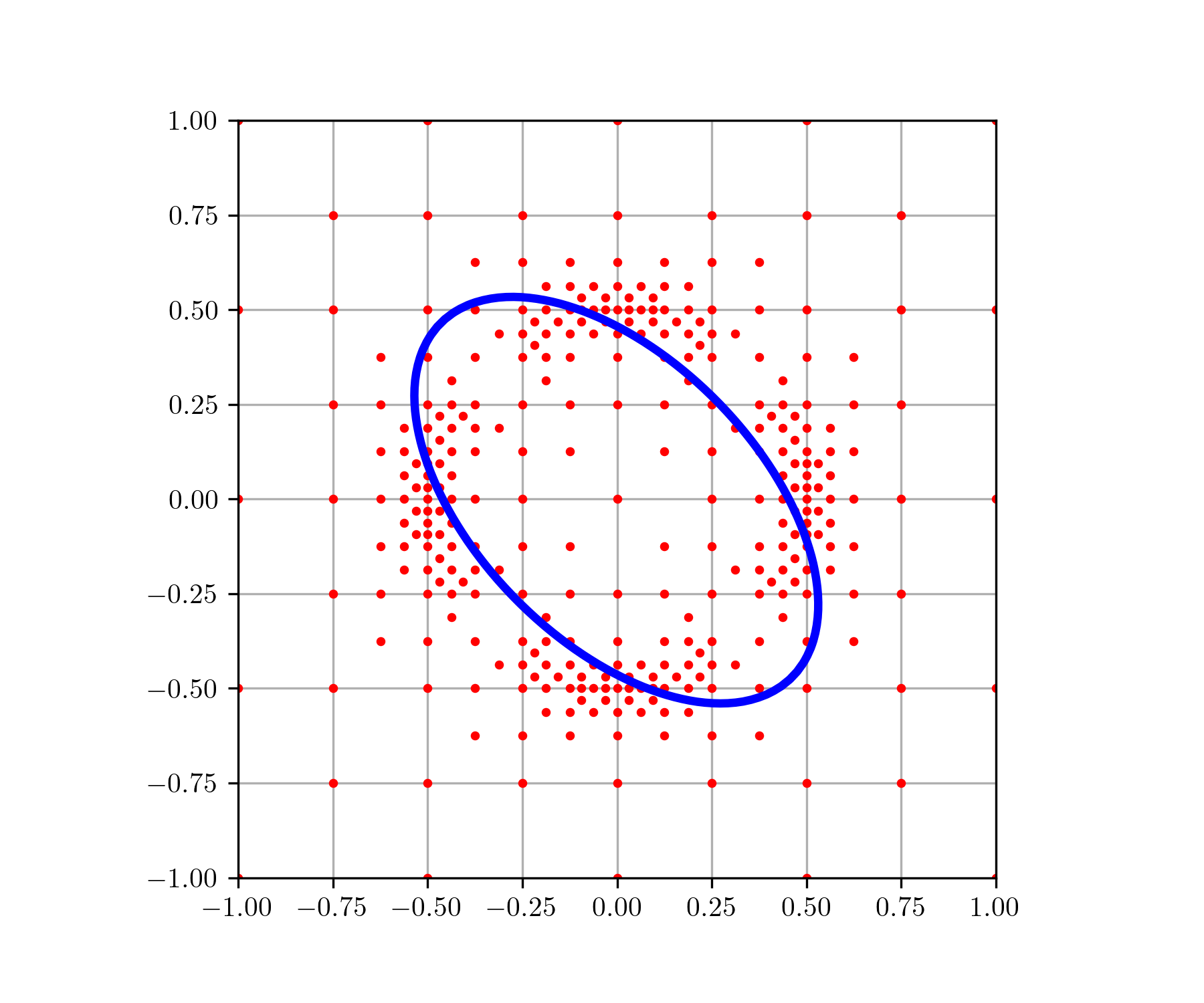}
    \caption{$\sigma_i^2=\frac{1}{2} \cdot 2^{2(R-i)}$ in \eqref{eq:weights}} 
    \label{fig:weight2}
    \end{subfigure}
    ~
    \begin{subfigure}[b]{0.31\textwidth}
    \includegraphics[width=\textwidth]{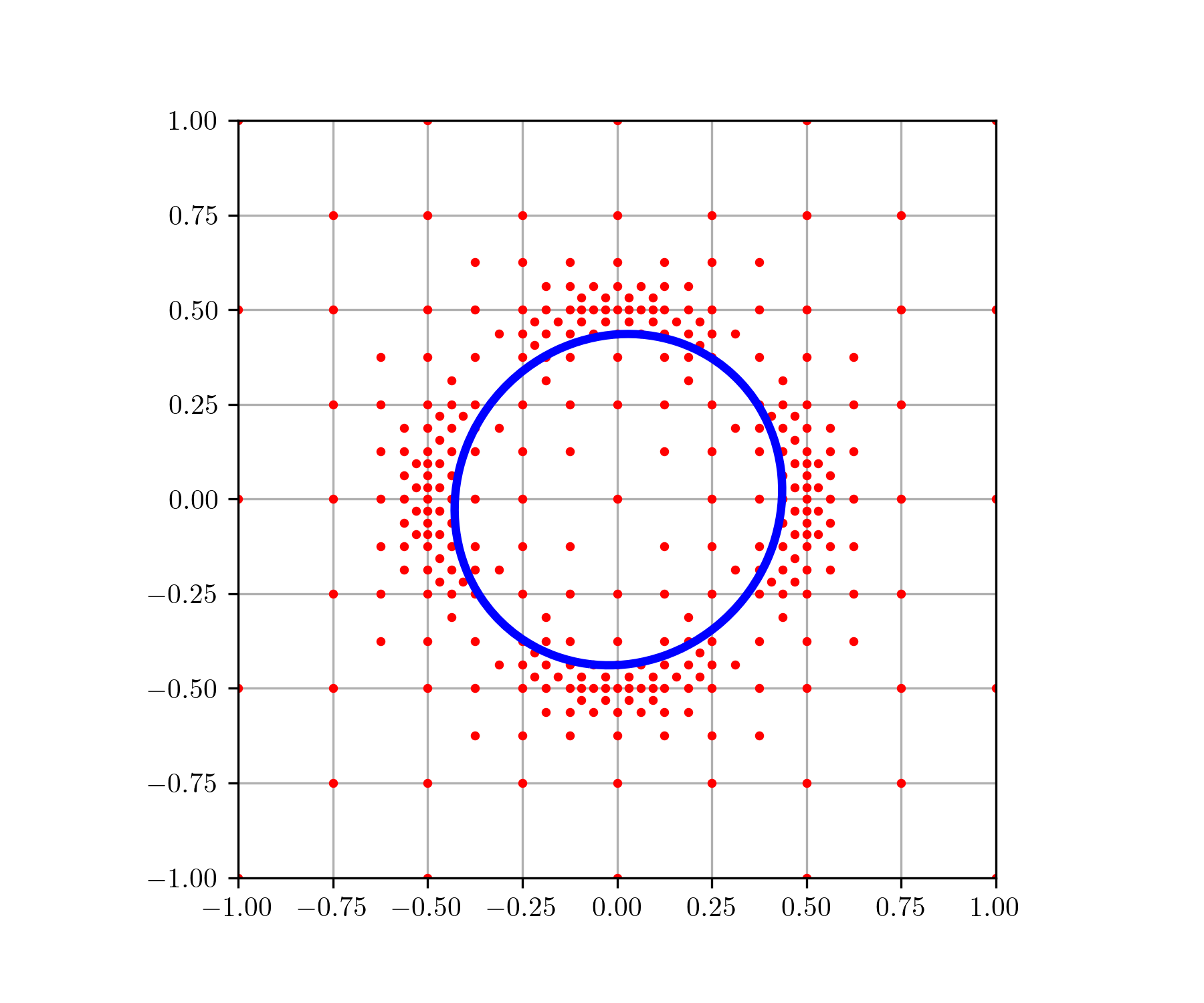}
    \caption{$\sigma_i^2=\frac{1}{2} \cdot 4^{2(R-i)}$ in \eqref{eq:weights}}
    \label{fig:weight3}
    \end{subfigure}
    \caption{Example 1. Baseline singularity detection (Algorithm \ref{alg:alg1}) for Type II data (322 nodes in total) with three choices of variance parameters $\sigma_i^2$. Red dots represent the training dataset, and the blue curve corresponding to the solid line represents the learned singularity using Algorithm~\ref{alg:alg1}  (without filtering).}\label{fig:weights}
\end{figure}

\subsection{Example 2. Algorithm \ref{alg:alg2} for Type I data}
\label{sec:preprocessing}

We see from Example 1 that Algorithm \ref{alg:alg1} may fail for small-scale Type I data. In this experiment, we demonstrate the benefits of filtering by applying Algorithm \ref{alg:alg2} to two Type I datasets: a small-scale set containing the first $300$ nodes from $X_0$ and $X_0$ itself (with $13664$ nodes) as a large-scale set.
We use KDE-based filtering with threshold parameter $\gamma=0.6$.

The results are shown in Figure \ref{fig:Ex3radius}.
The first row in Figure \ref{fig:Ex3radius} shows training data (dots) together with the detected singularity (curve).
In the second row, each plot shows the ``radius'' function $r(x,y)$ of the singularity set, which is the distance from a point $(x,y)$ on the true or detected singularity curve to the origin.
Two radius functions are graphed in each plot, corresponding to the true singularity curve and the detected singularity curve, respectively.
For the true singularity (a circle), $r(x,y)=0.5$ is constant and is represented by a solid red line.
For the detected singularity, to graph $r(x,y)$, we choose $100$ points from the detected curve and compute $r(x,y)$ at these locations.
These values are plotted using the symbol '+'.

It can be seen from Figure \ref{fig:Ex3radius} that the filtering procedure in Algorithm \ref{alg:alg2} successfully removes the ``outliers'' (nodes far away from the singularity) in the training data. 
For a small-scale dataset, this (Figure \ref{fig:Ex3-300-kde}) corrects the pathological result in Figure \ref{fig:Ex3-300} computed by Algorithm \ref{alg:alg1}. As can be seen from Figure \ref{fig:Ex3-radius-300kde}, even for small-scale data, the detected singularity set by Algorithm \ref{alg:alg2} is quite close to the true circle with radius $0.5$, where the detected radius $r(x,y)$ is around $0.504$ at any detected singularity. 
For large-scale dataset, by comparing Figure \ref{fig:Ex3-radius-X} and Figure \ref{fig:Ex3-radius-Xkde}, it is easy to see that though Algorithm \ref{alg:alg1} yields an accurate estimation of the singularity (with radius function around $0.506$), the SSL-based Algorithm \ref{alg:alg2} with filtering can substantially improve the accuracy, with radius function around $0.5006$, much closer to the true value of $0.5$.
By comparing Figure \ref{fig:Ex3-radius-300kde} and Figure \ref{fig:Ex3-radius-X}, we see that the filtering-based Algorithm \ref{alg:alg2} using only 300 nodes produces comparable (slightly better, in fact) result to Algorithm \ref{alg:alg1} with a much larger training set (13664 nodes). This demonstrates the advantage of Algorithm \ref{alg:alg2} to significantly reduce computational burden without sacrificing the estimation accuracy.

\begin{figure}
 \begin{subfigure}[t]{0.235\textwidth}
    \includegraphics[width=\textwidth]{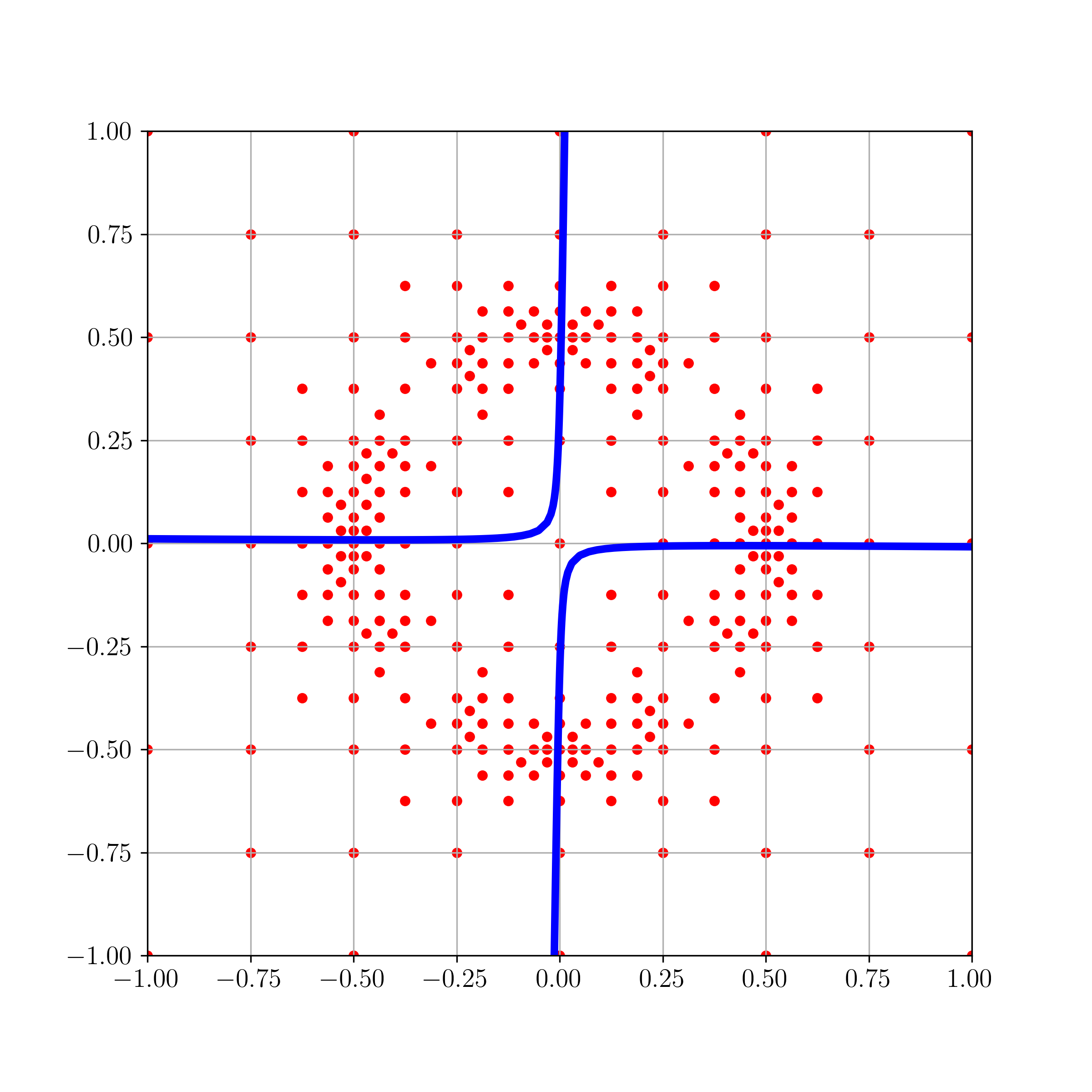}
    \caption{Algorithm \ref{alg:alg1} for small-scale data: first $300$ nodes in $X_0$}
    \label{fig:Ex3-300}
    \end{subfigure}
    ~
    \begin{subfigure}[t]{0.235\textwidth}
    \includegraphics[width=\textwidth]{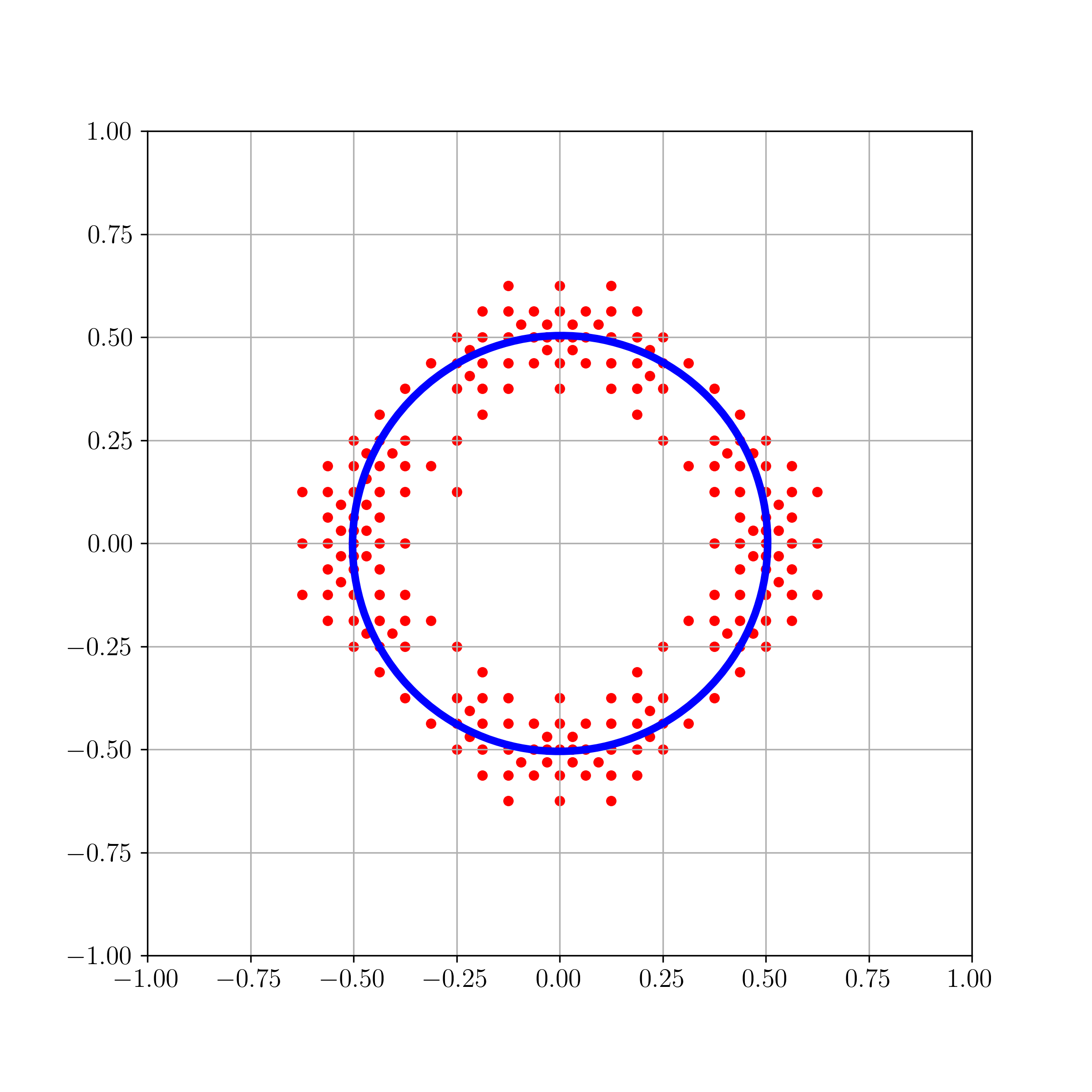}
    \caption{Algorithm \ref{alg:alg2} for small-scale data: first $300$ nodes in $X_0$}
    \label{fig:Ex3-300-kde}
    \end{subfigure}
    ~
 \begin{subfigure}[t]{0.235\textwidth}
    \includegraphics[width=\textwidth]{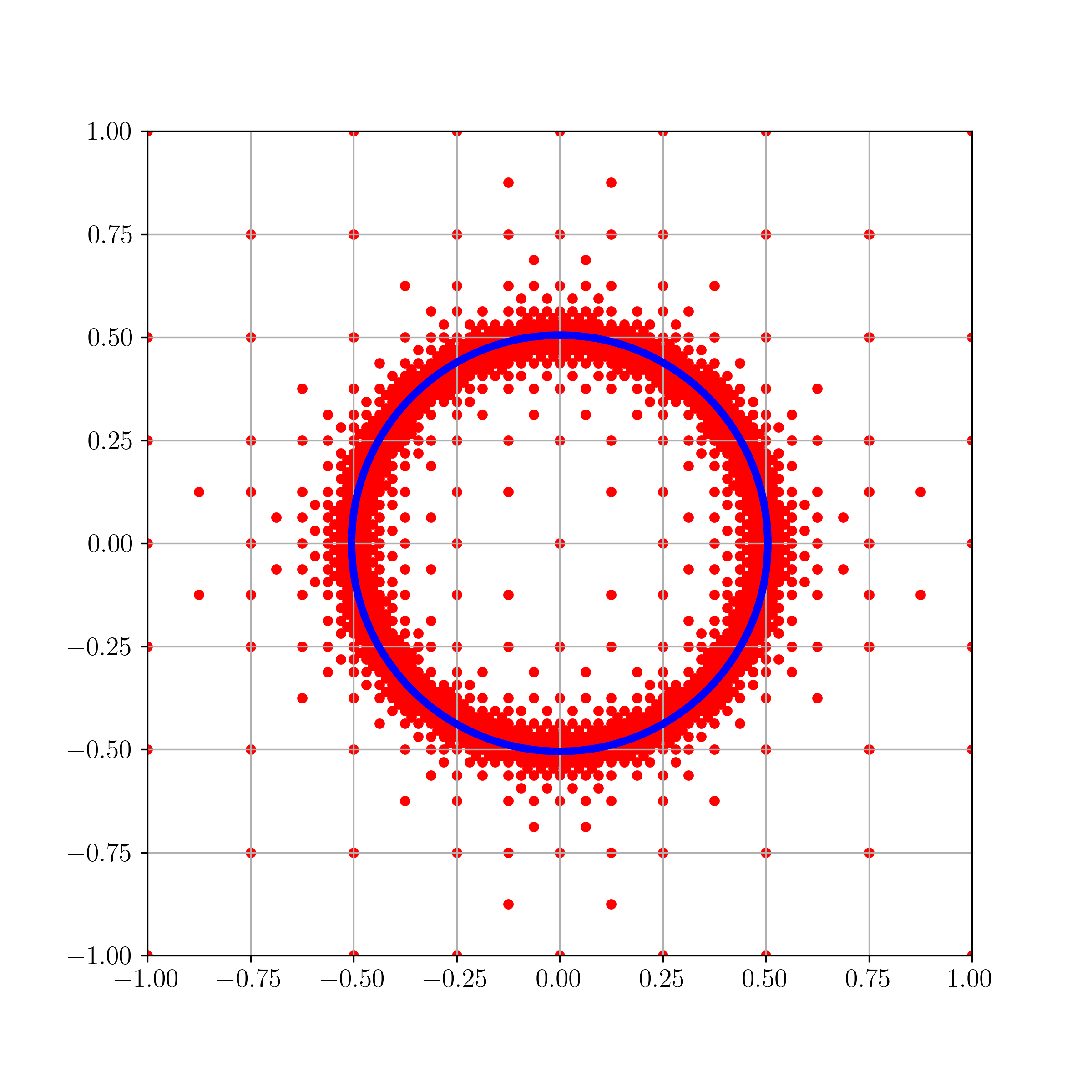}
    \caption{Algorithm \ref{alg:alg1} for $X_0$ (13664 nodes)}
    \label{fig:Ex3-C}
    \end{subfigure}
    ~
    \begin{subfigure}[t]{0.235\textwidth}
    \includegraphics[width=\textwidth]{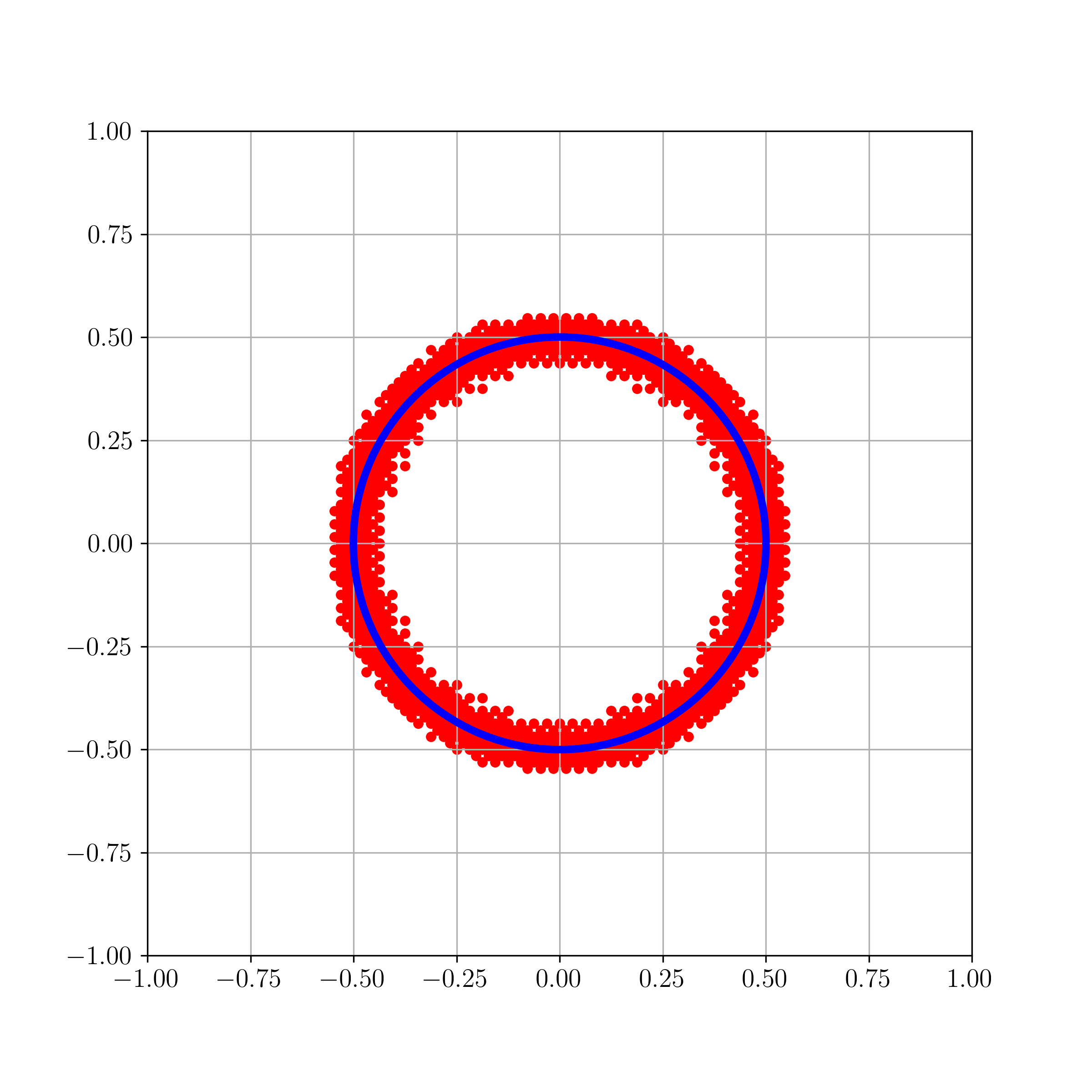}
    \caption{Algorithm \ref{alg:alg2} for $X_0$ (13664 nodes)}
    \label{fig:Ex3-C-kde}
    \end{subfigure}\\
    
     \begin{subfigure}[t]{0.235\textwidth}
    \includegraphics[width=\textwidth]{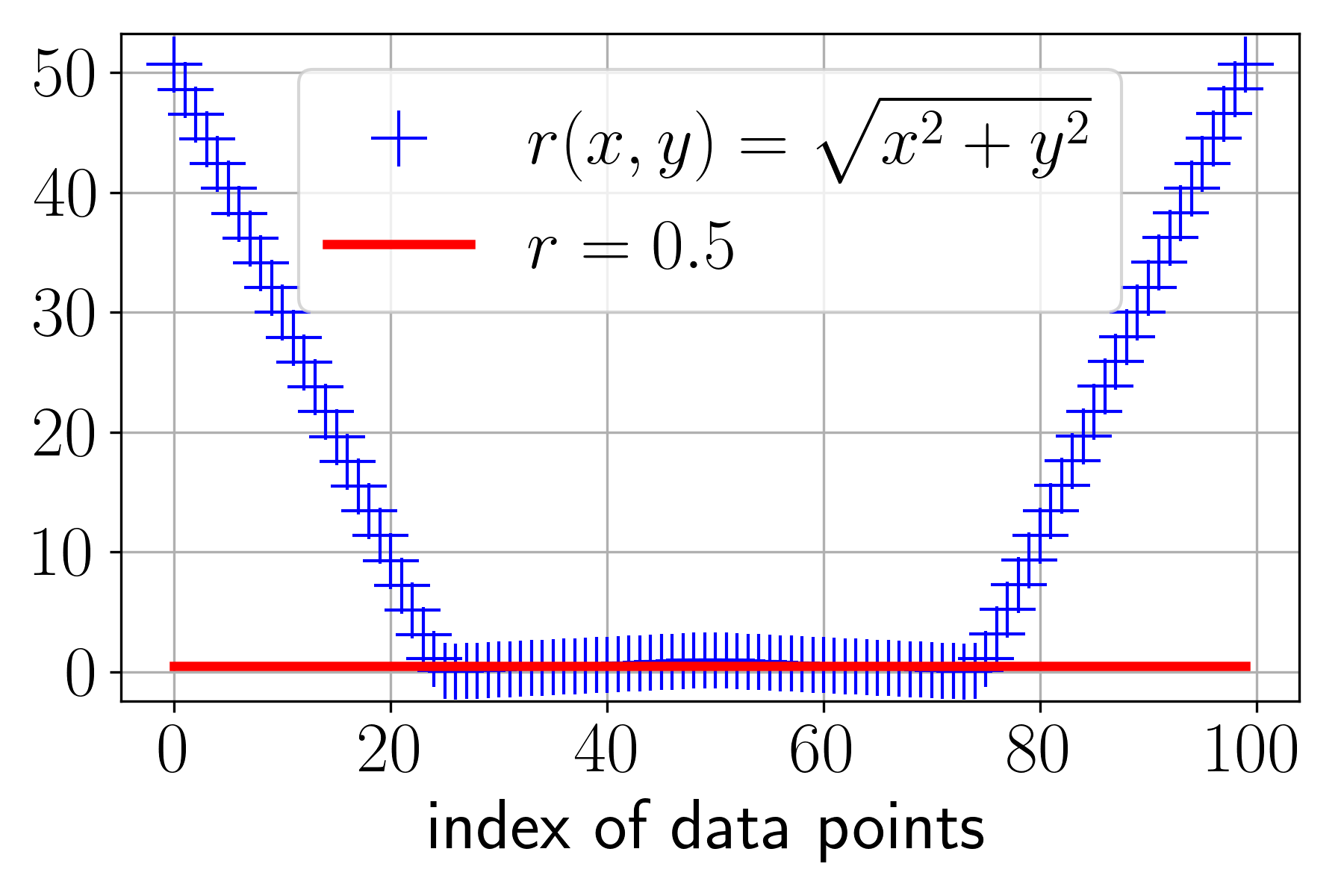}
        \caption{Radius function of singularity for Figure (a)}
    \label{fig:Ex3-radius-300}
    \end{subfigure}
    ~
  \begin{subfigure}[t]{0.235\textwidth}
    \includegraphics[width=\textwidth]{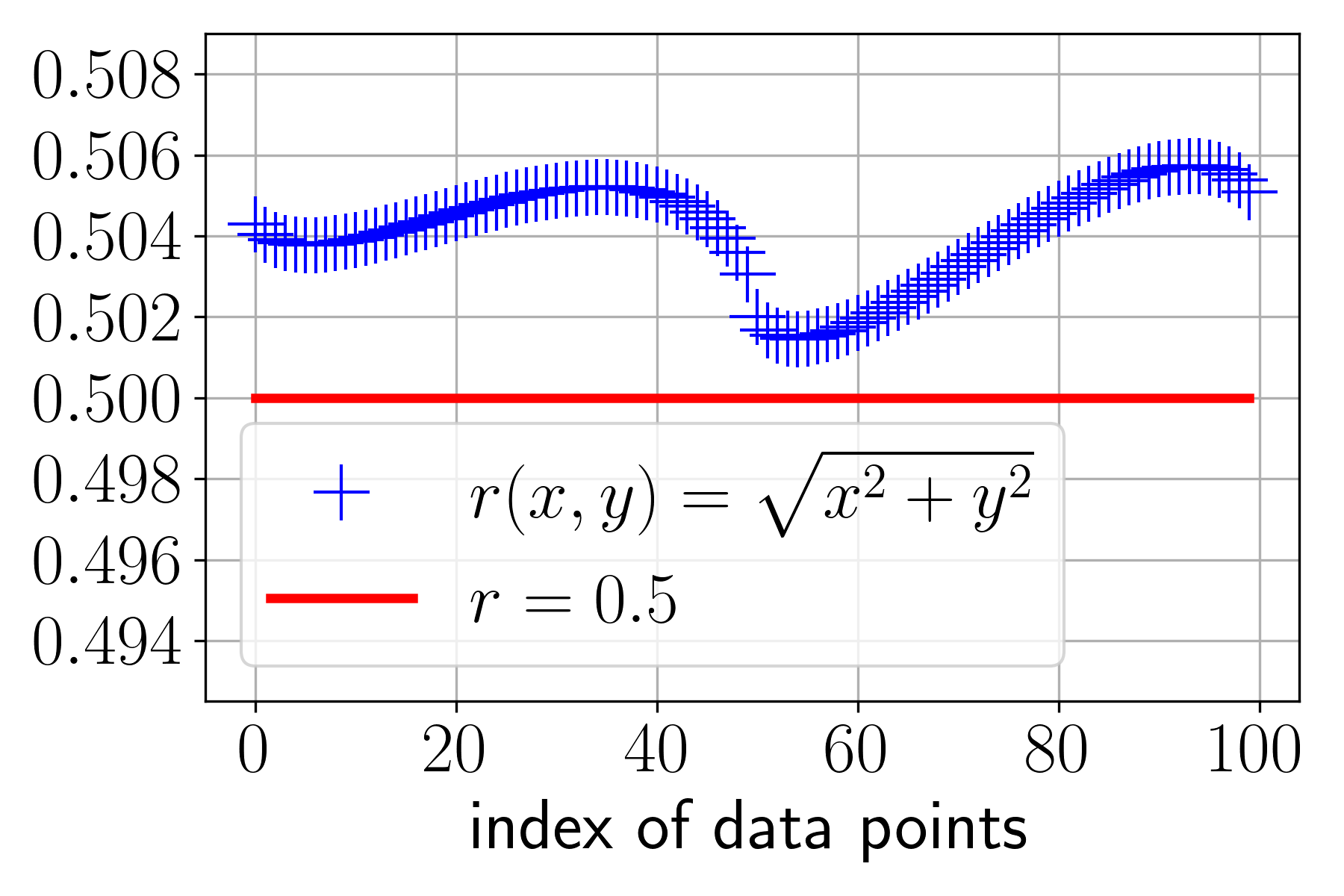}
        \caption{Radius function of singularity for Figure (b)}
    \label{fig:Ex3-radius-300kde}
    \end{subfigure}
    ~
     \begin{subfigure}[t]{0.235\textwidth}
    \includegraphics[width=\textwidth]{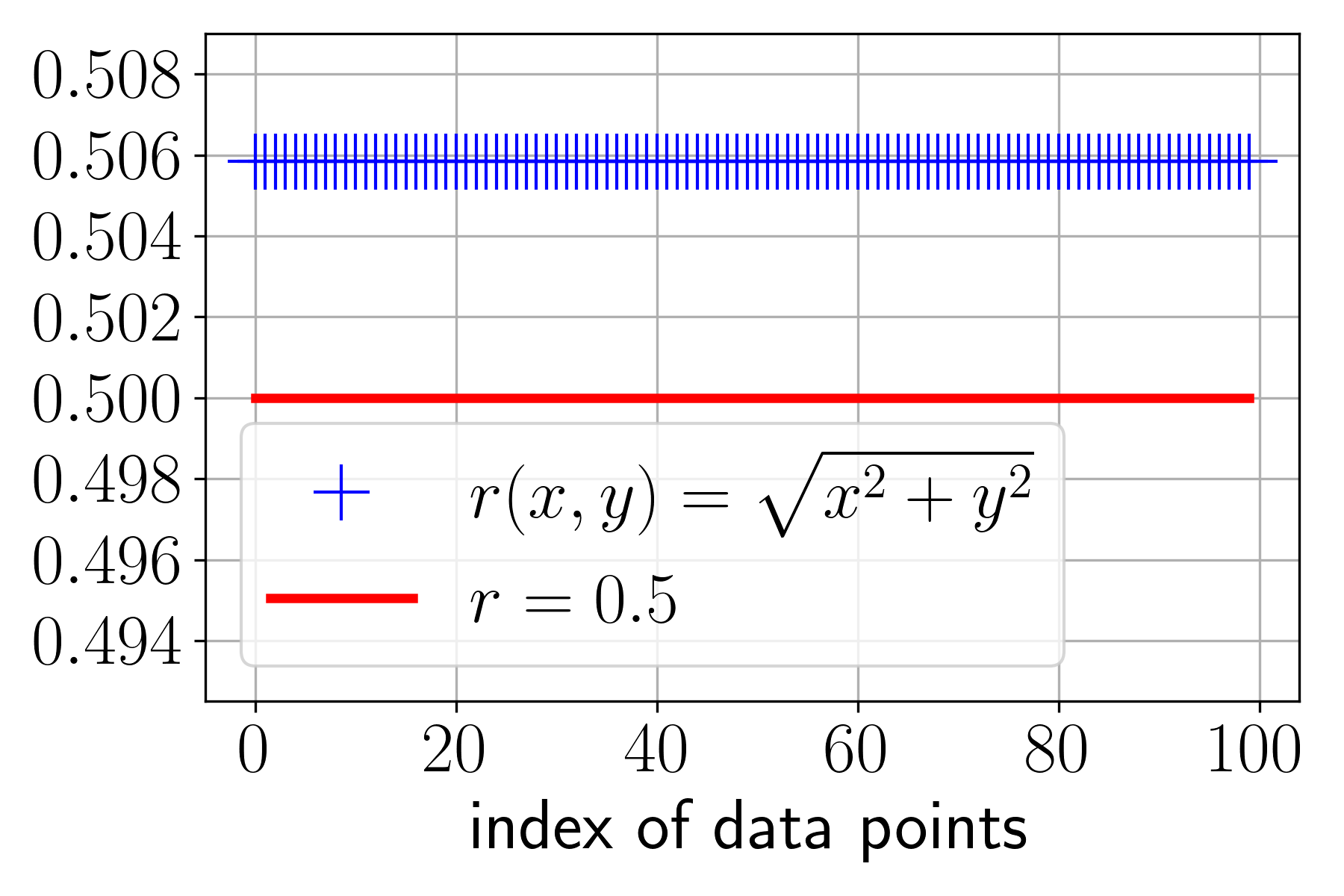}
        \caption{Radius function of singularity for Figure (c)}
\label{fig:Ex3-radius-X}
    \end{subfigure}
    ~
    \begin{subfigure}[t]{0.235\textwidth}
    \includegraphics[width=\textwidth]{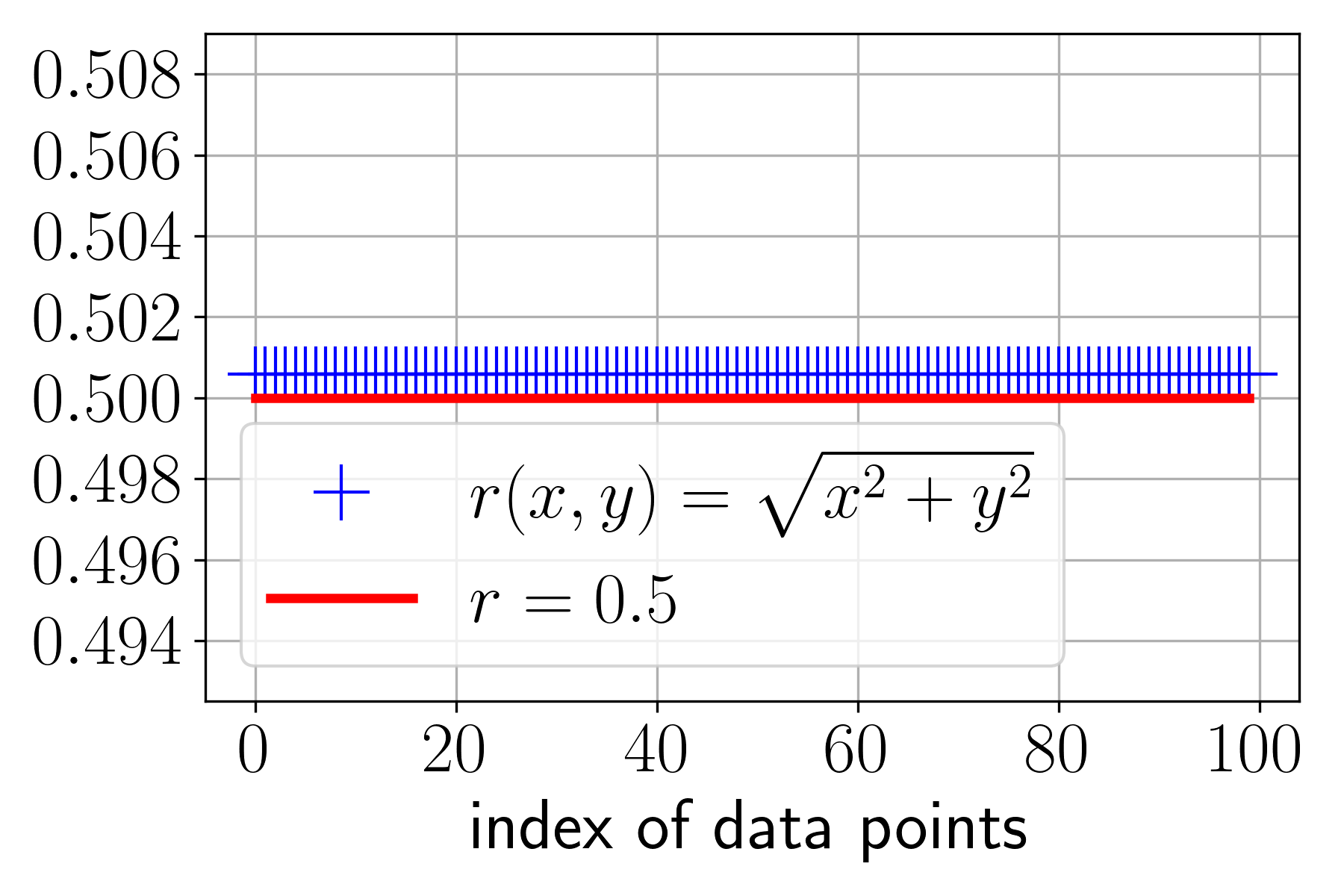}
        \caption{Radius function of singularity for Figure (d)}
    \label{fig:Ex3-radius-Xkde}
    \end{subfigure}
    \caption{Example 2. Top row: training data (dots) in each algorithm to compute the detected singularity set (curve), where (b) and (d) show the filtered data for (a) and (c), respectively. 
    Bottom row: plots of the radius function $r(x,y)$ over detected singularity curves from (a) to (d). Namely, $r(x,y)$ denotes the distance from a point $(x,y)$ on the singularity curve to the origin, where the true singularity satisfies $r(x,y)=0.5$ (solid line), and the detected singularity radius (`+' symbol) is graphed by sampling 100 points from the detected curve. The horizontal axis corresponds to the index of the 100 samples.}
    \label{fig:Ex3radius}
    \end{figure}

\subsection{Example 3. Different types of singularity curves}
In this experiment, we demonstrate the performance of the proposed self-supervised singularity detection Algorithm \ref{alg:alg2} on datasets corresponding to different types of singularity curves.
Results from both KDE-based filtering and $k$NN-based filtering will be shown.

\subsubsection{Boundary layer}
\label{subsec:Lshape}
    We consider the reaction-diffusion equation in \eqref{eq:reaction-difussion} with the exact solution given by
    $$u(x,y) = e^{-(x+1)/\sqrt{\varepsilon}}+ e^{-(y+1)/\sqrt{\varepsilon}},\quad (x,y)\in\Omega=(-1,1)^2.$$
    As shown in Figure \ref{fig:L-shape-true}, $u(x,y)$ exhibits a boundary singular layer with sharp gradients along the two sides of the boundary: $x=-1$ and $y = -1$.
    The raw mesh data (Type I), containing 796 nodes after 29 levels of refinements, is generated by AMR using the robust hybrid estimator from \cite{confusion}. Figure \ref{fig:L-shape} shows the singularity computed by (a) Algorithm \ref{alg:alg1}, (b) KDE-based Algorithm \ref{alg:alg2}, (c) $k$NN-based Algorithm \ref{alg:alg2}, and (d) the true singularity.
    The KDE-filtered subset (691 nodes) is shown in Figure \ref{fig:L-shape-kde}, where the threshold parameter is $\gamma=0.6$. The $k$NN-filtered subset (187 nodes) is shown in Figure \ref{fig:L-shape-knn} with parameters $k=5$ and $\gamma=0.6$.

    It is easy to see from Figure \ref{fig:L-shape-TypeI} that Algorithm \ref{alg:alg1} without filtering produces a highly inaccurate estimation, with the computed singularity curve deviating a lot from the true boundary layer near $(-1,-1)$, $(-1,1)$, $(1,-1)$.
    In contrast, the filtering-based Algorithm \ref{alg:alg2} captures the `L'-shaped boundary layer more accurately as shown in Figure \ref{fig:L-shape-kde} and Figure \ref{fig:L-shape-knn}.
    Table \ref{tab:L-shape-coef} shows the computed polynomial coefficients in $f(x,y)$ corresponding to the detected results from Figure \ref{fig:L-shape-TypeI} to Figure \ref{fig:L-shape-knn}, as compared to the polynomial below for the exact singularity.
    $$F_*(x,y)=0.5+0.5y+0.5x+0.5xy.$$
    We also observe from Table \ref{tab:L-shape-coef} that more concentrated training data obtained from $k$NN-based filtering in Figure \ref{fig:L-shape-knn} yield more accurate singularity detection than KDE-based filtering in Figure \ref{fig:L-shape-kde},
    while the baseline detection Algorithm \ref{alg:alg1} gives much less accurate results compared to the filtering-based Algorithm \ref{alg:alg2} (with KDE or $k$NN).

    \begin{table}[htbp]
    \centering
    \begin{tabular}{c|c|c|c|c|c|c}
    \hline
       Basis  & $1$ & $y$ & $y^2$ & $x$ & $xy$ & $x^2$\\
    \hline
        Coefficients (Figure \ref{fig:L-shape-TypeI}) & 0.583 & 0.451 & -0.116 & 0.450 & 0.476 & -0.118 \\
        Coefficients (Figure \ref{fig:L-shape-kde}) & 0.480 & 0.504 & 0.011 & 0.504 & 0.511 & 0.011\\
        Coefficients (Figure \ref{fig:L-shape-knn}) & 0.496 & 0.500 & -6.4E-4 & 0.500 & 0.504 & -6.5E-4\\
        Coefficients (exact) & 0.5 & 0.5 & 0 & 0.5 & 0.5 & 0\\
    \hline
    \end{tabular}
    \caption{Boundary singular layer: comparison of polynomial coefficients for curves in Figure \ref{fig:L-shape}}
    \label{tab:L-shape-coef}
\end{table}

    \begin{figure}[htbp]
    \centering    
    \begin{subfigure}[t]{0.24\textwidth}  \includegraphics[width=\textwidth]{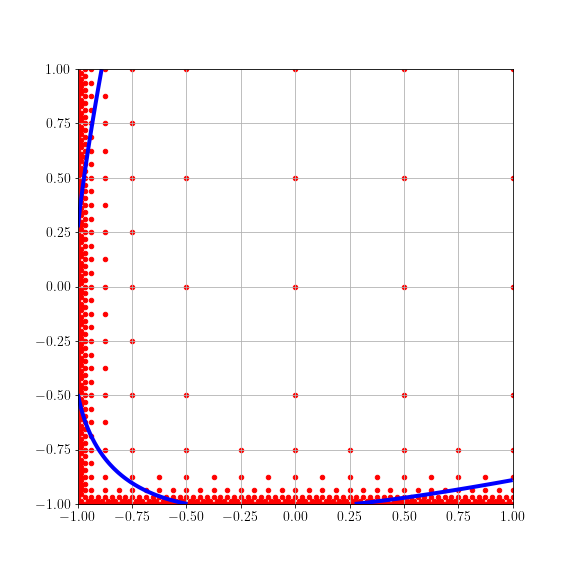}
    \caption{Algorithm~\ref{alg:alg1} (no filtering)}
    \label{fig:L-shape-TypeI}
    \end{subfigure}
    \begin{subfigure}[t]{0.24\textwidth}  \includegraphics[width=\textwidth]{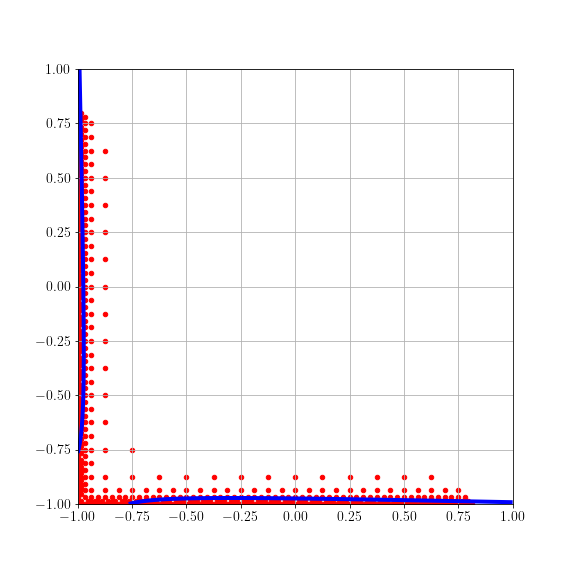}
    \caption{Algorithm \ref{alg:kde} (KDE filtering with $\gamma=0.6$)}
    \label{fig:L-shape-kde}
    \end{subfigure}
    \begin{subfigure}[t]{0.24\textwidth}  \includegraphics[width=\textwidth]{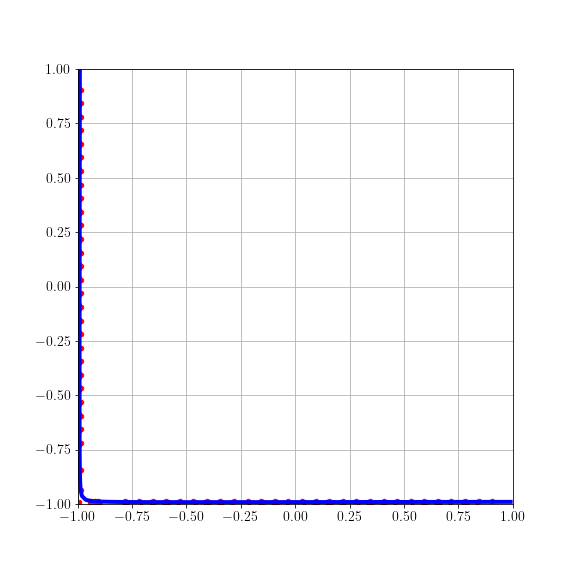}
    \caption{Algorithm \ref{alg:knn} ($k$NN-filtering with $\gamma=0.6$, $k=5$)}
    \label{fig:L-shape-knn}
    \end{subfigure}
      \begin{subfigure}[t]{0.24\textwidth}  \includegraphics[width=\textwidth]{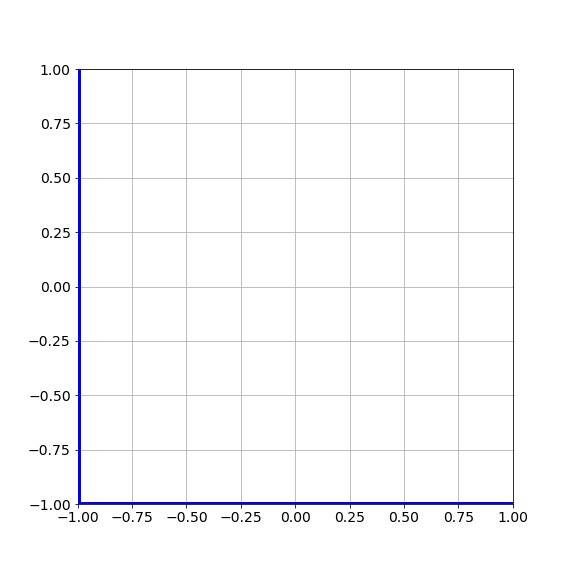}
    \caption{Exact singularity}
    \label{fig:L-shape-true}
    \end{subfigure}
    \caption{Example 3. Boundary singular layer}
    \label{fig:L-shape}
\end{figure}

\subsubsection{`X'-shaped singularity}
\label{ssec:X-shape}
    We consider an `X'-shaped singularity curve in $\Omega=(-1,1)^2$ from \cite{tang2003adaptive}.
    The singularity 
    $$\mathcal{S}=\{(x,y)\in\Omega: x^2-y^2=0\}.$$ 
    is associated with the piecewise constant function
    $$u(x,y) = \left\{
    \begin{array}{ll}
    1 & \text{ if } |x|\leq |y|,\\
    0 & \text{ otherwise.}
    \end{array}
    \right.$$
    The mesh is generated by Netgen/NGSolve \cite{schoberl1997netgen,ngsolve} and the training data (Type I) contains $1915$ nodes (shown in Figure \ref{fig:X-shape-raw}).
    Figure \ref{fig:X-shape} shows the learned singularity by (a) Algorithm \ref{alg:alg1}, (b) KDE-based Algorithm \ref{alg:alg2}, (c) $k$NN-based Algorithm \ref{alg:alg2}, as well as (d) the true singularity $\mathcal{S}$.
    The KDE filtering uses $\gamma=0.4$ and the filtered subset is shown in Figure \ref{fig:X-shape-kde} with $1725$ nodes.
    The $k$NN filtering uses $k=5$, $\gamma=0.4$, and the filtered subset is shown in Figure \ref{fig:X-shape-knn} with $937$ nodes.
    We see from Figure \ref{fig:X-shape} that filtering removes ``irrelevant'' points in the raw training data and helps to improve the accuracy of singularity detection. 
    Specifically, we compare in Table \ref{tab:X-shape-coef} the learned coefficients in $f(x,y)$ by the above algorithms to the true polynomial 
    $$F_*(x,y)=-\frac{1}{\sqrt{2}}y^2+\frac{1}{\sqrt{2}}x^2\approx -0.707y^2+0.707x^2$$ 
    corresponding to the exact singularity $\mathcal{S}=\{(x,y):x^2=y^2\}$.
    It can be seen from Table \ref{tab:X-shape-coef} that the $k$NN-based approach, with more concentrated training nodes after filtering, produces slightly better estimation than the KDE-based approach. Meanwhile, the baseline Algorithm \ref{alg:alg1} yields the least accurate approximation.

\begin{table}[htbp]
    \centering
    \begin{tabular}{c|c|c|c|c|c|c}
    \hline
       Basis  & $1$ & $y$ & $y^2$ & $x$ & $xy$ & $x^2$\\
    \hline
        Coefficients (Figure \ref{fig:X-shape-raw}) & 1.3E-2 & -1.1E-4 & -0.687 & -1.2E-4 & -2.1E-4 & 0.726\\
        Coefficients (Figure \ref{fig:X-shape-kde}) & 1.0E-2 & -3.9E-7 & -0.687 & -3.5E-7 & -3.3E-7 & 0.727\\
        Coefficients (Figure \ref{fig:X-shape-knn}) & 6.5E-3 & -2.9E-7 & -0.693 & -3.0E-7 & -4.8E-7 & 0.721\\
        Coefficients (exact) & 0 & 0 & -0.707 & 0 & 0 & 0.707\\
    \hline
    \end{tabular}
    \caption{`X'-shaped singularity: comparison of polynomial coefficients for curves in Figure \ref{fig:X-shape}}
    \label{tab:X-shape-coef}
\end{table}

    \begin{figure}
    \begin{subfigure}[t]{0.23\textwidth}
        \includegraphics[width=\textwidth]{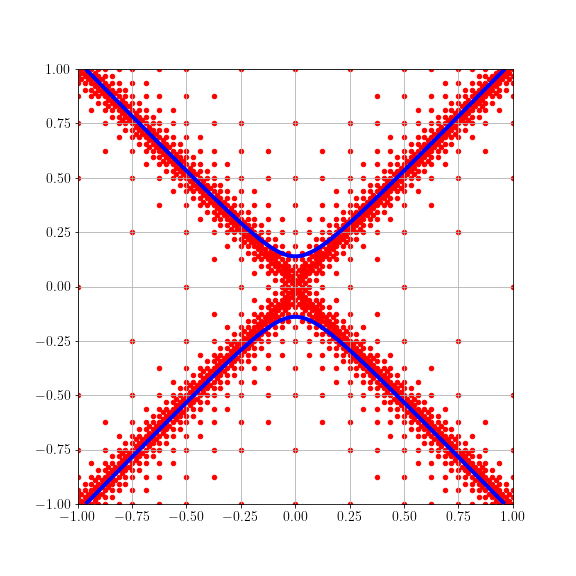}
        \caption{Algorithm \ref{alg:alg1}}
        \label{fig:X-shape-raw}
    \end{subfigure}
    ~
    \begin{subfigure}[t]{0.23\textwidth}
    \includegraphics[width=\textwidth]{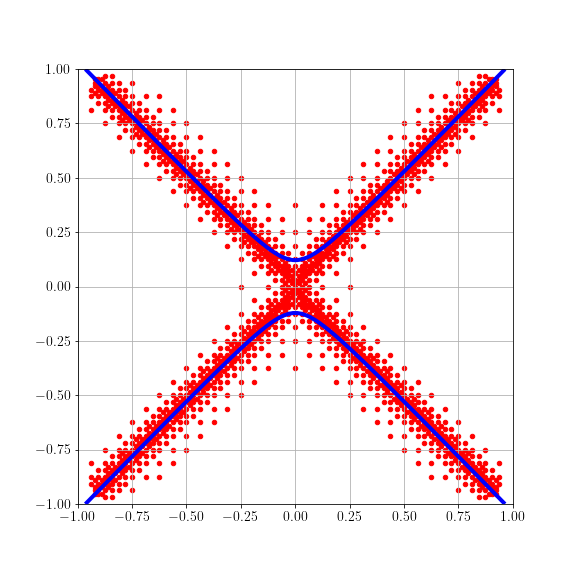}
        \caption{KDE-based Algorithm \ref{alg:alg2}}
        \label{fig:X-shape-kde}
    \end{subfigure}
    ~
    \begin{subfigure}[t]{0.23\textwidth}
        \includegraphics[width=\textwidth]{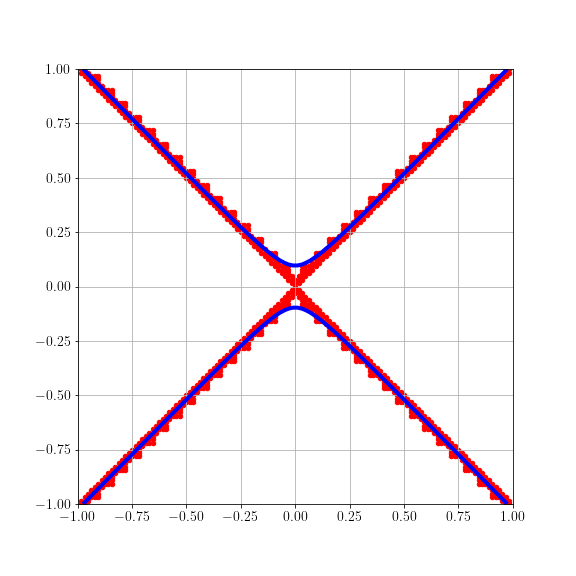}
        \caption{$k$NN-based Algorithm \ref{alg:alg2}}
        \label{fig:X-shape-knn}
    \end{subfigure}
    ~
    \begin{subfigure}[t]{0.23\textwidth}
\includegraphics[width=\textwidth]{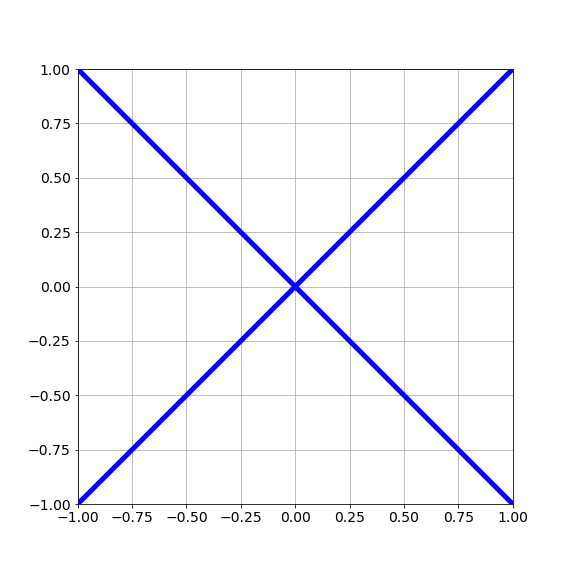}
        \caption{Exact singularity}
        \label{fig:Xshape-exact}
    \end{subfigure}
    \caption{Example 3. `X'-shaped singularity}
    \label{fig:X-shape}
    \end{figure}

\subsubsection{Concentric semicircles}
\label{subsec:vortical}
In this example, we consider the following singularly perturbed advection-diffusion problem in the domain $\Omega=(0,1)\times(-1,1)$ with a rotational advection field $\beta(x,y)=(-y,x)$, modeling a rotating flow. 
    \begin{align*}
    -\varepsilon \Delta u + \beta(x,y)\cdot \nabla u  & = 0 \;\; \text{ in } \Omega \quad \text{ with } \varepsilon = 10^{-4},\\
    u & = g_d \text{ on } \partial \Omega.
    \end{align*}
   The exact solution is chosen as
\begin{equation}
u(x,y)=\frac{1}{2}\tanh(\varepsilon^{-1} (x^2+y^2-0.5^2))-\frac{1}{2}\tanh(\varepsilon^{-1} (x^2+y^2-0.75^2)).
\end{equation}
    The solution $u(x,y)$ displays two singular layers in the form of two semicircles $x^2+y^2=0.5^2$ and $x^2+y^2=0.75^2$ in the domain $\Omega$.
    We generate two datasets (both Type I) using the Netgen/NGSolve software \cite{ngsolve}.
    The larger dataset $X_1$ contains 13253 nodes, and the smaller dataset $X_2$ contains 2685 nodes. 
    For each dataset, we apply the following three methods and compare their performance (a) Algorithm \ref{alg:alg1}, (b) KDE-based Algorithm \ref{alg:alg2}, (c) $k$NN-based Algorithm \ref{alg:alg2}.
    The results for $X_1$ and $X_2$ are shown in Figure \ref{fig:semicircle1} and Figure \ref{fig:semicircle2}, respectively.
    
    For $X_1$, it can be seen that from Figure \ref{fig:X1-rings-TypeI} that, even though $X_1$ is a relatively large-scale dataset, Algorithm \ref{alg:alg1} without filtering fails to capture the singularity accurately. In contrast, filtering-based Algorithm \ref{alg:alg2} yields much better results as shown in Figure \ref{fig:X1-rings-kde} and Figure \ref{fig:X1-rings-knn}.
    
    For $X_2$, we see from Figure \ref{fig:X2-rings-TypeI} that Algorithm \ref{alg:alg1} gives large errors in learning the singularity curve, mainly due to the small size of the training data. The KDE-based Algorithm \ref{alg:alg2} in Figure \ref{fig:X2-rings-kde} yields better estimation compared to Algorithm \ref{alg:alg1}, while the $k$NN-based Algorithm in Figure \ref{fig:X2-rings-knn} yields the best performance for detecting singularity.
    It should also be noticed that the $k$NN filtering used in each experiment yields the most concentrated nodes (with the smallest size of data). For example, the $k$NN-filtered subset in Figure \ref{fig:X1-rings-knn} contains only 3736 nodes, much smaller than the raw data in Figure \ref{fig:X1-rings-TypeI} with 13253 nodes and the KDE-filtered subset in Figure \ref{fig:X1-rings-kde} with 12856 nodes. We see that the local nature of $k$NN is more effective than KDE in reducing ``noise'' in the raw dataset.

    The singularity set $\mathcal{S}$ contains the roots of the quartic polynomial $(x^2+y^2-0.5^2)(x^2+y^2-0.75^2)$ inside $\Omega$. 
    The normalized polynomial such that the coefficient vector has unit norm is (coefficients truncated to 5 digits of accuracy)
$$F_*(x,y) \approx 0.05191-0.29990y^2+0.36910y^4-0.29990x^2+0.73821x^2y^2+ 0.36910x^4.$$
    Polynomial coefficients calculated by the algorithms listed in Figure \ref{fig:semicircle1} and Figure \ref{fig:semicircle2} are shown in Table \ref{tab:semicircle1} and Table \ref{tab:semicircle2}, respectively. Insignificant terms such as $x,y,xy,x^2y,xy^2,x^3,y^3,x^3y,xy^3$ are not shown, since the computed coefficients are all small (ranging from $10^{-7}$ to $10^{-3}$). It can be seen from Table \ref{tab:semicircle1} and Table \ref{tab:semicircle2} that the $k$NN-based Algorithm \ref{alg:alg2} yields the best performance compared to Algorithm \ref{alg:alg1} and KDE-based Algorithm \ref{alg:alg2}. For small-scale data $X_2$, we see from Table \ref{tab:semicircle2} that the baseline Algorithm \ref{alg:alg1} without filtering leads to completely wrong detection of the singularity (see also Figure \ref{fig:X1-rings-TypeI}). The filtering-based Algorithm \ref{alg:alg2}, regardless of the choice of the filtering method, gives a lot more reliable estimation of the singularity according to Table \ref{tab:semicircle2}.
    
\begin{table}[htbp]
    \centering
    \begin{tabular}{c|c|c|c|c|c|c}
    \hline
       Basis  & $1$ & $y^2$ & $y^4$ & $x^2$ & $x^2y^2$ & $x^4$\\
    \hline
        Coefficients (Figure \ref{fig:X1-rings-TypeI}) & 0.07438 & -0.34730 & 0.34807 & -0.34729 & 0.71481 & 0.34805\\
        Coefficients (Figure \ref{fig:X1-rings-kde}) & 0.05285 & -0.30115 & 0.36448 & -0.30113 & 0.74173 & 0.36445\\
        Coefficients (Figure \ref{fig:X1-rings-knn}) & 0.05184 & -0.29949 & 0.36844 & -0.29949 & 0.73920 & 0.36845\\
        Coefficients (exact) & 0.05191 & -0.29990 & 0.36910 & -0.29990 & 0.73821 & 0.36910\\
    \hline
    \end{tabular}
    \caption{Concentric-semicircle singularity: comparison of polynomial coefficients for curves in Figure \ref{fig:semicircle1} (large-scale data $X_1$)}
    \label{tab:semicircle1}
\end{table}

\begin{table}[htbp]
    \centering
    \begin{tabular}{c|c|c|c|c|c|c}
    \hline
       Basis  & $1$ & $y^2$ & $y^4$ & $x^2$ & $x^2y^2$ & $x^4$\\
    \hline
        Coefficients (Figure \ref{fig:X2-rings-TypeI}) & -0.11603 & 0.45933 & -0.41357 & 0.46038 & -0.46902 & -0.41543\\
        Coefficients (Figure \ref{fig:X2-rings-kde}) & 0.05923 & -0.30304 & 0.30721 & -0.30479 & 0.78795 & 0.31061\\
        Coefficients (Figure \ref{fig:X2-rings-knn}) & 0.05453 & -0.30345 & 0.35775 & -0.30359 & 0.74603 & 0.35809\\
        Coefficients (exact) & 0.05191 & -0.29990 & 0.36910 & -0.29990 & 0.73821 & 0.36910\\
    \hline
    \end{tabular}
    \caption{Concentric-semicircle singularity: comparison of polynomial coefficients for curves in Figure \ref{fig:semicircle2} (small-scale data $X_2$)}
    \label{tab:semicircle2}
\end{table}

\begin{figure}[h]
\begin{subfigure}[t]{0.32\textwidth}
\includegraphics[width=\textwidth]{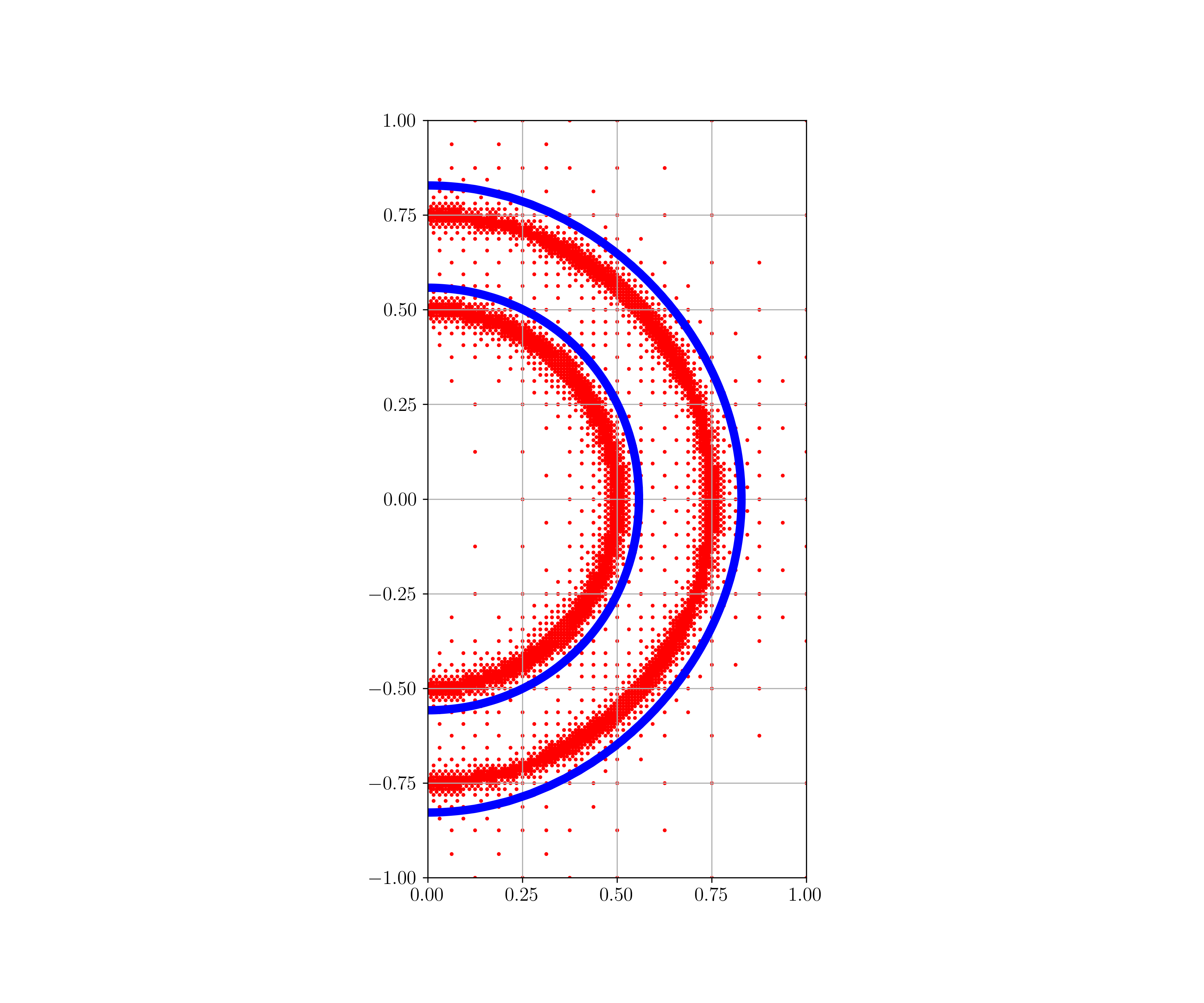}
\caption{Algorithm \ref{alg:alg1} for $X_1$}
\label{fig:X1-rings-TypeI}
\end{subfigure}
~
\begin{subfigure}[t]{0.32\textwidth}
\includegraphics[width=\textwidth]{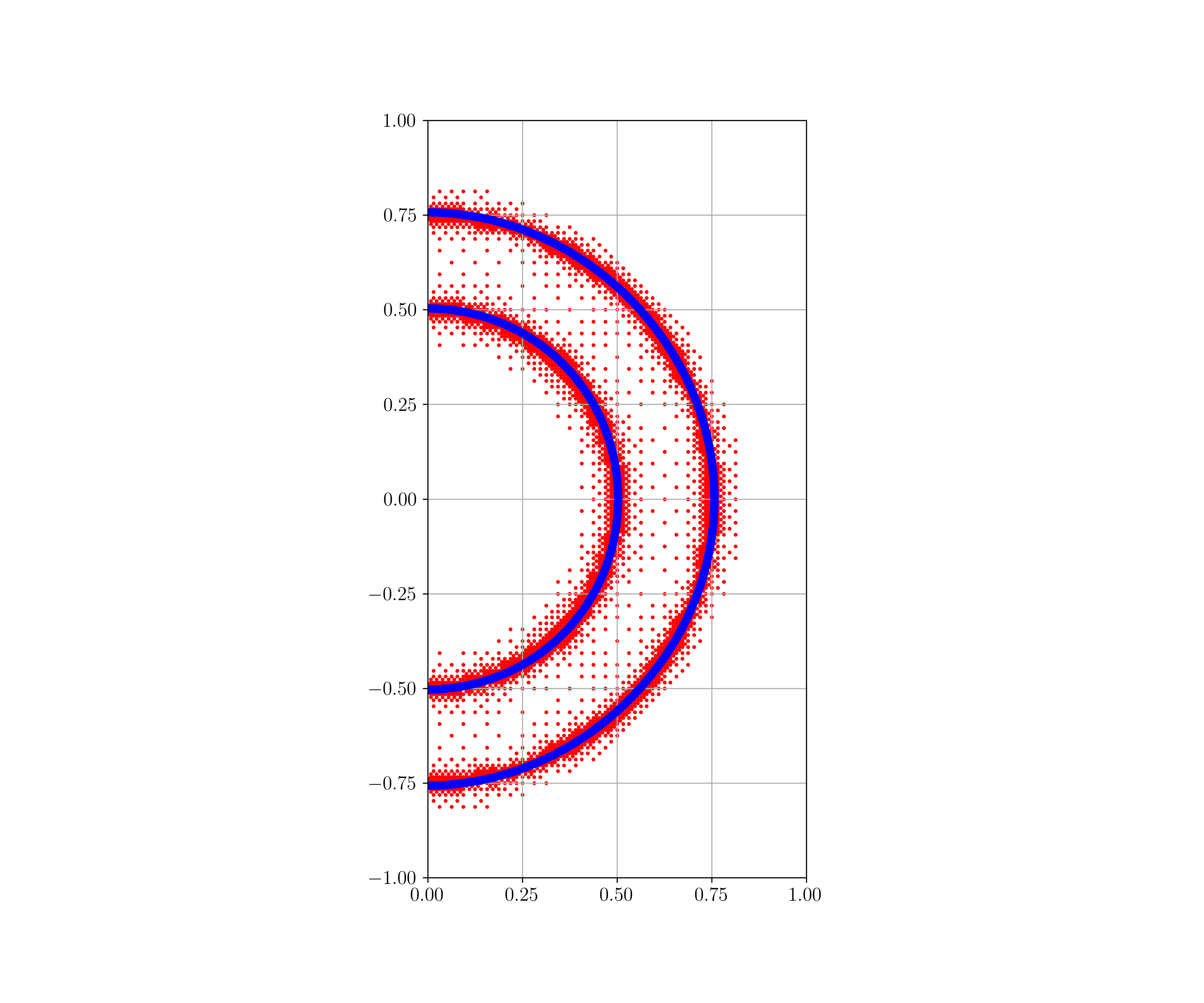}
\caption{KDE-based Algorithm \ref{alg:alg2} ($\gamma=0.6$)}
\label{fig:X1-rings-kde}
\end{subfigure}
~
\begin{subfigure}[t]{0.32\textwidth}
\includegraphics[width=\textwidth]{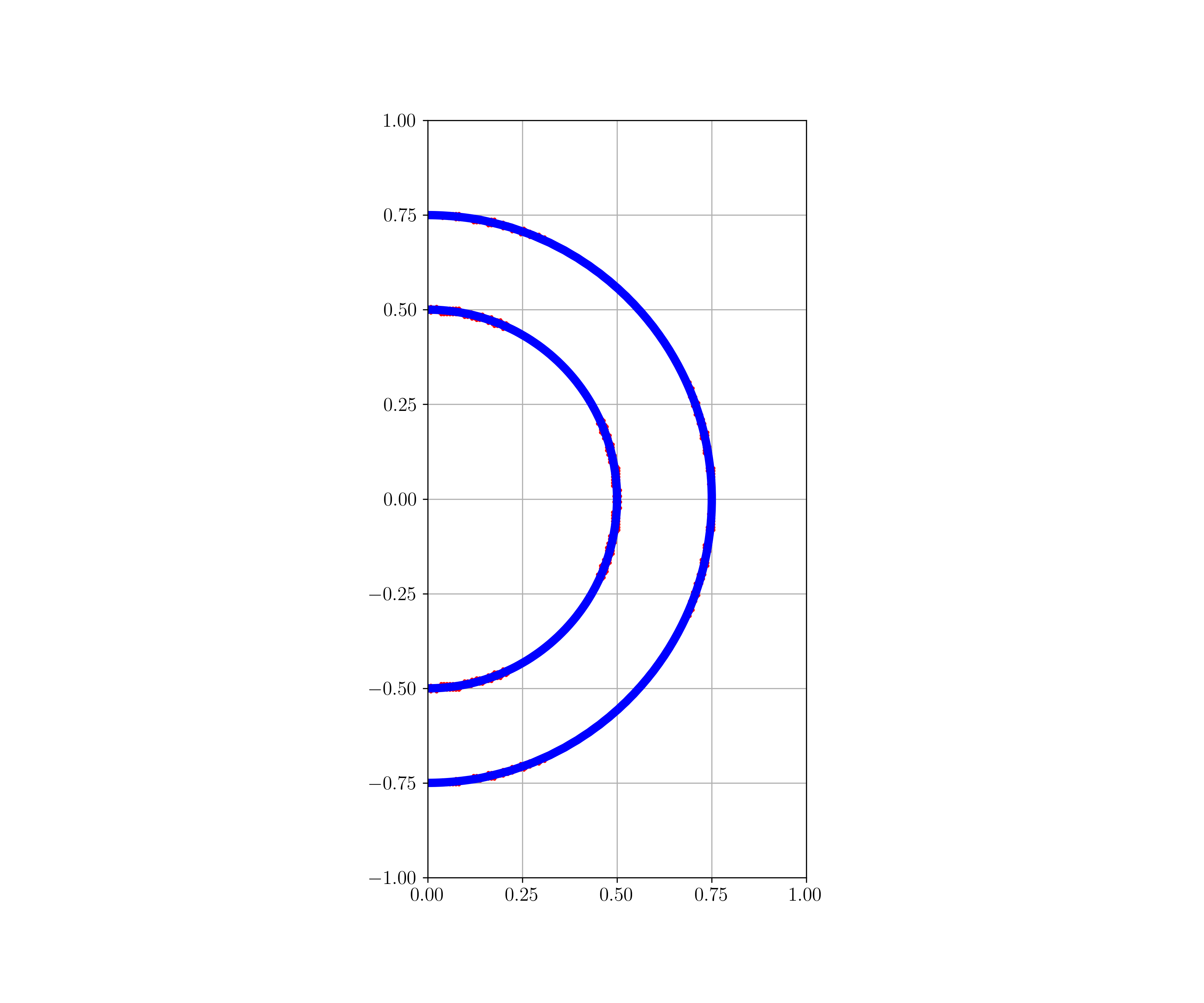}
\caption{$k$NN-based Algorithm \ref{alg:alg2} ($k=5$, $\gamma=0.6$)}
\label{fig:X1-rings-knn}
\end{subfigure}
\caption{Example 3. Concentric-semicircle singularity: large-scale raw data $X_1$}
\label{fig:semicircle1}
\end{figure}

\begin{figure}[h]
\begin{subfigure}[t]{0.32\textwidth}
\includegraphics[width=\textwidth]{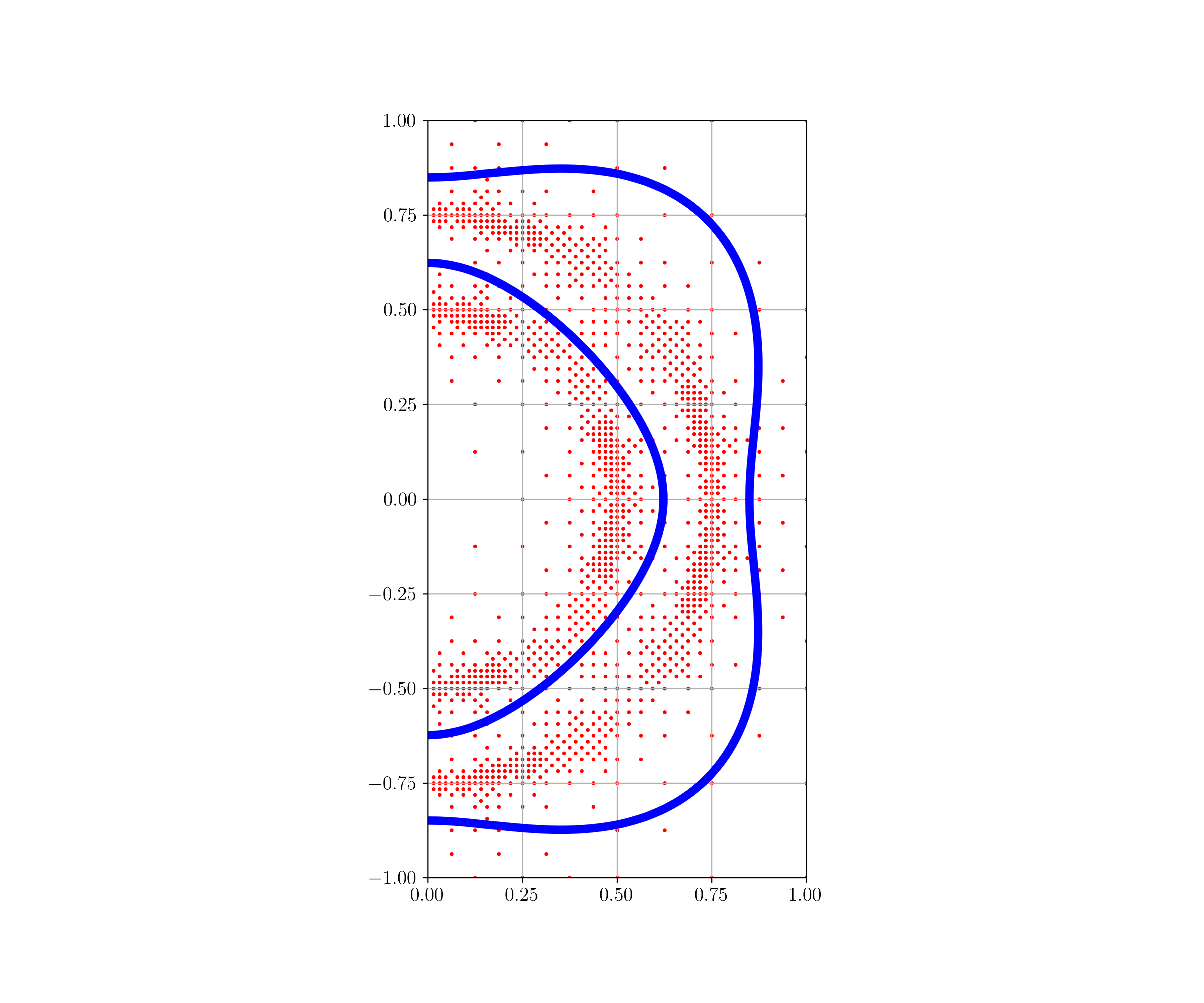}
\caption{Algorithm \ref{alg:alg1} for $X_2$}
\label{fig:X2-rings-TypeI}
\end{subfigure}
~
\begin{subfigure}[t]{0.32\textwidth}
\includegraphics[width=\textwidth]{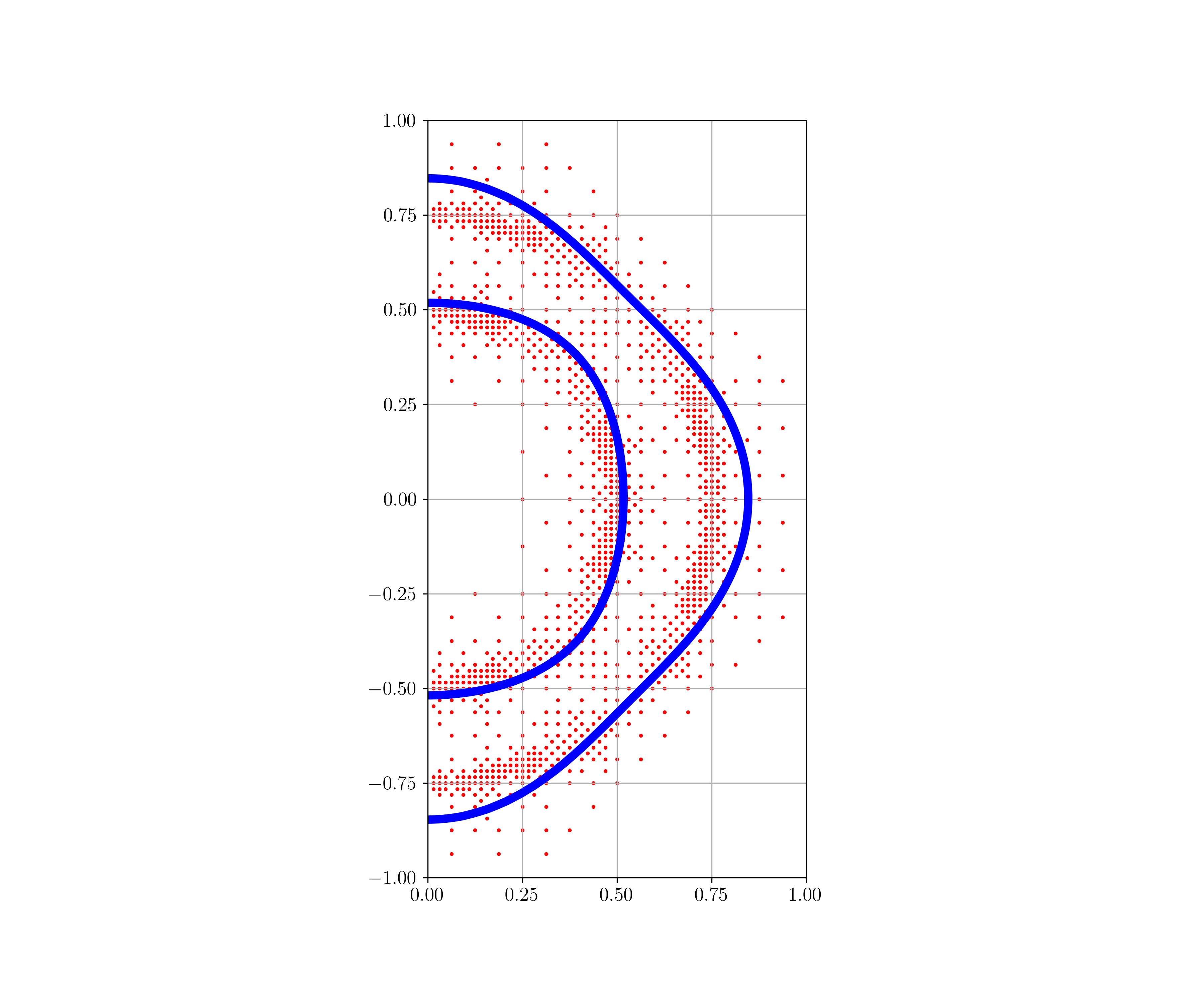}
\caption{KDE-based Algorithm \ref{alg:alg2} ($\gamma=0.2$)}
\label{fig:X2-rings-kde}
\end{subfigure}
~
\begin{subfigure}[t]{0.32\textwidth}
\includegraphics[width=\textwidth]{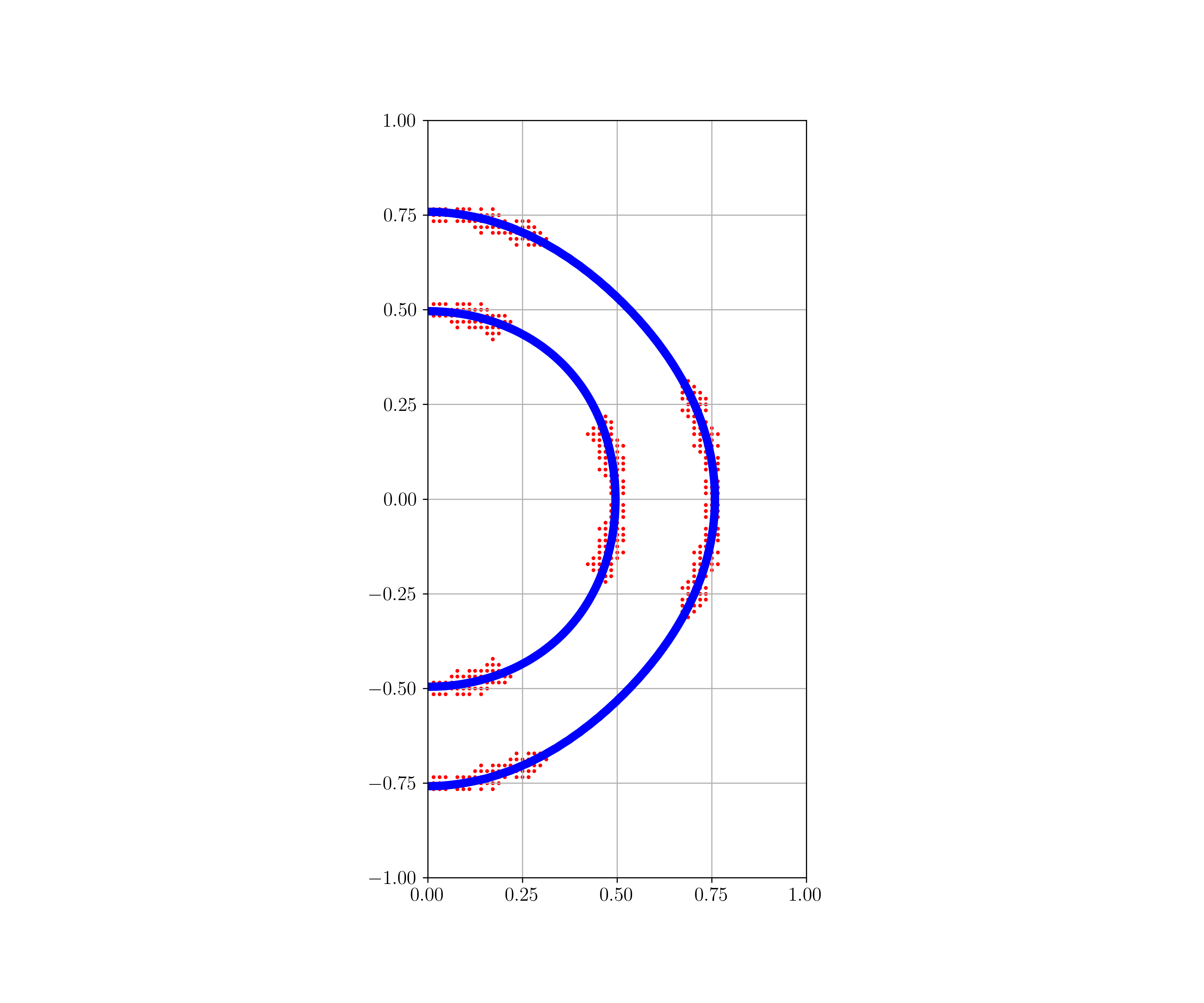}
\caption{$k$NN-based Algorithm \ref{alg:alg2} ($k=5$, $\gamma=0.6$)}
\label{fig:X2-rings-knn}
\end{subfigure}
\caption{Example 3. Concentric-semicircle singularity: small-scale raw data $X_2$}
\label{fig:semicircle2}
\end{figure}

\subsection{Example 4. Fourier basis for singularity detection}
In this example, we consider using the Fourier basis for the detection function \eqref{eq:f} to perform singularity detection. To this end, we represent training points in polar coordinates $(r,\theta)$ such that $x=r\cos\theta$ and $y=r\sin\theta$.
The detection function $f(r, \theta)$ in polar coordinates is chosen as
\begin{equation}
\label{eq:fpolar}
    f(r, \theta) =  \sum_{j=0}^{J} \sum_{m=0}^{M} a_{j,m} \, r^j \cos(m\theta) + b_{j,m}\, r^j \sin(m\theta),
\end{equation}
where $J,M$ are nonnegative integers and $n=m=0$ corresponds to the constant term $a_{0,0}$.
The goal is to determine the vector $\mathbf{c} := \{a_{0,0}, b_{0,0},\dots,a_{J,M},b_{J,M}\}$ containing all the unknown coefficients.
The loss function remains the same as in Theorem \ref{thm:MLE} except a change of the coordinate representation for each training point. The optimization problem (quadratic programming) can be solved using SLSQP as for the Cartesian case.
To evaluate the empirical performance, we consider the two datasets that have caused some difficulties for the polynomial detection function in previous examples: `X' shape and concentric semicircles.
We apply Algorithm \ref{alg:alg1} (no filtering) with the Fourier detection function in \eqref{eq:fpolar} to the raw data.

The result is shown in Figure \ref{fig:fpolar}.
We observe from Figure \ref{fig:Fourier-X-shape-raw} that the Fourier detection function in \eqref{eq:fpolar} produces an accurate estimation of the singularity for the `X' shape data, even better than all the algorithms (without or with filtering) illustrated in Figure \ref{fig:X-shape}. 
This implies that the Fourier basis is better at detecting this type of singularity than the polynomial basis.
Similarly, compared to Figure \ref{fig:X1-rings-TypeI} with polynomial basis, it is easy to see that Figure \ref{fig:Fourier-X1-rings-TypeI} with Fourier basis demonstrates substantially more accurate singularity detection for concentric semicircles data, without applying filtering to the data.

In Table \ref{tab:polar-X} and Table \ref{tab:polar-semicircles}, we compare the computed detection function to the true singularity function (with normalized coefficients) for each case.
\begin{itemize}
    \itemsep=0in
    \item For `X' shape, the singularity set $\mathcal{S}$ is given by 
$\mathcal{S}=\{(r,\theta):\theta=\pm\frac{\pi}{4}\},$
which is the root set of 
\begin{equation}
    \label{eq:F*polar-X}
    F_*(r,\theta)=r\cos(2\theta).
\end{equation}
    \item For concentric semicircles, $\mathcal{S}$ is the root set of $(r-0.5)(r-0.75)$, whose normalized version is 
    \begin{equation}
    \label{eq:F*polar-semicircles}
        F_*(r,\theta)\approx 0.608r^2-0.760r+0.228.
    \end{equation}
\end{itemize}
For Table \ref{tab:polar-semicircles}, the terms $$r\cos\theta,r\sin\theta,r\cos(2\theta),r\sin(2\theta),r^2\cos\theta,r^2\sin\theta,r^2\cos(2\theta),r^2\sin(2\theta)$$ are not shown since the error is too small (below 1E-5).
In parallel to the observation above, Table \ref{tab:polar-X} and Table \ref{tab:polar-semicircles} again demonstrate that the estimation accuracy by the Fourier basis is better than the polynomial basis for the case of `X' shape or concentric semicircles singularity in Example 3, both using the baseline Algorithm \ref{alg:alg1}.
From this example, we see that the proposed formulation provides sufficient flexibility to adjust the detection model in order to effectively handle different types of singularities.

\begin{figure}[htbp]
\begin{subfigure}[t]{0.44\textwidth}
\includegraphics[width=.8\textwidth]{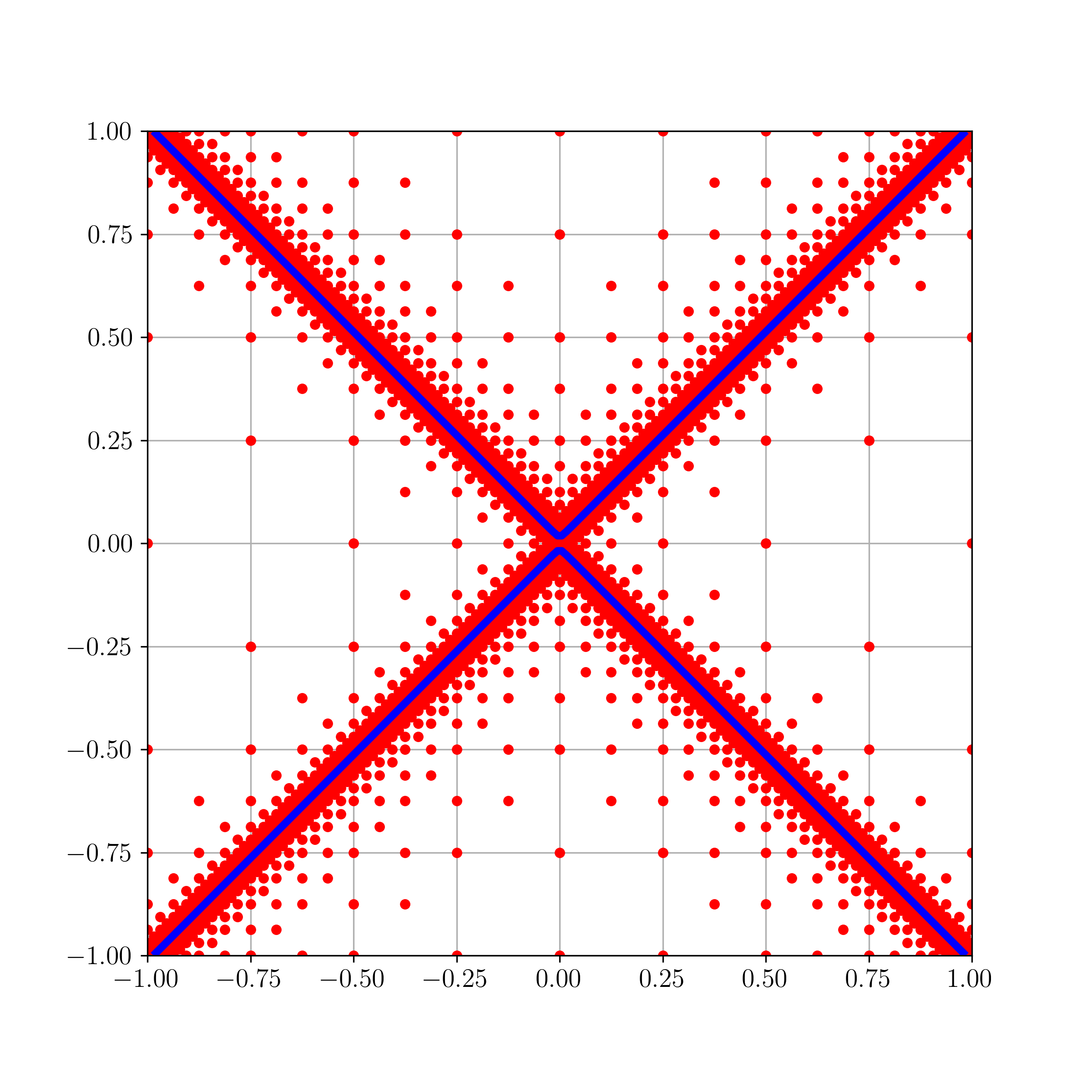}
\caption{Fourier basis \eqref{eq:fpolar} with $J=1,M=2$ gives better result than the polynomial basis in Section \ref{ssec:X-shape}}
\label{fig:Fourier-X-shape-raw}
\end{subfigure}
~
\begin{subfigure}[t]{0.44\textwidth}
    \includegraphics[width=.8\textwidth]{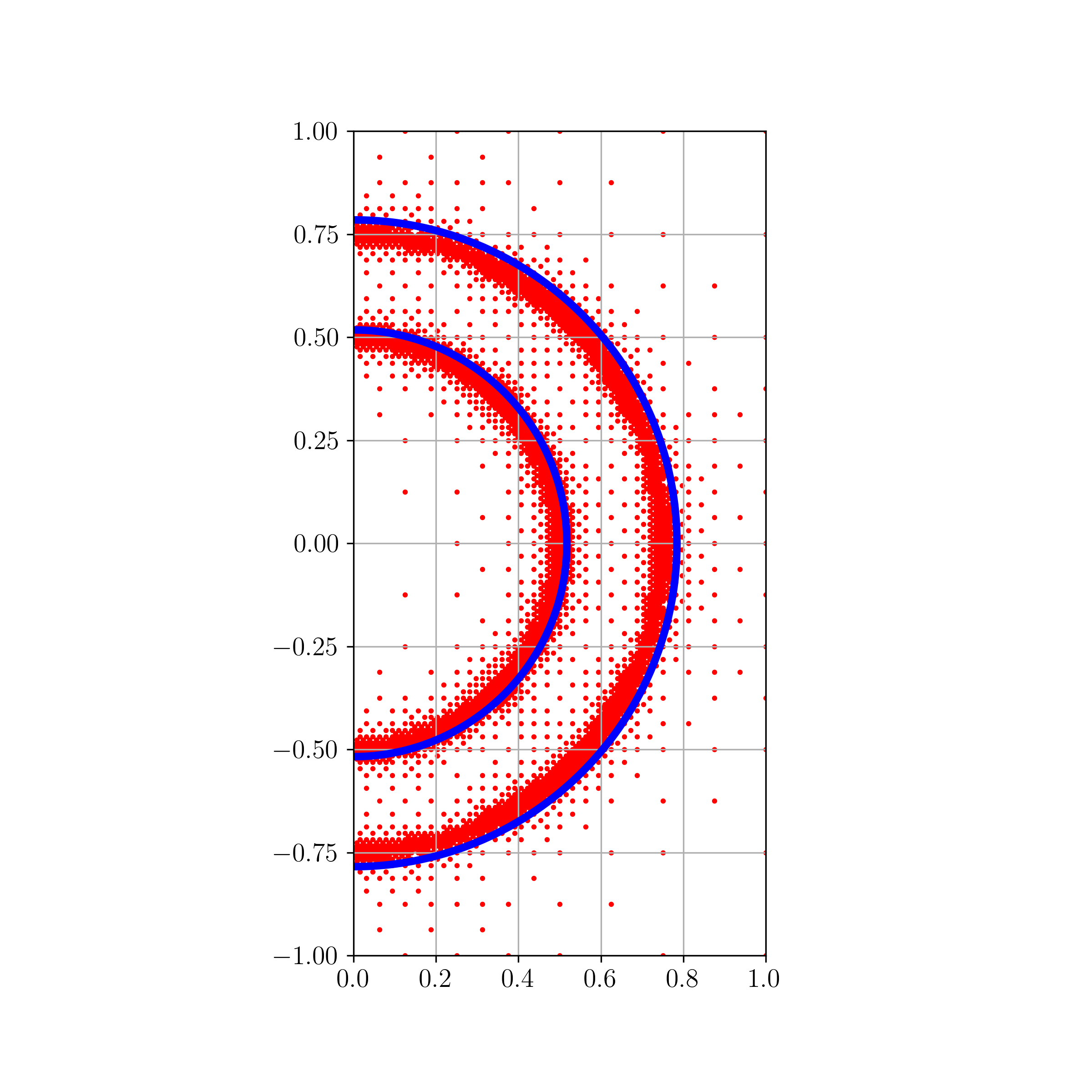}
    \caption{Fourier basis \eqref{eq:fpolar} with $J=M=2$ gives better result than the polynomial basis in Section \ref{subsec:vortical}}
    \label{fig:Fourier-X1-rings-TypeI}
\end{subfigure}
\caption{Singularity detection using Algorithm~\ref{alg:alg1} (no filtering) with Fourier basis}
\label{fig:fpolar}
\end{figure}

\begin{table}[htbp]
    \centering
    \begin{tabular}{c|c|c|c|c|c|c}
    \hline
       Basis  & $1$ & $r$ & $r\cos(\theta)$ & $r\sin(\theta)$ & $r\cos(2\theta)$ & $r\sin(2\theta)$   \\
    \hline
        Coefficients (Figure \ref{fig:Fourier-X-shape-raw}) & 0.0162 &  0.0030  & -5E-6 & -5E-6 & 0.9999 & -3E-6  \\
        Coefficients (exact)
          & 0 & 0 & 0 & 0 & 1 & 0\\
    \hline
    \end{tabular}
     \caption{`X'-shaped singularity:  comparison of the computed Fourier detection function and the exact $F_*$ in \eqref{eq:F*polar-X}}
     \label{tab:polar-X}
\end{table}

\begin{table}[htbp]
    \centering
    \begin{tabular}{c|c|c|c}
    \hline
       Basis  & $1$ & $r$ & $r^2$ \\
    \hline
        Coefficients (Figure \ref{fig:Fourier-X1-rings-TypeI}) & 0.240 &   -0.770  & 0.591\\
        Coefficients (exact) &0.228& -0.760 &  0.608 \\
    \hline
    \end{tabular}
    \caption{Concentric-semicircle singularity: comparison of the computed Fourier detection function and the exact $F_*$ in \eqref{eq:F*polar-semicircles}}
    \label{tab:polar-semicircles}
\end{table}

\section{Conclusion}
Resolving the potentially unknown singularity in the solution is critical for the efficient numerical discretization of PDEs.
In this paper, we considered detecting the unknown singularity from limited mesh data. We proposed a flexible self-supervised learning framework for effective singularity detection, with filtering as the pretext task.
Two filtering methods are investigated, kernel density estimation and $k$-nearest neighbors. 
Extensive experiments are presented to demonstrate the dramatically improved accuracy for estimating the unknown singularity compared to learning without filtering, especially for small-scale data.
Moreover, empirical results indicate that a smaller filtered subset with more concentrated nodes will yield more accurate singularity detection. Computationally, the filtering procedure will also help significantly reduce the training cost when the filtered subset is much smaller than the raw data. Though the examples presented are in two dimensions, extension to three dimensions is straightforward.
The promising results presented in this paper motivate a number of interesting directions to be explored in the future, including investigating the impact of hyperparameters (for example, the threshold parameter $\gamma$ in Algorithm \ref{alg:alg2}, the bandwidth parameter $h$ in KDE), studying more advanced ways to generate pseudo-labels, integrating the approach into adaptive mesh generation, testing other representations for characterizing the singularity, etc.
Moreover, the accurate singularity detection hinges on mesh data from AMR correctly refined towards the singularity in the \emph{pre-asymptotic} regime (at the coarse mesh level). The high quality mesh may not be guaranteed unless a \emph{robust} a posteriori error estimator is used in AMR. As a result, theoretical emphasis will be placed on the development of \emph{robust} a posteriori error estimators in order to generate high-quality training datasets for successful singularity detection.

\bibliography{refs}

\begin{thebibliography}{10}

\bibitem{ngsolve}
{Netgen/NGSolve}.
\newblock \url{https://ngsolve.org/}.

\bibitem{ainsworth1999reaction}
M.~Ainsworth and I.~Babu{\v s}ka.
\newblock Reliable and robust a posteriori error estimation for singularly
  perturbed reaction-diffusion problems.
\newblock {\em SIAM Journal on Numerical Analysis}, 36(2):331--353, 1999.

\bibitem{ainsworthoden2000book}
M.~Ainsworth and J.T. Oden.
\newblock {\em A Posteriori Error Estimation in Finite Element Analysis},
  volume~37.
\newblock John Wiley \& Sons, 2000.

\bibitem{ainsworth2011confusion}
M.~Ainsworth and T.~Vejchodsk{\'y}.
\newblock Fully computable robust a posteriori error bounds for singularly
  perturbed reaction--diffusion problems.
\newblock {\em Numerische Mathematik}, 119(2):219--243, 2011.

\bibitem{ainsworth2021GNN}
Mark Ainsworth and Justin Dong.
\newblock Galerkin neural networks: A framework for approximating variational
  equations with error control.
\newblock {\em SIAM Journal on Scientific Computing}, 43(4):A2474--A2501, 2021.

\bibitem{knn1998optimal}
S.~Arya, D.~M. Mount, N.~S. Netanyahu, R.~Silverman, and A.~Y. Wu.
\newblock An optimal algorithm for approximate nearest neighbor searching fixed
  dimensions.
\newblock {\em Journal of the ACM (JACM)}, 45(6):891--923, 1998.

\bibitem{baevski2022data2vec}
Alexei Baevski, Wei-Ning Hsu, Qiantong Xu, Arun Babu, Jiatao Gu, and Michael
  Auli.
\newblock Data2vec: A general framework for self-supervised learning in speech,
  vision and language.
\newblock In {\em International Conference on Machine Learning}, pages
  1298--1312. PMLR, 2022.

\bibitem{baevski2020wav2vec}
Alexei Baevski, Yuhao Zhou, Abdelrahman Mohamed, and Michael Auli.
\newblock wav2vec 2.0: A framework for self-supervised learning of speech
  representations.
\newblock {\em Advances in neural information processing systems},
  33:12449--12460, 2020.

\bibitem{SSLcookbook}
Randall Balestriero, Mark Ibrahim, Vlad Sobal, Ari~S. Morcos, Shashank Shekhar,
  Tom Goldstein, Florian Bordes, Adrien Bardes, Gr{\'e}goire Mialon, Yuandong
  Tian, Avi Schwarzschild, Andrew~Gordon Wilson, Jonas Geiping, Quentin
  Garrido, Pierre Fernandez, Amir Bar, Hamed Pirsiavash, Yann LeCun, and Micah
  Goldblum.
\newblock A cookbook of self-supervised learning.
\newblock {\em ArXiv}, abs/2304.12210, 2023.

\bibitem{bayona2017role}
Victor Bayona, Natasha Flyer, Bengt Fornberg, and Gregory~A Barnett.
\newblock {On the role of polynomials in RBF-FD approximations: II. Numerical
  solution of elliptic PDEs}.
\newblock {\em Journal of Computational Physics}, 332:257--273, 2017.

\bibitem{bernardi2000}
C.~Bernardi and R.~Verf{\"u}rth.
\newblock Adaptive finite element methods for elliptic equations with
  non-smooth coefficients.
\newblock {\em Numerische Mathematik}, 85(4):579--608, 2000.

\bibitem{annoy}
E.~Bernhardsson.
\newblock {ANNOY C++ library}.
\newblock \url{https://github.com/spotify/annoy}.

\bibitem{bishop2006book}
Christopher~M. Bishop.
\newblock {\em Pattern Recognition and Machine Learning}.
\newblock Springer-Verlag, Berlin, Heidelberg, 2006.

\bibitem{convex2004book}
Stephen~P Boyd and Lieven Vandenberghe.
\newblock {\em Convex optimization}.
\newblock Cambridge university press, 2004.

\bibitem{brink2021neural}
Adam~R Brink, David~A Najera-Flores, and Cari Martinez.
\newblock The neural network collocation method for solving partial
  differential equations.
\newblock {\em Neural Computing and Applications}, 33(11):5591--5608, 2021.

\bibitem{localL2}
D.~Cai and Z.~Cai.
\newblock A hybrid a posteriori error estimator for conforming finite element
  approximations.
\newblock {\em Computer Methods in Applied Mechanics and Engineering}, 339:320
  -- 340, 2018.

\bibitem{smash}
D.~Cai, E.~Chow, L.~Erlandson, Y.~Saad, and Y.~Xi.
\newblock {SMASH}: Structured matrix approximation by separation and hierarchy.
\newblock {\em Numerical Linear Algebra with Applications}, 25(6):e2204, 2018.

\bibitem{HiDR}
D.~Cai, H.~Huang, E.~Chow, and Y.~Xi.
\newblock Data-driven construction of hierarchical matrices with nested bases.
\newblock {\em SIAM Journal on Scientific Computing}, 46(2):S24--S50, 2024.

\bibitem{thesisDifengCai}
Difeng Cai.
\newblock {\em Robust and Explicit a posteriori Error Estimation Techniques in
  Adaptive Finite Element Method}.
\newblock PhD thesis, Purdue University, 2019.

\bibitem{confusion}
Difeng Cai and Zhiqiang Cai.
\newblock Hybrid a posteriori error estimators for conforming finite element
  approximations to stationary convection-diffusion-reaction equations.
\newblock {\em Numerische Mathematik}, 157:477--504, 2025.

\bibitem{GPpost}
Difeng Cai, Edmond Chow, and Yuanzhe Xi.
\newblock Posterior covariance structures in gaussian processes.
\newblock {\em SIAM Journal on Matrix Analysis and Applications}, to appear.

\bibitem{caizhang2009}
Z.~Cai and S.~Zhang.
\newblock Recovery-based error estimator for interface problems: Conforming
  linear elements.
\newblock {\em SIAM Journal on Numerical Analysis}, 47(3):2132--2156, 2009.

\bibitem{caizhang2010}
Zhiqiang Cai and Shun Zhang.
\newblock Flux recovery and a posteriori error estimators: conforming elements
  for scalar elliptic equations.
\newblock {\em SIAM journal on numerical analysis}, 48(2):578--602, 2010.

\bibitem{cavoretto2021adaptive}
Roberto Cavoretto.
\newblock Adaptive radial basis function partition of unity interpolation: A
  bivariate algorithm for unstructured data.
\newblock {\em Journal of Scientific Computing}, 87(2):41, 2021.

\bibitem{davydov2023stencil}
O.~Davydov, D.~T. Oanh, and N.~M. Tuong.
\newblock Improved stencil selection for meshless finite difference methods in
  3d.
\newblock {\em Journal of Computational and Applied Mathematics}, 425:115031,
  2023.

\bibitem{deng2011density}
H.~Deng and H.~Wickham.
\newblock Density estimation in r.
\newblock {\em Electronic publication}, 2011.

\bibitem{BERT}
Jacob Devlin, Ming-Wei Chang, Kenton Lee, and Kristina Toutanova.
\newblock {BERT}: Pre-training of deep bidirectional transformers for language
  understanding.
\newblock In Jill Burstein, Christy Doran, and Thamar Solorio, editors, {\em
  Proceedings of the 2019 Conference of the North {A}merican Chapter of the
  Association for Computational Linguistics: Human Language Technologies,
  Volume 1 (Long and Short Papers)}, pages 4171--4186, Minneapolis, Minnesota,
  June 2019. Association for Computational Linguistics.

\bibitem{NNdescent2011}
Wei Dong, Charikar Moses, and Kai Li.
\newblock Efficient k-nearest neighbor graph construction for generic
  similarity measures.
\newblock In {\em Proceedings of the 20th international conference on World
  wide web}, pages 577--586, 2011.

\bibitem{ViT2021}
A.~Dosovitskiy, L.~Beyer, A.~Kolesnikov, D.~Weissenborn, X.~Zhai,
  T.~Unterthiner, M.~Dehghani, M.~Minderer, G.~Heigold, S.~Gelly, J.~Uszkoreit,
  and N.~Houlsby.
\newblock An image is worth 16x16 words: Transformers for image recognition at
  scale.
\newblock In {\em 9th International Conference on Learning Representations,
  {ICLR} 2021, Virtual Event, Austria, May 3-7, 2021}, 2021.

\bibitem{garmanjani2024adaptive}
G~Garmanjani, Mohsen Esmaeilbeigi, and Roberto Cavoretto.
\newblock {Adaptive residual refinement in an RBF finite difference scheme for
  2D time-dependent problems}.
\newblock {\em Computational and Applied Mathematics}, 43(1):39, 2024.

\bibitem{cavoretto2024RBF}
G~Garmanjani, Mohsen Esmaeilbeigi, and Roberto Cavoretto.
\newblock Adaptive residual refinement in an rbf finite difference scheme for
  2d time-dependent problems.
\newblock {\em Computational and Applied Mathematics}, 43(1):39, 2024.

\bibitem{SSL2022fair}
P.~Goyal, Q.~Duval, I.~Seessel, M.~Caron, I.~Misra, L.~Sagun, A.~Joulin, and
  P.~Bojanowski.
\newblock Vision models are more robust and fair when pretrained on uncurated
  images without supervision.
\newblock {\em arXiv preprint arXiv:2202.08360}, 2022.

\bibitem{gray2000nbody}
A.~Gray and A.~Moore.
\newblock N-body problems in statistical learning.
\newblock {\em Advances in neural information processing systems}, 13, 2000.

\bibitem{gray2003nonparametric}
A.~Gray and A.~Moore.
\newblock Nonparametric density estimation: Toward computational tractability.
\newblock In {\em Proceedings of the 2003 SIAM International Conference on Data
  Mining}, pages 203--211. SIAM, 2003.

\bibitem{fastgauss1991}
L.~Greengard and J.~Strain.
\newblock The fast gauss transform.
\newblock {\em SIAM J. Sci. Statist. Comput.}, 12(1):79--94, 1991.

\bibitem{grill2020bootstrap}
Jean-Bastien Grill, Florian Strub, Florent Altch{\'e}, Corentin Tallec, Pierre
  Richemond, Elena Buchatskaya, Carl Doersch, Bernardo Avila~Pires, Zhaohan
  Guo, Mohammad Gheshlaghi~Azar, et~al.
\newblock Bootstrap your own latent-a new approach to self-supervised learning.
\newblock {\em Advances in neural information processing systems},
  33:21271--21284, 2020.

\bibitem{hack2015book}
W.~Hackbusch.
\newblock {\em Hierarchical Matrices: Algorithms and Analysis}.
\newblock Springer Series in Computational Mathematics. Springer Berlin
  Heidelberg, 2015.

\bibitem{TetGen}
Si~Hang.
\newblock Tetgen, a delaunay-based quality tetrahedral mesh generator.
\newblock {\em ACM Trans. on Mathematical Software}, 41(2):11, 2015.

\bibitem{masked2022vision}
K.~He, X.~Chen, S.~Xie, Y.~Li, P.~Dollar, and R.~Girshick.
\newblock Masked autoencoders are scalable vision learners.
\newblock In {\em Proceedings - 2022 IEEE/CVF Conference on Computer Vision and
  Pattern Recognition, CVPR 2022}, Proceedings of the IEEE Computer Society
  Conference on Computer Vision and Pattern Recognition, pages 15979--15988,
  2022.

\bibitem{he2022masked}
Kaiming He, Xinlei Chen, Saining Xie, Yanghao Li, Piotr Doll{\'a}r, and Ross
  Girshick.
\newblock Masked autoencoders are scalable vision learners.
\newblock In {\em Proceedings of the IEEE/CVF conference on computer vision and
  pattern recognition}, pages 16000--16009, 2022.

\bibitem{SSL2019robust}
D.~Hendrycks, M.~Mazeika, S.~Kadavath, and D.~Song.
\newblock Using self-supervised learning can improve model robustness and
  uncertainty.
\newblock {\em Advances in neural information processing systems}, 32, 2019.

\bibitem{VAE2014}
D.~P. Kingma and M.~Welling.
\newblock {Auto-Encoding Variational Bayes}.
\newblock In {\em 2nd International Conference on Learning Representations,
  {ICLR} 2014, Banff, AB, Canada, April 14-16, 2014}, 2014.

\bibitem{kraft1988software}
Dieter Kraft.
\newblock A software package for sequential quadratic programming.
\newblock {\em Forschungsbericht- Deutsche Forschungs- und Versuchsanstalt fur
  Luft- und Raumfahrt}, 1988.

\bibitem{nando2005fast}
Dustin Lang, Mike Klaas, and Nando de~Freitas.
\newblock Empirical testing of fast kernel density estimation algorithms.
\newblock 2005.

\bibitem{FNO}
Z.~Li, N.~B. Kovachki, K.~Azizzadenesheli, B.~Liu, K.~Bhattacharya, A.~M.
  Stuart, and A.~Anandkumar.
\newblock Fourier neural operator for parametric partial differential
  equations.
\newblock In {\em 9th International Conference on Learning Representations,
  {ICLR}}, 2021.

\bibitem{li2024stable}
Zhilin Li, Kejia Pan, and Juan Ruiz-{\'A}lvarez.
\newblock {Stable high order FD methods for interface and internal layer
  problems based on non-matching grids}.
\newblock {\em Numerical Algorithms}, 96(4):1647--1674, 2024.

\bibitem{randomFields2011}
Finn Lindgren, H{\aa}vard Rue, and Johan Lindstr{\"o}m.
\newblock An explicit link between gaussian fields and gaussian markov random
  fields: The stochastic partial differential equation approach.
\newblock {\em Journal of the Royal Statistical Society Series B: Statistical
  Methodology}, 73(4):423--498, 08 2011.

\bibitem{chengdi2025}
Chengdi Ma, Jizu Huang, Hao Luo, and Chao Yang.
\newblock {CPAFT: A consistent parallel advancing front technique for
  unstructured triangular/tetrahedral mesh generation}.
\newblock {\em Computer Physics Communications}, 310:109535, 2025.

\bibitem{knn2018app}
Y.~A. Malkov and D.~A. Yashunin.
\newblock Efficient and robust approximate nearest neighbor search using
  hierarchical navigable small world graphs.
\newblock {\em IEEE transactions on pattern analysis and machine intelligence},
  42(4):824--836, 2018.

\bibitem{mishra2024artificial}
Abhishek Mishra, Cosmin Anitescu, Pattabhi~Ramaiah Budarapu, Sundararajan
  Natarajan, Pandu~Ranga Vundavilli, and Timon Rabczuk.
\newblock An artificial neural network based deep collocation method for the
  solution of transient linear and nonlinear partial differential equations.
\newblock {\em Frontiers of Structural and Civil Engineering},
  18(8):1296--1310, 2024.

\bibitem{mitchell2016}
W.~F. Mitchell.
\newblock 30 years of newest vertex bisection.
\newblock In {\em AIP Conference Proceedings}, volume 1738. AIP Publishing,
  2016.

\bibitem{AMR2Dtests}
William~F. Mitchell.
\newblock A collection of 2d elliptic problems for testing adaptive grid
  refinement algorithms.
\newblock {\em Applied Mathematics and Computation}, 220:350--364, 2013.

\bibitem{knn2009app}
M.~Muja and D.~G. Lowe.
\newblock Fast approximate nearest neighbors with automatic algorithm
  configuration.
\newblock In {\em International Conference on Computer Vision Theory and
  Application VISSAPP'09)}, pages 331--340. INSTICC Press, 2009.

\bibitem{oanh2017adaptive}
Dang~Thi Oanh, Oleg Davydov, and Hoang~Xuan Phu.
\newblock {Adaptive RBF-FD method for elliptic problems with point
  singularities in 2D}.
\newblock {\em Applied Mathematics and Computation}, 313:474--497, 2017.

\bibitem{panayot2017random}
Sarah Osborn, Panayot~S Vassilevski, and Umberto Villa.
\newblock A multilevel, hierarchical sampling technique for spatially
  correlated random fields.
\newblock {\em SIAM Journal on Scientific Computing}, 39(5):S543--S562, 2017.

\bibitem{kde1962}
E.~Parzen.
\newblock On estimation of a probability density function and mode.
\newblock {\em The annals of mathematical statistics}, 33(3):1065--1076, 1962.

\bibitem{AMR2016comsol}
T~Preney, P~Namy, and JD~Wheeler.
\newblock Adaptive mesh refinement: Quantitative computation of a rising bubble
  using comsol multiphysics{\textregistered}.
\newblock In {\em COMSOL Conf}, 2016.

\bibitem{transfer2020}
C.~Raffel, N.~Shazeer, A.~Roberts, K.~Lee, S.~Narang, M.~Matena, Y.~Zhou,
  W.~Li, and P.~J. Liu.
\newblock Exploring the limits of transfer learning with a unified text-to-text
  transformer.
\newblock {\em Journal of Machine Learning Research}, 21(140):1--67, 2020.

\bibitem{raissi2017physics}
Maziar Raissi, Paris Perdikaris, and George~Em Karniadakis.
\newblock Physics informed deep learning (part i): Data-driven solutions of
  nonlinear partial differential equations.
\newblock {\em arXiv preprint arXiv:1711.10561}, 2017.

\bibitem{kde1956}
M.~Rosenblatt.
\newblock {Remarks on Some Nonparametric Estimates of a Density Function}.
\newblock {\em The Annals of Mathematical Statistics}, 27(3):832 -- 837, 1956.

\bibitem{kdtree1984}
H.~Samet.
\newblock The quadtree and related hierarchical data structures.
\newblock {\em ACM Computing Surveys (CSUR)}, 16(2):187--260, 1984.

\bibitem{schoberl1997netgen}
Joachim Sch{\"o}berl.
\newblock Netgen an advancing front 2d/3d-mesh generator based on abstract
  rules.
\newblock {\em Computing and visualization in science}, 1(1):41--52, 1997.

\bibitem{purdue1972}
E.~G. Sewell.
\newblock {\em Automatic generation of triangulations for piecewise polynomial
  approximation}.
\newblock PhD thesis, Purdue University, West Lafayette, IN, 1972.

\bibitem{silverman86density}
B.~W. Silverman.
\newblock {\em Density Estimation for Statistics and Data Analysis}.
\newblock Chapman \& Hall, London, 1986.

\bibitem{silverman2018density}
Bernard~W Silverman.
\newblock {\em Density estimation for statistics and data analysis}.
\newblock Routledge, 2018.

\bibitem{slak2019rbf}
J.~Slak and G.~Kosec.
\newblock Adaptive radial basis function--generated finite differences method
  for contact problems.
\newblock {\em International Journal for Numerical Methods in Engineering},
  119(7):661--686, 2019.

\bibitem{stynes2005confusion}
M.~Stynes.
\newblock Steady-state convection-diffusion problems.
\newblock {\em Acta Numerica}, 14:445--508, 2005.

\bibitem{tang2003adaptive}
Huazhong Tang and Tao Tang.
\newblock Adaptive mesh methods for one- and two-dimensional hyperbolic
  conservation laws.
\newblock {\em SIAM Journal on Numerical Analysis}, 41(2):487--515, 2003.

\bibitem{KDEvariable1992}
George~R Terrell and David~W Scott.
\newblock Variable kernel density estimation.
\newblock {\em The Annals of Statistics}, pages 1236--1265, 1992.

\bibitem{verf1998reaction}
R.~Verf{\"u}rth.
\newblock Robust a posteriori error estimators for a singularly perturbed
  reaction-diffusion equation.
\newblock {\em Numerische Mathematik}, 78(3):479--493, 1998.

\bibitem{verf2005confusion}
R.~Verf{\"u}rth.
\newblock Robust a posteriori error estimates for stationary
  convection-diffusion equations.
\newblock {\em SIAM Journal on Numerical Analysis}, 43(4):1766--1782, 2005.

\bibitem{verf2013book}
R.~Verf{\"u}rth.
\newblock {\em A Posteriori Error Estimation Techniques for Finite Element
  Methods}.
\newblock Numerical Mathematics and Scientific Computation. OUP Oxford, 2013.

\bibitem{kdtree2006}
I.~Wald and V.~Havran.
\newblock {On building fast kd-trees for ray tracing, and on doing that in $O(N
  \log N)$}.
\newblock In {\em 2006 IEEE Symposium on Interactive Ray Tracing}, pages
  61--69. IEEE, 2006.

\bibitem{kernelBandwidth1989}
Bruce~J Worton.
\newblock Optimal smoothing parameters for multivariate fized and adaptive
  kernel methods.
\newblock {\em Journal of Statistical Computation and Simulation},
  32(1-2):45--57, 1989.

\bibitem{yang2003IEEE}
C.~Yang, R.~Duraiswami, N.~A. Gumerov, and L.~Davis.
\newblock {Improved fast Gauss transform and efficient kernel density
  estimation}.
\newblock In {\em Computer Vision, 2003. Proceedings. Ninth IEEE International
  Conference on}, pages 664--671. IEEE, 2003.

\bibitem{deepRitz}
B.~Yu and W.~E.
\newblock The deep ritz method: a deep learning-based numerical algorithm for
  solving variational problems.
\newblock {\em Communications in Mathematics and Statistics}, 6(1):1--12, 2018.

\bibitem{AMReX2019}
Weiqun Zhang, Ann Almgren, Vince Beckner, John Bell, Johannes Blaschke,
  Cy~Chan, Marcus Day, Brian Friesen, Kevin Gott, Daniel Graves, et~al.
\newblock Amrex: a framework for block-structured adaptive mesh refinement.
\newblock {\em The Journal of Open Source Software}, 4(37):1370, 2019.

\end{thebibliography}
\bibliographystyle{plain}
\end{document}